\documentclass[12pt,oneside]{amsart}
\usepackage[margin=1in]{geometry}
\usepackage{enumitem,verbatim}
\usepackage{microtype}
\allowdisplaybreaks

\usepackage{stmaryrd}
\usepackage[export]{adjustbox}
\usepackage{amsmath,amssymb,amsthm}
\usepackage{graphicx}
\usepackage[hyphens]{url}
\usepackage[pdfusetitle]{hyperref}
\usepackage[nameinlink]{cleveref}
\usepackage{color}
\hypersetup{colorlinks,linkcolor=blue,citecolor=cyan}
\usepackage{physics}
\usepackage{tikz-cd, tikz-3dplot}
\tikzset{anchorbase/.style={baseline={([yshift=-0.5ex]current bounding box.center)}}}
\tikzstyle directed=[postaction={decorate,decoration={markings,
    mark=at position #1 with {\arrow{>}}}}]
\tikzset{cross/.style={cross out, draw=black, minimum size=2*(#1-\pgflinewidth), inner sep=0pt, outer sep=0pt},
cross/.default={1pt}}
\usetikzlibrary{patterns, knots, calc, arrows.meta, decorations.markings,shapes.misc,math,hobby,intersections}

\usepackage{asymptote}
\usepackage{standalone}
\usepackage{subcaption}
\usepackage{wasysym}

\newtheorem{thm}{Theorem}[section]

\newtheorem{lem}[thm]{Lemma}
\newtheorem{prop}[thm]{Proposition}
\newtheorem{conj}[thm]{Conjecture}

\theoremstyle{definition}
\newtheorem{defn}[thm]{Definition}
\newtheorem{eg}[thm]{Example}

\newtheorem{rmk}[thm]{Remark}

\title{On the length conjecture for twist knots}

\author[S. Panitch]{Samuel Panitch}
\address{Department of Mathematics, Yale University, New Haven, CT 06511, USA}
\email{\href{mailto:sam.panitch@yale.edu}{sam.panitch@yale.edu}}

\author[M. Romo]{Mauricio Romo}
\address{Center for Mathematics and Interdisciplinary Sciences,
Fudan University, Shanghai, 200433, China}
\address{Shanghai Institute for Mathematics and Interdisciplinary Sciences (SIMIS),
Shanghai, 200433, China}
\email{\href{mailto:mromoj@simis.cn}{mromoj@simis.cn}}

\begin{document}

\begin{abstract}
Using the newly developed $3$d quantum trace map, we compile more evidence for the \textit{length conjecture}, a refinement of the all-order volume conjecture that incorporates insertions of additional links in the skein module of the knot complement. 
In particular, we prove that the length conjecture holds up to first order for an infinite family of twist knots.
Along the way, we form a conjecture that the $3$d quantum trace map behaves naturally with respect to Dehn filling.
\end{abstract}

\maketitle

\tableofcontents

\section{Introduction}

The volume conjecture, first formulated by Kashaev~\cite{Kas} and reformulated in terms of the colored Jones polynomial by Murakami and Murakami~\cite{MM}, predicts that the asymptotic growth rate of the colored Jones polynomial of a hyperbolic knot $\mathcal{K}$, evaluated at a root of unity, recovers the hyperbolic volume of its complement:
\begin{equation}
2\pi \lim_{n \to \infty} \left. \frac{\log \left| J_n\left(\mathcal{K};q\right) \right|}{n} \right|_{q = e^{2 \pi i/ n}} = \mathrm{Vol}(S^3 \setminus \mathcal{K}).
\end{equation}
This has been sharpened to an equality of complex numbers incorporating the Chern-Simons invariant \cite{MMOTY02}, and further to an all-orders statement, with the full asymptotic expansion of $\log J_n(\mathcal{K};q)|_{q=e^{2 \pi i/n}}$ as $n \to \infty$ conjectured to match the perturbative partition function of $\mathrm{SL}(2, \mathbb{C})$ Chern-Simons theory on $S^3 \setminus \mathcal{K}$, expanded around the geometric flat connection \cite{Guk05, GM08}.
See, e.g.~\cite{KT00, Oht16, Ohtsuki:2017osb,ohtsuki2018asymptotic, Mur08}, for the substantial progress since.

In order to tackle the side of the conjecture involving the perturbative expansion of the $\mathrm{SL}(2, \mathbb{C})$ Chern-Simons partition function on a hyperbolic knot complement, we need a concrete, computable model: the \emph{state integral}, a finite-dimensional integral built from an ideal triangulation of the manifold.
State integral models were developed independently by Hikami~\cite{HIKAMI_2001, Hikami:2006cv}, Andersen and Kashaev~\cite{EllegaardAndersen:2011vps}, and Dimofte and Garoufalidis~\cite{Dimofte:2011gm, Dimofte:2012qj}. 
Each assigns to a cusped hyperbolic $3$-manifold $M$ with ideal triangulation $\mathcal{T}$ a finite-dimensional integral $Z_\hbar(\mathcal{T})$ whose asymptotic expansion, around the saddle point corresponding to the complete hyperbolic structure on $M$, is a formal power series in $\hbar$ recovering the complex hyperbolic volume of $M$ at leading order and higher-loop quantum corrections at subleading orders — conjecturally matching the asymptotic expansion of the colored Jones polynomial under the identification $\hbar = 2 \pi i / n$.

The volume conjecture and its refinements can be enriched further by including additional
links inside $S^3 \setminus \mathcal{K}$.
In Chern-Simons theory, the observable attached to a link $K \subset S^3 \setminus \mathcal{K}$
is the holonomy along $K$ of an $\mathrm{SL}(2,\mathbb{C})$ flat connection on
$S^3 \setminus \mathcal{K}$, whose trace is a regular function on the character variety
\[
X_{\mathrm{SL}(2,\mathbb{C})} := \mathrm{Hom}\!\left(\pi_1(S^3 \setminus \mathcal{K}),
\mathrm{SL}(2,\mathbb{C})\right) \sslash \mathrm{SL}(2,\mathbb{C}).
\]
A natural question is: \textit{can we incorporate $K$ into the state integral
$Z_\hbar(\mathcal{T})$ and relate its asymptotics to the colored Jones polynomial of
$\mathcal{K} \cup K$, under an appropriate limit on the colors of the components?}

The quantum counterpart of the aforementioned trace function is an element of the skein module
$\mathrm{Sk}(S^3 \setminus \mathcal{K})$
\cite{przytycki2006skein,turaev1990conway,przytyckia2000skein}, which quantizes
$X_{\mathrm{SL}(2,\mathbb{C})}$.  The state integral, however, is built from an ideal
triangulation $\mathcal{T}$, which the skein module does not see; we therefore need a
dictionary between the two.  In dimension two, that dictionary is the quantum trace map of
Bonahon and Wong~\cite{BW}, embedding the Kauffman bracket skein algebra of a punctured
surface into a quantum torus.  Agarwal, Gang, Lee, and the second author~\cite{AGLR}
conjectured a $3$-dimensional analog, which was constructed independently by the first author and
Park~\cite{PP} and by Garoufalidis and Yu~\cite{GY26}.

With this $3$d quantum trace map in hand, we can state the \textit{length conjecture} informally.
Fix a hyperbolic knot $\mathcal{K} \subset S^3$ together with an auxiliary curve $K \subset S^3 \setminus \mathcal{K}$, which we think of as a ``light'' knot inside the complement of the ``heavy'' $\mathcal{K}$.
The colored Jones polynomial of the link $\mathcal{K} \cup K$, with $\mathcal{K}$ assigned color $n$ and $K$ assigned color $2$, defines a sequence indexed by $n$, which we normalize by dividing by $J_n(\mathcal{K}; q)$.
The length conjecture predicts that the asymptotic expansion of this ratio at $q = e^{2 \pi i / n}$ as $n \to \infty$ matches the perturbative expansion of the state integral on $S^3 \setminus \mathcal{K}$ with an insertion encoding the $3$d quantum trace of $K$.
The leading order of this expansion is a simple function of the complex hyperbolic length of the geodesic representative of $K$ in $S^3 \setminus \mathcal{K}$, which is the source of the conjecture's name.

The length conjecture was originally proposed by Agarwal, Gang, Lee, and the second author~\cite{AGLR}, who verified it for the figure-eight knot $\mathcal{K} = 4_1$ with $K = K_b$ (a specific link wrapping around the figure-eight) and its $2$-cabling, analytically to first order in $\hbar = 2 \pi i / n$ and numerically to second order.
Beyond this single example, the conjecture has remained untested.
Our main contribution is to verify it to first order for an infinite family of twist knots $\mathcal{K}_p$, with arbitrarily many parallel cables of an analogous curve $K_b \subset S^3 \setminus \mathcal{K}_p$.

A key obstacle is that existing ideal triangulations of twist knot complements grow in combinatorial complexity with $|p|$, making a uniform formulation of the state integral across the family difficult.
To bypass this, we use the classical fact that twist knots can be realized as Dehn fillings of the Whitehead link complement.
On the state integral side, we then use the Dehn-filled state integral developed in \cite{Gang:2017cwq}, applied to a fixed four-tetrahedron ideal triangulation $\mathcal{T}$ of the Whitehead link complement.
For the quantum trace, the curve $K_b$ lifts to a curve in the unfilled Whitehead link complement; abusing notation, we denote this lift by $K_b$ as well.
Its image under Dehn filling is isotopic to the original $K_b \subset S^3 \setminus \mathcal{K}_p$.
This lift allows us to compute insertions using the established $3$d quantum trace of \cite{PP, GY26}, without any modification.
Our main theorem is:

\begin{thm}[See Theorem~\ref{thm:main_theorem}]
\label{thm:main}
Take $p \in \mathbb{Z}$ with $p \geq 6$, let $\mathcal{K}_p$ denote the twist knot with $p$ full twists, and let $K_b \subset S^3 \setminus \mathcal{K}_p$ be the curve depicted in Figure~\ref{fig:twist_knot_curves}, regarded as a curve in the Whitehead link complement via its preimage under Dehn filling.
For every integer $l$ with $0 \leq l \leq 2|p|$, the asymptotic expansion of
\[
\frac{J_{n, 2}\!\left( \mathcal{K}_p \cup K_b^{l};\, q \right)}{J_n(\mathcal{K}_p;\, q)} \bigg|_{q = e^{2 \pi i / n}} \quad \text{as } n \to \infty
\]
agrees, to first order in $\hbar = 2 \pi i / n$, with the perturbative expansion of the Dehn-filled state integral on $\mathcal{T}$ with insertion $\Tr_{\mathcal{T}}\!\left( K_b^{l} \right)$ around the geometric saddle point.
\end{thm}

The restriction $p \geq 6$ is needed because our proof draws on the asymptotic expansion of the colored Jones polynomial of twist knots established in \cite{CZ1}, which is proved under this hypothesis.

We emphasize that Theorem~\ref{thm:main} is stated and proved entirely in terms of the established $3$d quantum trace of \cite{PP, GY26} on the unfilled Whitehead link complement, and the insertion $\Tr_{\mathcal{T}}(K_b^{l})$ is an element of the standard square-root quantum gluing module of the Whitehead link.
In particular, the theorem makes no reference to any quantum trace map intrinsic to the filled manifold $S^3 \setminus \mathcal{K}_p$.
Nonetheless, the structure of the proof suggests that such an intrinsic map should exist.
Classical results on skein modules \cite{Pr} guarantee that the skein module of a Dehn-filled $3$-manifold is a quotient of the skein module of the parent, and we conjecture that this topological quotient corresponds to an algebraic quotient on the quantum trace side:

\begin{conj}[See Conjecture~\ref{conj:dehn_filled_trace}]
\label{conj:intro}
Let $M$ be an oriented, ideally triangulated $3$-manifold with toroidal cusps, and let $M(\mathfrak{p}, \mathfrak{q})$ denote the Dehn filling of $M$ along a chosen cusp with slope $(\mathfrak{p}, \mathfrak{q})$.
Then there exists a nontrivial quotient $\mathrm{SQGM}_{\mathcal{T}}(M)_{\mathfrak{p}, \mathfrak{q}}$ of the square-root quantum gluing module of $M$ and a Dehn-filled quantum trace map
\[
\Tr_{\mathcal{T}}^{\mathfrak{p}, \mathfrak{q}} : \mathrm{Sk}(M(\mathfrak{p}, \mathfrak{q})) \to \mathrm{SQGM}_{\mathcal{T}}(M)_{\mathfrak{p}, \mathfrak{q}}
\]
that is compatible with the $3$d quantum trace map of \cite{PP, GY26} on $M$.
\end{conj}

Conjecture~\ref{conj:intro} is the subject of ongoing work by the first author.  

\subsection*{Organization of the paper}
This paper is organized as follows.

We begin in Section \ref{sec:2} by introducing all of the necessary ingredients for a proper statement of the length conjecture, including our conventions for the colored Jones polynomial, the construction of the $3$d quantum trace, and the definition of the state integral. 

The actual proof of the length conjecture for twist knots begins in Section \ref{sec:skein_modules_and_triangulations}, where we give a basis for the skein module of twist knots and establish the triangulation of the Whitehead link complement we will use to construct the state integral for twist knots.
Section \ref{sec:colored_jones_polynomials} computes the colored Jones polynomial of twist knots with insertions and finds an integral approximation for its asymptotic expansion.
Section \ref{sec:state_integrals} computes the state integral for twist knots, including the insertion of the quantum trace of the basis elements of the skein module. 
Finally, Section \ref{sec:proof} brings it all together by expanding both the colored Jones and state integrals to first order in $\hbar$ and making the comparison. 

We conclude with Section \ref{sec:filled_quantum_trace}, which gives some more detail on Conjecture~\ref{conj:intro}.

\subsection*{Acknowledgements}
We give many thanks to Andrew Neitzke for countless helpful discussions.
We also thank Dongmin Gang, Pavel Putrov, and Ka Ho Wong for useful conversations.
This collaboration began during the program ``Geometric, Algebraic, and Physical
Structures around the Moduli of Meromorphic Quadratic Differentials'' at the Simons
Center for Geometry and Physics, whose hospitality we gratefully acknowledge. In addition, MR thanks Chebyshev Laboratory at St. Petersburg State U. for hospitality while part
of this work was performed. The first author was supported in part by NSF grant DMS-2405256.

%%%%%%%%%%%%%%%%%%%%%%%%%%%%%%%%%%%%%%%%%%%%%%%%

\section{Statement of the length conjecture} \label{sec:2}

This section is dedicated to defining our conventions and introducing the necessary ingredients for the full statement of the length conjecture. 

\subsection{Colored Jones polynomials}

In this subsection, we remind the reader of the construction of the colored Jones polynomial from the Kauffman bracket, as well as establish some of our conventions for the colored Jones polynomials of links used in this paper. 

\begin{defn}
    The \textit{Kauffman bracket} of an unoriented, framed link diagram $L$, denoted $\langle L \rangle$, is the polynomial in $A$ defined by the following recursive relations: 
    \begin{enumerate}
        \item 
        \begin{gather*}
        \left \langle
        \vcenter{\hbox{
        \begin{tikzpicture}[scale=0.7]
        \draw[dotted] (0,0) circle (1);
        \draw[line width=3] ({sqrt(2)/2},{-sqrt(2)/2}) -- ({-sqrt(2)/2},{sqrt(2)/2});
        \draw[white, line width=6] ({-sqrt(2)/2},{-sqrt(2)/2}) -- ({sqrt(2)/2},{sqrt(2)/2});
        \draw[line width=3] ({-sqrt(2)/2},{-sqrt(2)/2}) -- ({sqrt(2)/2},{sqrt(2)/2});
        \end{tikzpicture}
        }}
        \right \rangle
        \;\;=\;\;
        A\;
        \left \langle
        \vcenter{\hbox{
        \begin{tikzpicture}[scale=0.7]
        \draw[dotted] (0,0) circle (1);
        \draw[line width=3] ({sqrt(2)/2},{sqrt(2)/2}) arc (135:225:1);
        \draw[line width=3] ({-sqrt(2)/2},{-sqrt(2)/2}) arc (-45:45:1);
        \end{tikzpicture}
        }}
        \right \rangle
        \;\;+\;\;
        A^{-1}\;
        \left \langle
        \vcenter{\hbox{
        \begin{tikzpicture}[scale=0.7]
        \draw[dotted] (0,0) circle (1);
        \draw[line width=3] ({sqrt(2)/2},{sqrt(2)/2}) arc (315:225:1);
        \draw[line width=3] ({-sqrt(2)/2},{-sqrt(2)/2}) arc (135:45:1);
        \end{tikzpicture}
        }}
        \right \rangle,
        \end{gather*}
    \item 
        \begin{gather*}
        \left \langle
        \vcenter{\hbox{
        \begin{tikzpicture}[scale=0.7]
        \draw[dotted] (0,0) circle (1);
        \draw[line width=3] (0.5,0) arc (0:370:0.5);
        \end{tikzpicture}
        }}
        \right \rangle
        \;\;=\;\;
        (-A^{2}-A^{-2})\;
        \left \langle
        \vcenter{\hbox{
        \begin{tikzpicture}[scale=0.7]
        \draw[dotted] (0,0) circle (1);
        \end{tikzpicture}
        }}
        \right \rangle,
        \end{gather*}
    \end{enumerate}
    in addition to being normalized so that the bracket of the empty link is $1$.
\end{defn}

The colored Jones polynomial can then be defined from the Kauffman bracket. 
Fix a framed, unoriented link diagram $L$ with $k$ components $L_1, \ldots, L_k$. 
\begin{defn}
Denote by $L^{(n_1,\ldots,n_k)}$ the diagram obtained from $L$ by replacing each component $L_i$ with $n_i$ of its parallels. 
Then, the \textit{cabling of $L$ by the polynomial} 
\[f(z_1,\ldots,z_k) = \sum_{i_1=0}^{n_1} \ldots \sum_{i_k=0}^{n_k} a_{i_1,\ldots,i_k}z_1^{i_1}\ldots z_k^{i_k},\]
denoted $L_{f(z_1,\ldots,z_k)}$, is the linear combination of link diagrams 
\[\sum_{i_1=0}^{n_1} \ldots \sum_{i_k=0}^{n_k} a_{i_1,\ldots,i_k}L^{(i_1,\ldots,i_k)}.\]
\end{defn}

\begin{defn}
The \textit{Chebyshev polynomials of the second kind}, denoted $S_n(z)$, are defined by the following recursive sequence: 
\begin{align*}
    S_0(z) &= 1, \\
    S_1(z) &= z, \\
    S_n(z) &= zS_{n-1}(z)-S_{n-2}(z).
\end{align*}
\end{defn}

\begin{defn}
The \textit{$(n_1,\ldots,n_k)$ colored Jones polynomial of $L$} is 
\[J_{n_1,\ldots,n_k}(L) := \langle \, L_{S_{n_1-1}(z_1)S_{n_2-1}(z_2)\cdots S_{n_k-1}(z_k)} \, \rangle \big|_{A = q^{1/4}}.\]
\end{defn}

As an example, the usual Jones polynomial of a knot $K$, in our conventions, is the $2$-colored Jones polynomial of $K$.\footnote{The usual Jones polynomial is an invariant of unframed, oriented knots, and relies on an additional factor of $(-A^3)^{-\textrm{w}(K)}$ to be well-defined with respect to Reidemeister moves of type I ($\textrm{w}(K)$ is the writhe of the knot $K$). Using framed links sidesteps the need for this normalization.} 

\subsection{State integrals}

State integral models for analytically continued $\mathrm{SL}(2,\mathbb{C})$ Chern-Simons theory \cite{Witten:1989ip,Witten:2010cx} on $r$-cusped, oriented hyperbolic $3$-manifolds $M$ are finite integrals that serve as models for the Feynman path integrals. There exist different approaches to these models \cite{HIKAMI_2001,Hikami:2006cv,EllegaardAndersen:2011vps,Dimofte:2011gm,Dimofte:2012qj}, but we will follow the approach of \cite{Dimofte:2011gm,Dimofte:2012qj}. First, we need to review some aspects of ideal triangulations of $M$ \cite{Thurston2022GeometryTopology,Neumann:1985yjj}. An ideal triangulation 
\begin{equation}
\mathcal{T}=\bigcup_{i=1}^{N}\Delta_{i}
\end{equation}
of an $r$-cusped $3$-manifold consists of a collection of $N$ ideal tetrahedra whose shape parameters are chosen as in Figure~\ref{fig:shape_parameter_labels}, where we denote $z_{i}=\exp Z_{i}$ and likewise for $z'_{i}$ and $z''_{i}$. For each ideal tetrahedron, the shape parameters satisfy the relations
\begin{equation}
Z_{i}+Z'_{i}+Z''_{i}=-i\pi,\quad z_i=1-z_i''^{-1},\quad z_i''=1-z_i'^{-1},
 \end{equation}
along with $N$ gluing equations (or gluing constraints). 
There is one such relation for each edge of $\mathcal{T}$, taking the form
 \begin{equation}
\mathcal{C}_{i}:=\sum_{j=1}^{N}\left(M_{i,j}Z_{j}+M'_{i,j}Z'_{j}+M''_{i,j}Z''_{j}\right)=-2\pi i,\qquad i=1,\ldots,N
 \end{equation}
where $M_{i,j},M'_{i,j},M''_{i,j}\in\{0,1,2\}$ for all $i,j$ and only $N-r$ of these equations are linearly independent. We will be interested in nondegenerate configurations, i.e. solutions satisfying
\begin{equation}
-\pi<\mathrm{Im}(Z_{i})< 0,\qquad \text{for all \ }i=1,\ldots,N,
 \end{equation}
since these configurations will include the geometric connection. This requires us to slightly modify our equations, in comparison to the usual conventions used in \cite{Dimofte:2011gm,Dimofte:2012qj}. We can associate to each cusp a meridian and longitude equation:
\begin{eqnarray}
\text{Meridian:}\quad &&\sum_{j=1}^{N}\left(\mathfrak{M}_{\alpha,j}Z_{j}+\mathfrak{M}'_{\alpha,j}Z'_{j}+\mathfrak{M}''_{\alpha,j}Z''_{j}\right)=0,\nonumber\\
\text{Longitude:}\quad &&\sum_{j=1}^{N}\left(\mathfrak{L}_{\alpha,j}Z_{j}+\mathfrak{L}'_{\alpha,j}Z'_{j}+\mathfrak{L}''_{\alpha,j}Z''_{j}\right)=0,\qquad\alpha=1,\dots,r
 \end{eqnarray}
where $\mathfrak{M}$ and $\mathfrak{L}$ are $\mathbb{Z}$-valued $r\times N$ matrices.
\begin{defn}[\cite{Dimofte:2012qj}]
We define a \textit{choice of quad}, for a triangulation $\mathcal{T}$, as a choice of a pair of opposite edges for each tetrahedron $\Delta_{i}\in\mathcal{T}$. In general we have $3^N$ possible choices of quads for $\mathcal{T}$. Once a choice of quad is made, we will denote as $y_{i}=\exp Y_{i}$ the chosen variables collectively, and $y''_{i}$ as the variable satisfying
$$
y''_{i}=1-y_{i}^{-1}.
$$
\end{defn}

Given $\mathcal{T}$ and a choice of quad, we can choose $N-r$ equations from $\mathcal{C}_{i}$ and $r$ meridian equations to get $N$ independent equations that we write in terms of $y_{i}$ and  $y''_{i}$ as\footnote{Throughout this work, we adopt the convention that, in \eqref{gluingclass}, the equations $i=1,\ldots, N-r$ correspond to a subset of the (exponentiated) gluing constraints and $i=N-r+1,\ldots, N$ correspond to the meridian equations.}
\begin{equation}\label{gluingclass}
\prod_{j=1}^{N}y_{j}^{A_{ij}}y''^{B_{ij}}_{j}=e^{-i\pi \nu_{i}},\qquad i=1,\ldots, N.
 \end{equation}
The equations \eqref{gluingclass} define the $N\times N$ $\mathbb{Z}$-valued matrices $A,B$ and the $\mathbb{Z}$-valued vector $\nu$ of length $N$. Likewise, we can write the longitude equations as
\begin{equation}\label{longclass}
\prod_{j=1}^{N}y_{j}^{2C_{\alpha, j}}y''^{2D_{\alpha,j}}_{j}=e^{-2i\pi \nu'_{\alpha}},\qquad \alpha=1,\ldots,r.
\end{equation}
Then, the equations \eqref{longclass} define the $r\times N$ $\frac{1}{2}\mathbb{Z}$-valued matrices of components $C_{\alpha,j},D_{\alpha,j}$ and the $\frac{1}{2}\mathbb{Z}$-valued vector $\nu'$ of length $r$. These matrices can be completed to $N\times N$ $\frac{1}{2}\mathbb{Z}$-valued matrices $C,D$ satisfying the relations:
\begin{eqnarray}\label{condsympl1}
D^{t}A-B^{t}C=\mathbf{1},\qquad B^{t}D-D^{t}B=0.
\end{eqnarray}
When $B$ is invertible, conditions \eqref{condsympl1} are equivalent to imposing\footnote{In our conventions $M\in \mathrm{Sp}(2n,\mathbb{R})$ iff $M^{t}\Omega M=\Omega$ with $\Omega=\left[\begin{array}{cc}
0 & \mathbf{1}_{n}
\\
-\mathbf{1}_{n} & 0 
\end{array}\right]$.}
\begin{eqnarray}\label{condsympl}
\left[\begin{array}{cc}
A & B
\\
C & D
\end{array}\right]\in \mathrm{Sp}(2N,\mathbb{R}).
\end{eqnarray}
We also will need the definition of a combinatorial flattening \cite{Dimofte:2012qj}:
\begin{defn} 
A \textit{combinatorial flattening}, compatible with the longitude, consists of a pair of vectors $f,f''\in\mathbb{Z}^{N}$ satisfying
\begin{eqnarray}
A\cdot f+B\cdot f''=\nu,\qquad \sum_{i=1}^{N}(C_{\alpha,i}f_{i}+D_{\alpha,i}f''_{i})=\nu'_\alpha,\quad \alpha=1,\ldots,r.
\end{eqnarray}
Given an ideal triangulation $\mathcal{T}$, a combinatorial flattening always exists \cite{neumann1992combinatorics}.
\end{defn}
We are ready to define the state integral model for $\mathcal{T}$ associated to $M$:
\begin{defn} \label{def:state_integral} Given an ideal triangulation $\mathcal{T}$, a choice of quad, the matrices $(A,B,C,D,\nu)$ defined in \eqref{gluingclass}, \eqref{longclass} and \eqref{condsympl}, and a choice of combinatorial flattening $(f,f'')$, we define the state integral model associated to $(\mathcal{T},A,B,C,D,\nu,f,f'')$, whenever\footnote{It is shown in \cite{Dimofte:2012qj} that there always exists a choice of quad such that this is possible.} $\mathrm{det}B\neq 0$, as
\begin{equation}\label{PartFn}
Z_{\hbar}(\mathcal{T};X)=\frac{2}{\sqrt{\mathrm{det}B}}\int \prod_{i=1}^{N}\frac{dY_{i}}{\sqrt{2\pi \hbar}}\exp\left(\frac{1}{\hbar}Q(Y,X)\right) \prod_{i=1}^{N}\psi_{\hbar}(Y_{i}),
\end{equation}
where $X\in \mathbb{C}^{N}$ is a vector of parameters of the form
\begin{eqnarray}
X=(\underbrace{0,\ldots,0}_{N-r},X_{1},\ldots,X_{r})
\end{eqnarray}
and
\begin{eqnarray}\label{Qstateint}
Q(Y,X)&=&\frac{1}{2}YB^{-1}AY+2XDB^{-1}X+(\hbar-2\pi i)fB^{-1}X+\frac{1}{2}\left(\frac{\hbar}{2}-i\pi\right)^{2}fB^{-1}\nu \nonumber\\
&-&YB^{-1}\left( \left(\frac{\hbar}{2}-i\pi\right)\nu+2X\right), \qquad Y:=(Y_{1},\ldots,Y_{N}).
\end{eqnarray}
Because of the choice of quad, the combinatorial flattening and completion \eqref{condsympl} of the matrices $C,D$ are not unique: \eqref{PartFn} is defined up to a factor
\begin{eqnarray}
\exp\left(a\frac{\pi^{2}}{6\hbar}+b\frac{i\pi}{4}+c\frac{\hbar}{24}\right),\qquad a,b,c\in\mathbb{Z}.
\end{eqnarray}
\end{defn}
The function $\psi_{\hbar}$ in \eqref{PartFn} is called Faddeev's quantum dilogarithm and we collect several properties relevant for our computations in Appendix \ref{App:QDL}.
\begin{rmk} \begin{enumerate}
    \item The expression \eqref{PartFn} is defined with the integration contour unspecified. This is because \eqref{PartFn} should be understood as defined by its asymptotic expansion around any choice of critical point of the integrand, via the techniques detailed in Appendix \ref{app:Feynmann}.
    \item The conventions for \eqref{PartFn} are slightly different from the ones in \cite{Dimofte:2012qj}, because we are interested in computing asymptotic expansions around the hyperbolic or geometric connection, which is equivalent to the critical point satisfying
    \begin{eqnarray}
-\pi<\mathrm{Im}(Y_{i})<0,\qquad \text{for all \ }i=1,\ldots, N.
\end{eqnarray}
This critical point satisfies (see \cite{BGP}):
 \begin{eqnarray}
\mathrm{Im}(W_{0})= \mathrm{Vol}(M).
\end{eqnarray}
Therefore, we need to use the conjugate of the state integral model conventions used in \cite{Dimofte:2012qj}, which translates into the slight change of \eqref{Qstateint} used here. We will always work with
\begin{eqnarray}
\hbar\in i\mathbb{R}_{<0}.
\end{eqnarray}
This is the \emph{conjugate} of the conventions used in \cite{Gang:2017cwq}, for instance. 
\end{enumerate}
\end{rmk}

\subsection{The Dehn-filled state integral}\label{sec:DehnFillsi}
Triangulations of twist knots have been built in the literature (see, for instance, \cite{BGP}), but these triangulations increase in complexity with the knot.
To bypass this complication, we will study twist knots via Dehn fillings on the Whitehead link. 
This requires understanding how state integrals are modified under filling.

Let's first briefly establish our conventions for Dehn filling.
Let $M$ be an oriented $3$-manifold whose boundary consists of a union of $r$ toroidal cusps, equipped with an ideal triangulation $\mathcal{T}$.
Choose one specific boundary torus, which we will denote by $\partial_i M$.
Fix a choice of a meridian $\mu$ and a longitude $\lambda$ on $\partial_i M$, which together form a basis for the first homology.
Any isotopy class of a simple closed curve on this boundary component can then be uniquely parameterized by a pair of co-prime integers $(\mathfrak{p},\mathfrak{q})$, corresponding to the curve $\mathfrak{p}\mu + \mathfrak{q}\lambda$.
The \textit{$(\mathfrak{p},\mathfrak{q})$ Dehn-filled manifold, denoted $M(\mathfrak{p},\mathfrak{q})$}, is obtained by gluing a solid torus to $M$ such that the meridian curve of the newly attached solid torus is glued to the curve $\mathfrak{p}\mu + \mathfrak{q}\lambda$.

In the following we will focus on the case $r=2$, which is the most relevant for the present work. The state integral \eqref{PartFn} can be modified to implement Dehn filling along any of its cusps. This modification is inspired by the physics of M-theory \cite{Bae:2016jpi} and it was tested in \cite{Gang:2017cwq}, against the volume conjecture for certain closed $3$-manifolds. In order to implement the Dehn surgery along the boundary cycle $\mathfrak{p}\mu_{1}+\mathfrak{q}\lambda_{1}$ (with holonomies $2 X_{1}$ and $L_{1}$ around $\mu_1$ and $\lambda_1$, respectively)\footnote{The factor of $2$ in $2X_{1}$ is purely conventional, in order to follow the notation of \cite{Gang:2017cwq}.} on the state integral side, we just insert the following function in  
\eqref{PartFn}:
\begin{eqnarray}\label{kerpq}
\mathrm{K}_{\mathfrak{p},\mathfrak{q}}(X_1)&:=&\exp\left(\frac{\mathfrak{s}}{\mathfrak{q}}\left(\frac{\pi^{2}}{\hbar}-\frac{\hbar}{4}\right)+\frac{\mathfrak{p}X_1^2}{\mathfrak{q}\hbar}\right)\nonumber\\
&\times&\left(e^{-\frac{2\pi iX_1}{\hbar \mathfrak{q}}}\sinh\left(\frac{X_1+i\pi \mathfrak{s} }{\mathfrak{q}}\right)-e^{\frac{2\pi iX_1}{\hbar \mathfrak{q}}}\sinh\left(\frac{X_1-i\pi \mathfrak{s} }{\mathfrak{q}}\right)\right),
\end{eqnarray}
where 
\begin{eqnarray}\label{conddhen}
\left[\begin{array}{cc}
\mathfrak{h} & \mathfrak{s}
\\
\mathfrak{p} & \mathfrak{q}
\end{array}\right]\in \mathrm{PSL}(2,\mathbb{Z})
\end{eqnarray}
and the parameter $\mathfrak{h}$ does not appear explicitly in \eqref{kerpq} when we assume $\mathfrak{q}\neq 0$. So, altogether, we define
\begin{eqnarray}\label{dehnfillpf}
Z_{\hbar}^{(\mathfrak{p},\mathfrak{q})}(\mathcal{T};X_{2}):=\int \frac{dX_{1}}{\sqrt{2\pi \hbar \mathfrak{q}}}\mathrm{K}_{\mathfrak{p},\mathfrak{q}}(X_1)Z_{\hbar}(\mathcal{T};X_{1},X_{2}),
\end{eqnarray}
with $Z_{\hbar}^{(\mathfrak{p},\mathfrak{q})}(\mathcal{T};X_{2})$ identified with the partition function after Dehn surgery. The formula \eqref{dehnfillpf} must be understood in the same sense as \eqref{PartFn}, i.e. the integration $\int dX_{1}$ in \eqref{dehnfillpf} is left undefined and we just understand this formula as an asymptotic expansion around the different critical points of the integrand.
\begin{rmk}
The function \eqref{kerpq} differs slightly from the one used in \cite{Bae:2016jpi,Gang:2017cwq}, for the same reasons as \eqref{PartFn}: we need to take the conjugate in order to expand around the critical point satisfying $-\pi<\mathrm{Im}(Y)<0$ and $-\pi<\mathrm{Im}(X)<0$. In addition, we also need to consider $\hbar\in i\mathbb{R}_{<0}$.
\end{rmk}

\subsection{$3$d quantum trace map} \label{sec:quantum_trace}

In this subsection, we give a brief overview of the $3$d quantum trace map, following \cite{PP, PP2}.

\begin{defn}
    Let $M$ be an oriented $3$-manifold with boundary. 
    
    A \textit{boundary marking} for $M$ is a smoothly embedded, oriented graph $\Gamma$ in $\partial M$ such that every vertex of $\Gamma$ is either a source or a sink. 
    The pair $(M,\Gamma)$ is called a \textit{boundary-marked $3$-manifold}.

    A \textit{ribbon tangle in $(M,\Gamma)$} is a ribbon tangle whose boundary points all lie on $\Gamma$ away from the vertices, and whose framing at these boundary points is parallel to the tangent vector of the boundary marking. 
    We call such a tangle \textit{stated} if each of its boundary points is also adorned with an extra sign $\mu \in \{\pm 1\}.$
\end{defn}

Fix $R:= \mathbb{Z}[A^{\pm \frac 1 2},(-A^2)^{\pm \frac 1 2}].$

\begin{defn}
    Given a boundary marked $3$-manifold $(M,\Gamma)$, the \textit{stated skein module} $\mathrm{Sk}(M,\Gamma)$ is the $R$-module generated by the isotopy classes of stated ribbon tangles in $M$, modulo the following skein relations:\footnote{The orange arrows depict sections of the boundary marking.}
    \begin{gather}
    \vcenter{\hbox{
    \begin{tikzpicture}[scale=0.7]
    \draw[dotted] (0,0) circle (1);
    \draw[line width=3] ({sqrt(2)/2},{-sqrt(2)/2}) -- ({-sqrt(2)/2},{sqrt(2)/2});
    \draw[white, line width=6] ({-sqrt(2)/2},{-sqrt(2)/2}) -- ({sqrt(2)/2},{sqrt(2)/2});
    \draw[line width=3] ({-sqrt(2)/2},{-sqrt(2)/2}) -- ({sqrt(2)/2},{sqrt(2)/2});
    \end{tikzpicture}
    }}
    \;\;=\;\;
    A\;
    \vcenter{\hbox{
    \begin{tikzpicture}[scale=0.7]
    \draw[dotted] (0,0) circle (1);
    \draw[line width=3] ({sqrt(2)/2},{sqrt(2)/2}) arc (135:225:1);
    \draw[line width=3] ({-sqrt(2)/2},{-sqrt(2)/2}) arc (-45:45:1);
    \end{tikzpicture}
    }}
    \;\;+\;\;
    A^{-1}\;
    \vcenter{\hbox{
    \begin{tikzpicture}[scale=0.7]
    \draw[dotted] (0,0) circle (1);
    \draw[line width=3] ({sqrt(2)/2},{sqrt(2)/2}) arc (315:225:1);
    \draw[line width=3] ({-sqrt(2)/2},{-sqrt(2)/2}) arc (135:45:1);
    \end{tikzpicture}
    }}
    \;, \label{item:skeinRelation1}
    \\
    \vcenter{\hbox{
    \begin{tikzpicture}[scale=0.7]
    \draw[dotted] (0,0) circle (1);
    \draw[line width=3] (0.5,0) arc (0:370:0.5);
    \end{tikzpicture}
    }}
    \;\;=\;\;
    (-A^{2}-A^{-2})\;
    \vcenter{\hbox{
    \begin{tikzpicture}[scale=0.7]
    \draw[dotted] (0,0) circle (1);
    \end{tikzpicture}
    }}
    \;, \label{item:skeinRelation2}
    \\
    \vcenter{\hbox{
    \tdplotsetmaincoords{20}{30}
    \begin{tikzpicture}[tdplot_main_coords, rotate=30]
    \begin{scope}[scale = 1.0, tdplot_main_coords]
    \fill[color=lightgray, opacity=0.3] (0, -1, -1) -- (0, -1, 1) -- (0, 1, 1) -- (0, 1, -1) -- cycle;
    \draw[ultra thick, orange, ->] (0, -1, 0) -- (0, 0, 0);
    \draw[ultra thick, orange] (0, -1, 0) -- (0, 1, 0);
    \draw[dotted] (0, -1, -1) -- (0, -1, 1) -- (0, 1, 1) -- (0, 1, -1) -- cycle;
    \draw[dotted] (-1.5, -1, -1) -- (-1.5, -1, 1) -- (-1.5, 1, 1) -- (-1.5, 1, -1) -- cycle;
    \draw[dotted] (0, -1, -1) -- (-1.5, -1, -1);
    \draw[dotted] (0, -1, 1) -- (-1.5, -1, 1);
    \draw[dotted] (0, 1, -1) -- (-1.5, 1, -1);
    \draw[dotted] (0, 1, 1) -- (-1.5, 1, 1);
    \draw[line width=5] (0, 0.5, 0) .. controls (-0.8, 0.5, 0) and (-0.8, -0.5, 0) .. (0, -0.5, 0);
    \node[anchor = west] at (0, 0.5, 0) {$\mu$};
    \node[anchor = west] at (0, -0.5, 0) {$\nu$};
    \end{scope}
    \end{tikzpicture}
    }}
    \;\;=\;\;
    \delta_{\mu, -\nu}\,(-A^2)^{\frac{\mu}{2}}
    \;,
    \quad\quad \mu, \nu \in \{\pm 1\}\;,
    \\
    \vcenter{\hbox{
    \tdplotsetmaincoords{20}{30}
    \begin{tikzpicture}[tdplot_main_coords, rotate=30]
    \begin{scope}[scale = 1.0, tdplot_main_coords]
    \fill[color=lightgray, opacity=0.3] (0, -1, -1) -- (0, -1, 1) -- (0, 1, 1) -- (0, 1, -1) -- cycle;
    \draw[ultra thick, orange, ->] (0, -1, 0) -- (0, 0, 0);
    \draw[ultra thick, orange] (0, -1, 0) -- (0, 1, 0);
    \draw[dotted] (0, -1, -1) -- (0, -1, 1) -- (0, 1, 1) -- (0, 1, -1) -- cycle;
    \draw[dotted] (-1.5, -1, -1) -- (-1.5, -1, 1) -- (-1.5, 1, 1) -- (-1.5, 1, -1) -- cycle;
    \draw[dotted] (0, -1, -1) -- (-1.5, -1, -1);
    \draw[dotted] (0, -1, 1) -- (-1.5, -1, 1);
    \draw[dotted] (0, 1, -1) -- (-1.5, 1, -1);
    \draw[dotted] (0, 1, 1) -- (-1.5, 1, 1);
    \draw[line width=5] (-1.5, 0.5, 0) .. controls (-0.5, 0.5, 0) and (-0.5, -0.5, 0) .. (-1.5, -0.5, 0);
    \end{scope}
    \end{tikzpicture}
    }}
    \;\;=\;\;
    \sum_{\mu \in \{\pm 1\}}\,(-A^2)^{\frac{\mu}{2}}
    \vcenter{\hbox{
    \tdplotsetmaincoords{20}{30}
    \begin{tikzpicture}[tdplot_main_coords, rotate=30]
    \begin{scope}[scale = 1.0, tdplot_main_coords]
    \fill[color=lightgray, opacity=0.3] (0, -1, -1) -- (0, -1, 1) -- (0, 1, 1) -- (0, 1, -1) -- cycle;
    \draw[ultra thick, orange, ->] (0, -1, 0) -- (0, 0, 0);
    \draw[ultra thick, orange] (0, -1, 0) -- (0, 1, 0);
    \draw[dotted] (0, -1, -1) -- (0, -1, 1) -- (0, 1, 1) -- (0, 1, -1) -- cycle;
    \draw[dotted] (-1.5, -1, -1) -- (-1.5, -1, 1) -- (-1.5, 1, 1) -- (-1.5, 1, -1) -- cycle;
    \draw[dotted] (0, -1, -1) -- (-1.5, -1, -1);
    \draw[dotted] (0, -1, 1) -- (-1.5, -1, 1);
    \draw[dotted] (0, 1, -1) -- (-1.5, 1, -1);
    \draw[dotted] (0, 1, 1) -- (-1.5, 1, 1);
    \draw[line width=5] (0, 0.5, 0) -- (-1.5, 0.5, 0);
    \draw[line width=5] (0, -0.5, 0) -- (-1.5, -0.5, 0);
    \node[anchor = west] at (0, 0.5, 0) {$\mu$};
    \node[anchor = west] at (0, -0.5, 0) {$-\mu$};
    \end{scope}
    \end{tikzpicture}
    }}
    \\
    \vcenter{\hbox{
    \begin{tikzpicture}[scale=0.7]
    \fill[top color=black, bottom color=white] (-0.085, 1) to[out=-90, in=100] (0, 0) to[out=80, in=-90] (0.085, 1)--cycle;
    \fill[bottom color=black, top color=white] (-0.085, -1) to[out=90, in=-100] (0, 0) to[out=-80, in=90] (0.085, -1)--cycle;
    \draw[line width=1] (-0.075, -1) to[out=90, in=-90] (0.075, 1);
    \draw[dotted] (0,0) circle (1);
    \end{tikzpicture}
    }}
    \;\;=\;\;
    (-A^3)^{\frac{1}{2}}\;
    \vcenter{\hbox{
    \begin{tikzpicture}[scale=0.7]
    \draw[dotted] (0,0) circle (1);
    \draw[line width=3] (0, -1) -- (0, 1);
    \end{tikzpicture}
    }}
    \;. \label{item:halfTwistRel}
    \end{gather}
\end{defn}

Stated skein modules are useful because they admit a splitting map, allowing them to be chopped into simpler pieces. 
For the purposes of the $3$d quantum trace, the relevant splitting map comes from cutting an ideally triangulated $3$-manifold into \textit{face suspensions}, denoted $Sf$, and shown in Figure~\ref{fig:face_suspension}. 
Face suspensions have a natural boundary marking, shown in Figure~\ref{fig:face_suspension_markings}. 
We call each face of a face suspension an \textit{edge cone}, as they naturally correspond to an edge of a tetrahedron. 
Stated skein modules have natural module structures over skein algebras associated to vertices of their boundary marking. 
Face suspensions are naturally left modules over the skein algebra of the triangle, one copy for each of the two sinks, and right modules over the skein algebra of the biangle, one copy for each of the three sources.

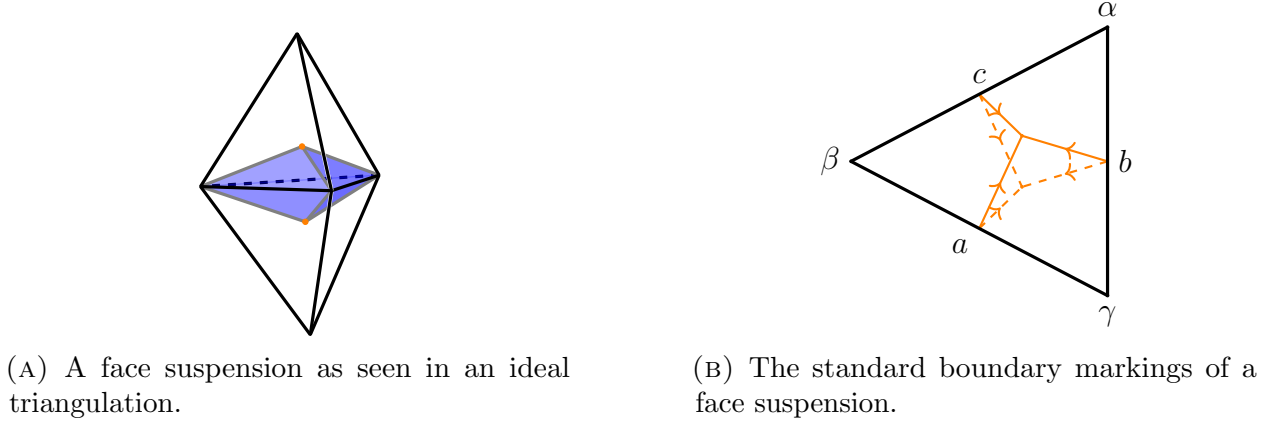
\begin{figure}[htbp]
    \centering
    \begin{subfigure}[b]{0.45\textwidth}
        \centering
        \tdplotsetmaincoords{90}{90}
        \begin{tikzpicture}[tdplot_main_coords]
        %rotated coords for a (imo) better pov
        \tdplotsetrotatedcoords{-30}{5}{45}
        \begin{scope}[scale = 0.5, tdplot_rotated_coords]
            %fix coordinates of the first tetrahedra
            \coordinate (o) at (0, 0, 0);
            \coordinate (a) at (0, 0, 3);
            \coordinate (b) at ({2*sqrt(2)}, 0, -1);    
            \coordinate (c) at ({-sqrt(2)}, {sqrt(6)}, -1);    
            \coordinate (d) at ({-sqrt(2)}, {-sqrt(6)}, -1);
            %extra coordinates for the second tetrahedra
            \coordinate (e) at (0,0,-5); 
            \coordinate (o2) at (0,0,-2);
            %draw dotted edge
            \draw[very thick, dashed] (c) -- (d);
            %draw the edge cones in tet 1
            \filldraw[blue, opacity=0.3, line width=0] (d)--(c)--(o)--cycle;
            \filldraw[blue, opacity=0.1, line width=0] (d)--(b)--(o)--cycle;
            \filldraw[blue, opacity=0.3, line width=0] (c)--(b)--(o)--cycle;
            %draw edge cones in tet 2
            \filldraw[blue, opacity=0.3, line width=0] (d)--(c)--(o2)--cycle;
            \filldraw[blue, opacity=0.2, line width=0] (d)--(b)--(o2)--cycle;
            \filldraw[blue, opacity=0.4, line width=0] (c)--(b)--(o2)--cycle;
            %draw vertex cones in tet 1
            \draw[very thick, gray] (o) -- (b);
            \draw[very thick, gray] (o) -- (c);
            \draw[very thick, gray] (o) -- (d);
            %draw vertex cones in tet 2
            \draw[very thick, gray] (o2) -- (b);
            \draw[very thick, gray] (o2) -- (c);
            \draw[very thick, gray] (o2) -- (d);
            %draw the barycenters
            \filldraw[orange] (o) circle (2pt);
            \filldraw[orange] (o2) circle (2pt);
            %fill in the rest of the edges of the tetrehedra
            \draw[ultra thick, white] (b) -- (e);
            \draw[very thick] (b) -- (e);
            \draw[very thick] (c) -- (e);
            \draw[very thick] (d) -- (e);
            \draw[very thick] (b) -- (c);
            \draw[white, ultra thick] (a) -- (b);
            \draw[very thick] (a) -- (b);
            \draw[very thick] (a) -- (c);
            \draw[very thick] (a) -- (d);
            \draw[very thick] (b) -- (d);
            %label vertices of central face
            \filldraw[] (b) circle (.5pt);
            \filldraw[] (d) circle (.5pt);
            \filldraw[] (c) circle (.5pt);
            \filldraw[] (a) circle (.5pt);
            \filldraw[] (e) circle (.5pt);
        \end{scope}
        \end{tikzpicture}
        \caption{A face suspension as seen in an ideal triangulation.}
        \label{fig:face_suspension}
    \end{subfigure}
    \hfill
    \begin{subfigure}[b]{0.45\textwidth}
        \centering
        \tdplotsetmaincoords{25}{60}
        \begin{tikzpicture}[tdplot_main_coords]
        \begin{scope}[scale = 0.8, tdplot_main_coords]
            \coordinate (o) at (0, 0, 1);
            \coordinate (p) at (0, 0, -1);
            \coordinate (a) at (0, 0, 4);
            \coordinate (a1) at (0, 0, 0);
            
            \coordinate (b) at ({2*sqrt(2)}, 0, 0);
            \coordinate (b1) at ({-2*sqrt(2)/3}, 0, 4/3);
            
            \coordinate (c) at ({-sqrt(2)}, {sqrt(6)}, 0);
            \coordinate (c1) at ({sqrt(2)/3}, {-sqrt(6)/3}, 4/3);
            
            \coordinate (d) at ({-sqrt(2)}, {-sqrt(6)}, 0);
            \coordinate (d1) at ({sqrt(2)/3}, {sqrt(6)/3}, 4/3);
            
            \coordinate (ab) at ({sqrt(2)}, 0, 2);
            \coordinate (ac) at ({-sqrt(2)/2}, {sqrt(6)/2}, 2);
            \coordinate (ad) at ({-sqrt(2)/2}, {-sqrt(6)/2}, 2);
            \coordinate (bc) at ({sqrt(2)/2}, {sqrt(6)/2}, 0);
            \coordinate (bd) at ({sqrt(2)/2}, {-sqrt(6)/2}, 0);
            \coordinate (cd) at ({-sqrt(2)}, 0, 0);

            \draw[white, ultra thick] (o) -- (a1);
            
            \begin{scope}[thick,decoration={
            markings,
            mark=at position 0.5 with {\arrow{>}}}
            ] 
            \draw[postaction={decorate}, orange] (bd) -- (o);
            \draw[postaction={decorate}, orange] (bc) -- (o);
            \draw[postaction={decorate}, orange] (cd) -- (o);
            \draw[postaction={decorate}, dashed, orange] (bd) -- (p);
            \draw[postaction={decorate}, dashed, orange] (bc) -- (p);
            \draw[postaction={decorate}, dashed, orange] (cd) -- (p);
            \end{scope}
        
            \draw[very thick] (b) -- (c);
            \draw[very thick] (b) -- (d);
            \draw[very thick] (c) -- (d);
            
            \filldraw[orange] (o) circle (0.05em);
            \filldraw[orange] (p) circle (0.05em);
            \filldraw (b) circle (0.05em);
            \filldraw (c) circle (0.05em);
            \filldraw (d) circle (0.05em);
        
            \node[anchor=north east] at (bd) {$a$};
            \node[anchor=west] at (bc) {$b$};
            \node[anchor=south] at (cd) {$c$};
            \node[anchor=south] at (c) {$\alpha$};
            \node[anchor=east] at (d) {$\beta$};
            \node[anchor=north] at (b) {$\gamma$};
            
            % \node at (b) {$(b)$};
            % \node at (c) {$(c)$};
            % \node at (d) {$(d)$};
            % \node at (o) {$(o)$};
            % \node at (p) {$(p)$};
        \end{scope}
        \end{tikzpicture}
        \caption{The standard boundary markings of a face suspension.}
        \label{fig:face_suspension_markings}
    \end{subfigure}
    \caption{Face suspensions and their boundary markings.}
\end{figure}

It is necessary to take an additional quotient of the stated skein module of the face suspension. 
\begin{defn}
The reduced stated skein module of the face suspension is defined as 
\[\overline{\mathrm{Sk}}(Sf) := I^{\mathrm{bad},+} \setminus \mathrm{Sk}(Sf) \,/ \, I^{\mathrm{bad},-},\]
where $I^{\mathrm{bad},+}$ denotes the right ideal generated by the \textit{bad arcs} near the sinks, shown in Figure~\ref{fig:triangle_bad_arcs}, and $I^{\mathrm{bad},-}$ denotes the left ideal generated by the \textit{bad arcs} near the sources, shown in Figure~\ref{fig:biangle_bad_arcs}.
\end{defn}

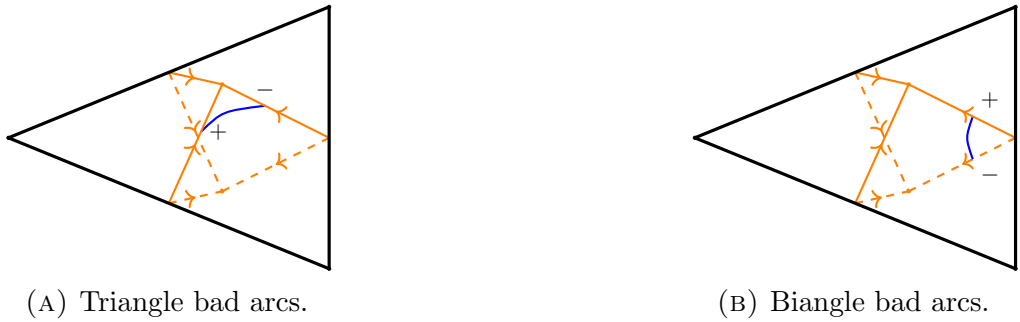
\begin{figure}[htbp]
    \centering
    \begin{subfigure}[b]{0.45\textwidth}
        \centering
        \tdplotsetmaincoords{45}{60}

        \begin{tikzpicture}[tdplot_main_coords]
        
        \begin{scope}[tdplot_main_coords]
        \coordinate (N) at (0, 0, 1);
        \coordinate (S) at (0, 0, -1);  
        \coordinate (b) at ({2*sqrt(2)*cos(0)}, {2*sqrt(2)*sin(0)}, 0);
        \coordinate (c) at ({2*sqrt(2)*cos(120)}, {2*sqrt(2)*sin(120)}, 0);
        \coordinate (d) at ({2*sqrt(2)*cos(240)}, {2*sqrt(2)*sin(240)}, 0);   
        \coordinate (bc) at ($1/2*(b) + 1/2*(c)$);
        \coordinate (bd) at ($1/2*(b) + 1/2*(d)$);
        \coordinate (cd) at ($1/2*(c) + 1/2*(d)$);
        \coordinate (kb1) at ($(bd)!.6!(N)$);
        \coordinate (kb2) at ($(bc)!.6!(N)$);
        \coordinate (kb3) at ($(bc)!.4!(N)$);
        \coordinate (kb4) at ($(bc)!.4!(S)$);   
        \begin{scope}[thick,decoration={
        markings,
        mark=at position 0.5 with {\arrow{>}}}
        ] 
        %kb and the marking it is "above"
        \draw[blue] (kb1) .. controls (0,0,.5) .. (kb2);
        \draw[postaction={decorate}, dashed, orange] (bc) -- (S);
        %markings
        \draw[postaction={decorate}, orange] (bd) -- (N);
        \draw[postaction={decorate}, orange] (bc) -- (N);
        \draw[postaction={decorate}, orange] (cd) -- (N);
        \draw[postaction={decorate}, dashed, orange] (bd) -- (S);
        \draw[postaction={decorate}, dashed, orange] (cd) -- (S);
        \end{scope}
        %edges
        \begin{scope}[decoration={
        markings,
        mark=between positions 0.6 and .66 step 0.06 with \arrow{stealth}},very thick]
        \draw[] (b)--(c);
        \draw[] (b)--(d);
        \end{scope}
        \begin{scope}[decoration={
        markings,
        mark=at position 0.6 with \arrow{stealth}},very thick]
        \draw[] (d)--(c);
        \end{scope}    
        \filldraw[orange] (N) circle (0.05em);
        \filldraw[orange] (S) circle (0.05em);
        \filldraw (b) circle (0.05em);
        \filldraw (c) circle (0.05em);
        \filldraw (d) circle (0.05em);
        %\node[anchor=north east] at (bd) {$a$};
        %\node[anchor=west] at (bc) {$b$};
        %\node[anchor=south] at (cd) {$c$};
        %\node[anchor=south] at (c) {$\alpha$};
        %\node[anchor=east] at (d) {$\beta$};
        %\node[anchor=north] at (b) {$\gamma$};
        \node[right,scale=.75] at (kb1) {$+$};
        \node[above,scale=.75] at (kb2) {$-$};  
        \end{scope}
        \end{tikzpicture}
        \caption{Triangle bad arcs.}
        \label{fig:triangle_bad_arcs}
    \end{subfigure}
    \hfill
    \begin{subfigure}[b]{0.45\textwidth}
        \centering
        \tdplotsetmaincoords{45}{60}
        \begin{tikzpicture}[tdplot_main_coords]
        
        \begin{scope}[tdplot_main_coords]
        \coordinate (N) at (0, 0, 1);
        \coordinate (S) at (0, 0, -1);  
        \coordinate (b) at ({2*sqrt(2)*cos(0)}, {2*sqrt(2)*sin(0)}, 0);
        \coordinate (c) at ({2*sqrt(2)*cos(120)}, {2*sqrt(2)*sin(120)}, 0);
        \coordinate (d) at ({2*sqrt(2)*cos(240)}, {2*sqrt(2)*sin(240)}, 0);   
        \coordinate (bc) at ($1/2*(b) + 1/2*(c)$);
        \coordinate (bd) at ($1/2*(b) + 1/2*(d)$);
        \coordinate (cd) at ($1/2*(c) + 1/2*(d)$);
        \coordinate (kb1) at ($(bd)!.6!(N)$);
        \coordinate (kb2) at ($(bc)!.6!(N)$);
        \coordinate (kb3) at ($(bc)!.4!(N)$);
        \coordinate (kb4) at ($(bc)!.4!(S)$);   
        \begin{scope}[thick,decoration={
        markings,
        mark=at position 0.5 with {\arrow{>}}}
        ] 
        %kb and the marking it is "above"
        \draw[postaction={decorate}, dashed, orange] (bc) -- (S);
        \draw[blue] (kb3) .. controls ({3/4*cos(60)},{3/4*sin(60)},0) .. (kb4);
        %markings
        \draw[postaction={decorate}, orange] (bd) -- (N);
        \draw[postaction={decorate}, orange] (bc) -- (N);
        \draw[postaction={decorate}, orange] (cd) -- (N);
        \draw[postaction={decorate}, dashed, orange] (bd) -- (S);
        \draw[postaction={decorate}, dashed, orange] (cd) -- (S);
        \end{scope}
        %edges
        \begin{scope}[decoration={
        markings,
        mark=between positions 0.6 and .66 step 0.06 with \arrow{stealth}},very thick]
        \draw[] (b)--(c);
        \draw[] (b)--(d);
        \end{scope}
        \begin{scope}[decoration={
        markings,
        mark=at position 0.6 with \arrow{stealth}},very thick]
        \draw[] (d)--(c);
        \end{scope}    
        \filldraw[orange] (N) circle (0.05em);
        \filldraw[orange] (S) circle (0.05em);
        \filldraw (b) circle (0.05em);
        \filldraw (c) circle (0.05em);
        \filldraw (d) circle (0.05em);
        %\node[anchor=north east] at (bd) {$a$};
        %\node[anchor=west] at (bc) {$b$};
        %\node[anchor=south] at (cd) {$c$};
        %\node[anchor=south] at (c) {$\alpha$};
        %\node[anchor=east] at (d) {$\beta$};
        %\node[anchor=north] at (b) {$\gamma$};
        \node[above right,scale=.75] at (kb3) {$+$};
        \node[below right,scale=.75] at (kb4) {$-$};    
        \end{scope}
        \end{tikzpicture}
        \caption{Biangle bad arcs.}
        \label{fig:biangle_bad_arcs}
    \end{subfigure}
    \caption{Bad arcs in face suspensions.}
    \label{fig:main}
\end{figure}

Now, let $Y$ be a $3$-manifold equipped with an ideal triangulation $\mathcal{T}$.
Denote by $F$ the set of faces of this ideal triangulation. 
\begin{thm} \label{thm:splitting_map}
    There is a well-defined splitting map, 
    \begin{align*}
    \overline{\sigma}: \mathrm{Sk}(Y) &\rightarrow \overline{\bigotimes}_{f\in F} \overline{\mathrm{Sk}}(Sf), \\
    [L] &\mapsto \left[ \sum_{\vec{\epsilon}\in\{\pm 1\}^{\cup_f Sf \cap L}} \otimes_{f\in F}\left[L_f^{\vec{\epsilon}}\right] \right],
    \end{align*}
    where each $\vec{\epsilon}$ is an assignment of states to the newly created boundary points of $L$\footnote{As usual, we require this assignment of states to be \textit{compatible}, i.e., two boundary points that were ``glued'' together in the uncut manifold must carry the same state.}, $L_f^{\vec{\epsilon}}$ is the part of $L$ in $Sf$ after splitting with its newly assigned boundary states, and $\overline{\bigotimes}_{f\in F} \overline{\mathrm{Sk}}(Sf)$ is the relative tensor product of the $\overline{\mathrm{Sk}}(Sf)$ bimodules (see Definition $2.7$ in \cite{PP2}).
\end{thm}

The reduced stated skein module of a face suspension is quite simple. 
\begin{prop}
    \[
    \overline{\mathrm{Sk}}(Sf) \cong \frac{\mathbb{T}^{\otimes 2} \otimes \mathbb{B}^{\otimes 3}}{
    \langle
    (-A^2)\alpha_1\alpha_2 = x_b x_c, \;
    (-A^2)\beta_1\beta_2 = x_c x_a, \;
    (-A^2)\gamma_1\gamma_2 = x_a x_b
    \rangle}
    \]
    as a $\mathbb{T}^{\otimes 2} \otimes \mathbb{B}^{\otimes 3}$-module,
    where
    \begin{align*}
    \mathbb{T}^{\otimes 2} &=
    \frac{R\langle \alpha_1^{\pm 1}, \beta_1^{\pm 1}, \gamma_1^{\pm 1}\rangle}{\langle \beta_1\alpha_1 = A\alpha_1\beta_1,\; \gamma_1\beta_1 = A\beta_1\gamma_1,\; \alpha_1\gamma_1 = A\gamma_1\alpha_1 \rangle}\\
    &\quad \otimes
    \frac{R\langle \alpha_2^{\pm 1}, \beta_2^{\pm 1}, \gamma_2^{\pm 1}\rangle}{\langle \alpha_2\beta_2 = A\beta_2\alpha_2,\; \beta_2\gamma_2 = A\gamma_2\beta_2,\; \gamma_2\alpha_2 = A\alpha_2\gamma_2 \rangle}
    \end{align*}
    and
    \[
    \mathbb{B}^{\otimes 3} = R[x_a^{\pm 1}, x_b^{\pm 1}, x_c^{\pm 1}]. 
    \]
\end{prop}
In this presentation, $\alpha_1$ is the skein that connects the edge cones abutting the vertex labeled by $\alpha$ in Figure~\ref{fig:face_suspension_markings} and sitting in the top tetrahedron with the state $+$ on both ends. 
The other generators can be realized similarly. 

We use this presentation to build a quantum trace on each face suspension. 
\begin{defn}
    The \emph{face suspension module} $\mathbb{S}f$ of $Sf$ is the quantum torus
    \begin{align*}
    \mathbb{S}f := \widetilde{\mathbb{T}}^{\otimes 2} &= 
    \frac{R\langle a_1^{\pm 1}, b_1^{\pm 1}, c_1^{\pm 1}\rangle}{\langle b_1a_1 = Aa_1b_1,\; c_1b_1 = Ab_1c_1,\; a_1c_1 = Ac_1a_1\rangle} \\
    &\quad \otimes \frac{R\langle a_2^{\pm 1}, b_2^{\pm 1}, c_2^{\pm 1}\rangle}{\langle a_2b_2 = Ab_2a_2,\; b_2c_2 = Ac_2b_2,\; c_2a_2 = Aa_2c_2\rangle},
    \end{align*}
    where we think of the $6$ generators being associated to the $6$ faces (edge cones) of the face suspension.
\end{defn}

\begin{prop}
There is a well-defined $\mathbb{T}^{\otimes 2}$-$\mathbb{B}^{\otimes 3}$-bimodule homomorphism
\[
\Tr_{Sf} : \overline{\mathrm{Sk}}(Sf) \rightarrow \mathbb{S}f
\]
mapping the empty skein $[\emptyset]$ to $1$, induced by the embeddings of algebras
\begin{align*}
\mathbb{T}^{\otimes 2} &\hookrightarrow \widetilde{\mathbb{T}}^{\otimes 2}\\
\alpha_1 &\mapsto -(-A^2)^{-\frac 1 2} [b_1c_1],\\
\beta_1 &\mapsto -(-A^2)^{-\frac 1 2} [c_1a_1],\\
\gamma_1 &\mapsto -(-A^2)^{-\frac 1 2} [a_1b_1],\\
\alpha_2 &\mapsto -(-A^2)^{-\frac 1 2} [b_2c_2],\\
\beta_2 &\mapsto -(-A^2)^{-\frac 1 2} [c_2a_2],\\
\gamma_2 &\mapsto -(-A^2)^{-\frac 1 2} [a_2b_2],
\end{align*}
and
\begin{align*}
\mathbb{B}^{\otimes 3} &\hookrightarrow \widetilde{\mathbb{T}}^{\otimes 2}\\
x_a &\mapsto a_1 \otimes a_2,\\
x_b &\mapsto b_1 \otimes b_2,\\
x_c &\mapsto c_1 \otimes c_2.
\end{align*}
\end{prop}

After splitting into face suspensions, an edge cone $e$ of a tetrahedron is naturally associated to two faces of two face suspensions. 
Let the corresponding face suspension module variables be denoted $a$ and $a'$. 
The element $\hat{z}_e = a\otimes a' \in \bigotimes_{f \in F}\mathbb{S}f$ is called the \textit{square-root quantized shape parameter} associated to the edge $e$.
The co-domain of the quantum trace is a quotient of the quantum torus generated by these square-root quantized shape parameters. 
\begin{defn}
    The \emph{square-root quantum gluing module} $\mathrm{SQGM}_{\mathcal{T}}(Y)$ is the two-sided quotient of the quantum torus 
    \[
    \bigotimes_{T \in \mathcal{T}} \frac{R \langle \hat{z}^{\pm 1}, \hat{z}^{'\pm 1}, \hat{z}^{''\pm 1}\rangle}{\langle \hat{z}\hat{z}' = A \hat{z}' \hat{z},\; \hat{z}'\hat{z}'' = A \hat{z}'' \hat{z}',\; \hat{z}''\hat{z} = A \hat{z} \hat{z}''\rangle},
    \]
    where $\hat{z}$, $\hat{z}'$, $\hat{z}''$ are generators associated to the 3 pairs of opposite edges of $T$ as shown in Figure~\ref{fig:shape_parameter_labels} by the following relations: 
    \begin{enumerate}
    \item For each tetrahedron, 
    \[
    [\hat{z} \hat{z}' \hat{z}''] = -(-A^2)^{\frac 1 2},
    \]
    and
    \[
    1 = \hat{z}^{2} + \hat{z}''^{-2},
    \]
    \item \label{item:gluing_rel} For each internal edge $e$ of the triangulation, 
    \[
    \hat{e} = -A^2,
    \]
    where
    \[
    \hat{e} := \left[\prod_{\hat{z}\text{ abutting }e} \hat{z}\right].
    \] 
    \end{enumerate}
\end{defn}

\begin{figure}
    \centering
    \tdplotsetmaincoords{60}{80}
    \begin{tikzpicture}[tdplot_main_coords, scale=0.8]
    \begin{scope}[scale = 0.8, tdplot_main_coords]
        \newcommand*{\defcoords}{
            \coordinate (a) at (0, 0, 3);
            \coordinate (b) at ({2*sqrt(2)}, 0, -1);
            \coordinate (c) at ({-sqrt(2)}, {sqrt(6)}, -1);
            \coordinate (d) at ({-sqrt(2)}, {-sqrt(6)}, -1);
            \coordinate (e) at (0, 0, -5);
            \coordinate (ab) at ($1/2*(a) + 1/2*(b)$);
            \coordinate (ac) at ($1/2*(a) + 1/2*(c)$);
            \coordinate (ad) at ($1/2*(a) + 1/2*(d)$);
            \coordinate (eb) at ($1/2*(e) + 1/2*(b)$);
            \coordinate (ec) at ($1/2*(e) + 1/2*(c)$);
            \coordinate (ed) at ($1/2*(e) + 1/2*(d)$);
            \coordinate (bc) at ($1/2*(b) + 1/2*(c)$);
            \coordinate (bd) at ($1/2*(b) + 1/2*(d)$);
            \coordinate (cd) at ($1/2*(c) + 1/2*(d)$);
        }
        
        \defcoords
    
        \draw[white, line width=5] (b) -- (c);
        \draw[white, line width=5] (b) -- (d);
        \draw[very thick] (b) -- (c);
        \draw[very thick] (b) -- (d);
        
        \draw[very thick] (c) -- (d);
        
        \draw[white, line width=5] (a) -- (b);
        \draw[very thick] (a) -- (b);
        
        \draw[very thick] (b) -- (c);
        \draw[very thick] (b) -- (d);
        
        \draw[very thick] (a) -- (c);
        \draw[very thick] (a) -- (d);
        
        \filldraw (b) circle (0.05em);
        \filldraw (c) circle (0.05em);
        \filldraw (d) circle (0.05em);
        \filldraw (a) circle (0.05em);
        
        \node[left] at (ad){$\hat{z}$};
        \node[right] at (bc){$\hat{z}$};
        \node[right] at (ab){$\hat{z}''$};
        \node[below] at (cd){$\hat{z}''$};
        \node[right] at (ac){$\hat{z}'$};
        \node[below] at (bd){$\hat{z}'$};
    
    \end{scope}
    \end{tikzpicture}
    \caption{The ordering of the shape parameters on the pairs of edges of a tetrahedron.}
    \label{fig:shape_parameter_labels}
\end{figure}
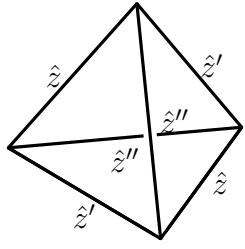
The quantum trace is then the composition of splitting and applying the (tensor product of the) face suspension quantum trace. 
\begin{thm} 
    There is an $R$-module homomorphism $\Tr_{\mathcal{T}} : \mathrm{Sk}(Y) \rightarrow \mathrm{SQGM}_{\mathcal{T}}(Y)$ defined as the composition 
    \[
    \begin{tikzcd}
    \mathrm{Sk}(Y) \arrow[rrd, bend right, swap, "\Tr_\mathcal{T}"] \arrow[r, "\overline{\sigma}"] & \underset{f\in F}{\overline{\bigotimes}} \overline{\mathrm{Sk}}(Sf) \arrow[r, "\overline{\otimes}_{f\in F} \Tr_{Sf}"] &[3em] \underset{f\in F}{\overline{\bigotimes}} \mathbb{S}f \\
    & & \mathrm{SQGM}_{\mathcal{T}}(Y) \arrow[hookrightarrow, u, ""]
    \end{tikzcd}.
    \]
\end{thm}

\subsection{The length conjecture}

We can finally state the length conjecture.
Start with an oriented, ideally triangulated $r$-cusped hyperbolic $3$-manifold $M$ with triangulation $\mathcal{T}$, together with a designated cusp $\partial_i M$ and a slope $(\mathfrak{p}, \mathfrak{q})$ on $\partial_i M$ such that $M(\mathfrak{p},\mathfrak{q}) = S^3 \setminus \mathcal{K}$ for some hyperbolic link $\mathcal{K} \subset S^3$.
Call a framed link $K \subset M(\mathfrak{p},\mathfrak{q})$ \textit{even} if it represents the zero homology class in $H_1(M(\mathfrak{p},\mathfrak{q}); \mathbb{Z}_2) \cong \mathbb{Z}_2^{r-1}$.
Given such an even link $K \subset M(\mathfrak{p},\mathfrak{q})$ that is disjoint from the core of the filling solid torus, $K$ lifts canonically to a framed link in $M$, which we also denote by $K \subset M$.
The $3$d quantum trace map of Section~\ref{sec:quantum_trace} produces an element $\Tr_\mathcal{T}(K) \in \mathrm{SQGM}_\mathcal{T}(M)$.
This element can be inserted into the integrand of the Dehn-filled state integral $Z_\hbar^{(\mathfrak{p},\mathfrak{q})}(\mathcal{T}; X)$ of Equation \eqref{dehnfillpf}, yielding a modified state integral $Z_\hbar^{(\mathfrak{p},\mathfrak{q})}(\mathcal{T}; X; \Tr_\mathcal{T}(K))$.
Expanding the ratio of Dehn-filled state integrals around the geometric saddle point at $X = 0$ yields a formal power series in $\hbar$:
\[
\log \left( \frac{Z_\hbar^{(\mathfrak{p},\mathfrak{q})}(\mathcal{T}; 0; \Tr_\mathcal{T}(K))}{Z_\hbar^{(\mathfrak{p},\mathfrak{q})}(\mathcal{T}; 0)} \right) \xrightarrow{\hbar \to 0} \sum_{s \geq 0} W_s^{\mathrm{hyp}}(K; \mathcal{T}, \mathfrak{p}, \mathfrak{q})\, \hbar^s.
\]
The leading coefficient $W_0^{\mathrm{hyp}}(K; \mathcal{T}, \mathfrak{p}, \mathfrak{q})$ is a simple function of the complex hyperbolic length $\ell_\mathbb{C}(K)$ of the geodesic representative of $K$ in $M(\mathfrak{p},\mathfrak{q})$:
\[
W_0^{\mathrm{hyp}}(K; \mathcal{T}, \mathfrak{p}, \mathfrak{q}) = \log \left( - 2 \cosh\left( \tfrac{1}{2} \ell_\mathbb{C}(K) \right) \right).
\]
The length conjecture states that this asymptotic series can be completely recovered from the colored Jones polynomial with an ``insertion'' of $K$.
\begin{conj}[Dehn-filled length conjecture, after \cite{AGLR}]
For every oriented, ideally triangulated $r$-cusped hyperbolic $3$-manifold $(M, \mathcal{T})$, every slope $(\mathfrak{p},\mathfrak{q})$ on a designated cusp of $M$ with $M(\mathfrak{p},\mathfrak{q}) = S^3 \setminus \mathcal{K}$ for some hyperbolic link $\mathcal{K} \subset S^3$, and every even framed link $K \subset M(\mathfrak{p},\mathfrak{q})$ disjoint from the core of the filling solid torus, the asymptotic expansion of a particular ratio of colored Jones polynomials, evaluated at the root of unity $q = e^{2 \pi i / n}$, matches the perturbative expansion of the Dehn-filled state integral with insertion of $\Tr_\mathcal{T}(K)$, under the identification $\hbar = 2 \pi i / n$:
\[
\left. \log \left( \frac{J_{n, 2}(\mathcal{K} \cup K;q)}{J_n(\mathcal{K};q)} \right) \right|_{q =  e^{2 \pi i / n}} \xrightarrow{n \to \infty} \sum_{s \geq 0} W_s^{\mathrm{hyp}}(K; \mathcal{T}, \mathfrak{p}, \mathfrak{q})\, \hbar^s.
\]
\end{conj}
In the special case where $M = S^3 \setminus \mathcal{K}$ has a single cusp and no Dehn filling is performed, the Dehn-filled state integral $Z_\hbar^{(\mathfrak{p},\mathfrak{q})}(\mathcal{T}; X)$ reduces to the standard state integral $Z_\hbar(\mathcal{T}; X)$ of Definition~\ref{def:state_integral}, and the conjecture above recovers the original length conjecture of \cite{AGLR}.

%%%%%%%%%%%%%%%%%%%%%%%%%%%%%%%%%%%%%%%%%%%%%%%%

\section{Skein modules and triangulations of twist knots} \label{sec:skein_modules_and_triangulations}

This section is dedicated to establishing our notation for the basis of the skein module of twist knots, in addition to obtaining a triangulation of twist knot complements through a Dehn filling on the complement of the Whitehead link. 

\subsection{A basis for the skein module of twist knots}

For any integer $p$, we define the twist knot $\mathcal{K}_p$ by inserting $p$ full twists into the twist region of the diagram shown in Figure~\ref{fig:twist_knot_curves}, where the sign of $p$ dictates the direction of the twists. Familiar examples include the left-handed trefoil knot ($p=1$) and the figure-eight knot ($p=-1$).

To investigate the length conjecture for this family, we first need a clear picture of the skein module of the knot complement.
We consider two specific curves in the complement of $\mathcal{K}_p$, denoted $K_m$ and $K_b$. 
The curve $K_m$ is the standard meridian of the knot, while $K_b$ is a second curve that wraps around the twist region; these curves are shown in Figure~\ref{fig:twist_knot_curves}.
Equip them with the blackboard framing.

\begin{figure}[h]
    \centering
    \includegraphics[scale=.75]{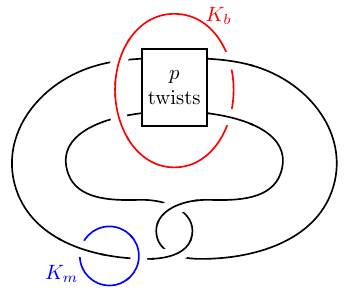}
    \caption{The twist knot $\mathcal{K}_p$ and the curves $K_m$ (shown in blue) and $K_b$ (shown in red).}
    \label{fig:twist_knot_curves}
\end{figure}

The structure of the skein module can be described explicitly using these skeins:

\begin{thm}[\cite{BL}] \label{thm:twist_knot_basis}
The skein module $\mathrm{Sk}(S^3 \setminus \mathcal{K}_p)$ is free with basis 
\[ \{K_m^k K_b^l \mid k \ge 0, \ 0 \le l \le 2|p|\}. \]
\end{thm}

\subsection{A triangulation of the Whitehead link complement}
We will study twist knots via Dehn filling on the Whitehead link. 
To this end, we first need an ideal triangulation of the Whitehead link complement. 

Begin with the standard  diagram of the Whitehead link shown in Figure~\ref{fig:whitehead_link}.
\begin{figure}[h]
    \centering
    \includegraphics[scale=.75]{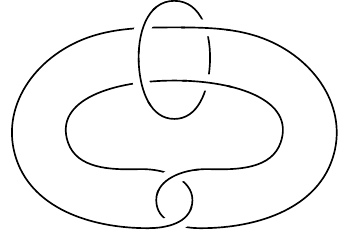}
    \caption{A standard diagram of the Whitehead link.}
    \label{fig:whitehead_link}
\end{figure}
To construct the triangulation, we use a hybrid approach that integrates the standard polyhedral decomposition algorithm (see, e.g., \cite{M2}) with the techniques specifically designed for fully augmented links \cite{Pur2}. 

As in the standard algorithm, we start by slicing the link complement along the projection sphere, which separates the space into two 3-balls. 
Call the upper ball $B_+$ and the lower ball $B_-$. 
The loop surrounding the parallel strands is a crossing circle, so according to the algorithm in \cite{Pur2}, we additionally wish to cut along the disk bounded by this circle. 
The projection sphere splits this spanning disk into two half-disks, one lying in $B_+$ and one in $B_-$. 
We slice the respective 3-balls open across these half-disks. 
Figure~\ref{fig:half_disk_cuts} illustrates this slicing process specifically for the top ball. 

\begin{figure}[h]
    \centering
    \includegraphics{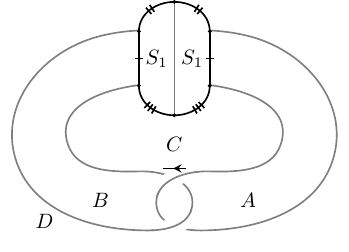}
    \caption{Slicing the upper 3-ball ($B_+$) along the half-disk bounded by the crossing circle. A similar process is carried out in $B_-$. We've also labeled regions of the projection and highlighted the future edges of our polyhedral decomposition in black.}
    \label{fig:half_disk_cuts}
\end{figure}

It is now straightforward to follow the usual triangulation procedure. 
Keeping track of how all of the faces glue back together, we obtain the final configurations on the boundaries of $B_+$ and $B_-$ shown in Figure~\ref{fig:3ball_boundaries}.

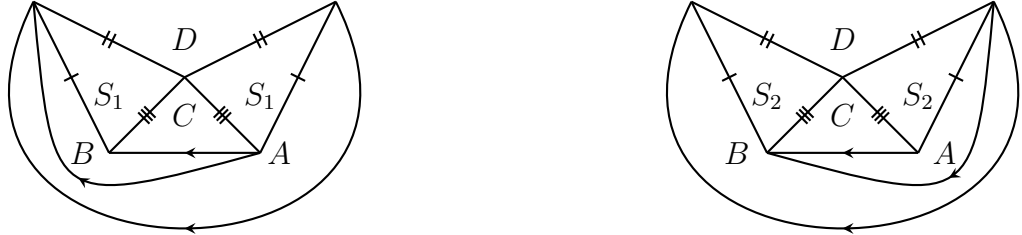
\begin{figure}[h]
    \centering
    \begin{subfigure}[t]{.45\textwidth}
        \centering
        \begin{tikzpicture}[scale=1]
        \begin{scope}[thick,
        decoration = {markings,mark=at position 0.5 with 
        {\draw(0,-3pt)--(0,3pt);}}
        ] 
        \draw[thick,postaction={decorate}] (1,-1)--(2,1);
        \draw[thick,postaction={decorate}] (-1,-1)--(-2,1);
        \end{scope}            
        \begin{scope}[thick,decoration = {markings,mark=between positions 0.48 and .52 step .04 with 
        {\draw(0,-3pt)--(0,3pt);}}
        ] 
        \draw[thick,postaction={decorate}] (0,0)--(2,1);
        \draw[thick,postaction={decorate}] (0,0)--(-2,1);
        \end{scope}       
        \begin{scope}[thick,decoration = {markings,mark=between positions 0.46 and .54 step .04 with 
        {\draw(0,-3pt)--(0,3pt);}}
        ] 
        \draw[thick,postaction={decorate}] (0,0)--(1,-1);
        \draw[thick,postaction={decorate}] (0,0)--(-1,-1);
        \end{scope}
        \begin{scope}[thick,decoration={
        markings,
        mark=at position 0.5 with \arrow{stealth}}
        ] 
        \draw[thick,postaction={decorate}] (1,-1) -- (-1,-1);
        \draw[thick,postaction={decorate}] (1,-1) ..controls (-1.75,-1.75) .. (-2,1);
        \draw[thick,postaction={decorate}] (2,1) ..controls (4,-3) and (-4,-3).. (-2,1);
        \end{scope}
        \node[] at (0,-.5) {$C$};
        \node[] at (1,-.25) {$S_1$};
        \node[] at (-1,-.25) {$S_1$};
        \node[] at (1.25,-1) {$A$};
        \node[] at (-1.35,-1) {$B$};
        \node[] at (0,.5) {$D$};
        \end{tikzpicture}
        \caption{The triangulation of $B_+$. Face $D$ wraps above the diagram.}
    \end{subfigure}
    \hspace{1cm}
    \begin{subfigure}[t]{.45\textwidth}
        \centering  
        \begin{tikzpicture}[scale=1]
        \begin{scope}[thick,
        decoration = {markings,mark=at position 0.5 with 
        {\draw(0,-3pt)--(0,3pt);}}
        ] 
        \draw[thick,postaction={decorate}] (1,-1)--(2,1);
        \draw[thick,postaction={decorate}] (-1,-1)--(-2,1);
        \end{scope}       
        \begin{scope}[thick,decoration = {markings,mark=between positions 0.48 and .52 step .04 with 
        {\draw(0,-3pt)--(0,3pt);}}
        ] 
        \draw[thick,postaction={decorate}] (0,0)--(2,1);
        \draw[thick,postaction={decorate}] (0,0)--(-2,1);
        \end{scope}       
        \begin{scope}[thick,decoration = {markings,mark=between positions 0.46 and .54 step .04 with 
        {\draw(0,-3pt)--(0,3pt);}}
        ] 
        \draw[thick,postaction={decorate}] (0,0)--(1,-1);
        \draw[thick,postaction={decorate}] (0,0)--(-1,-1);
        \end{scope}
        \begin{scope}[thick,decoration={
        markings,
        mark=at position 0.5 with \arrow{stealth}}
        ] 
        \draw[thick,postaction={decorate}] (1,-1) -- (-1,-1);
        \draw[thick,postaction={decorate}] (2,1) ..controls (1.75,-1.75) .. (-1,-1);
        \draw[thick,postaction={decorate}] (2,1) ..controls (4,-3) and (-4,-3).. (-2,1);
        \end{scope}
        \node[] at (0,-.5) {$C$};
        \node[] at (1,-.25) {$S_2$};
        \node[] at (-1,-.25) {$S_2$};
        \node[] at (1.35,-1) {$A$};
        \node[] at (-1.4,-1) {$B$};
        \node[] at (0,.5) {$D$};
        \end{tikzpicture}
        \caption{The triangulation of $B_-$. Face $D$ wraps below the diagram.}
    \end{subfigure}    
    \caption{Triangulations of the boundaries of the upper and lower 3-balls induced by the link diagram and the half-disk cuts.}
    \label{fig:3ball_boundaries}
\end{figure}

The polyhedral decomposition can be cut into four ideal tetrahedra, depicted in Figure~\ref{fig:Whitehead_tetrahedra}. 

\begin{figure}[htbp]
    \centering
    \begin{gather*}
    \vcenter{\hbox{
    \tdplotsetmaincoords{60}{80}
    \begin{tikzpicture}[tdplot_main_coords]
    \begin{scope}[scale = 0.8, tdplot_main_coords]
        \coordinate (o) at (0, 0, 0);
        \coordinate (a) at (0, 0, 3);   
        \coordinate (b) at ({2*sqrt(2)}, 0, -1);
        \coordinate (c) at ({-sqrt(2)}, {sqrt(6)}, -1);   
        \coordinate (d) at ({-sqrt(2)}, {-sqrt(6)}, -1);
        \coordinate (ab) at ({sqrt(2)}, 0, 1);
        \coordinate (ac) at ({-sqrt(2)/2}, {sqrt(6)/2}, 1);
        \coordinate (ad) at ({-sqrt(2)/2}, {-sqrt(6)/2}, 1);
        \coordinate (bc) at ({sqrt(2)/2}, {sqrt(6)/2}, -1);
        \coordinate (bd) at ({sqrt(2)/2}, {-sqrt(6)/2}, -1);
        \coordinate (cd) at ({-sqrt(2)}, 0, -1);
        %add back edge
        \begin{scope}[thick,decoration = {markings,mark=between positions 0.46 and .54 step .04 with 
        {\draw(0,-3pt)--(0,3pt);}}
        ] 
        \draw[postaction={decorate},very thick] (d) -- (c);
        \end{scope}
        %adds appearance of crossing info
        \draw[white, ultra thick] (a) -- (b);   
        %draw "blue edges"
        \begin{scope}[thick,decoration={
        markings,
        mark=at position 0.5 with \arrow{stealth}}
        ] 
        \draw[postaction={decorate},very thick] (b) -- (a);
        \draw[postaction={decorate},very thick] (b) -- (d);
        \end{scope}
        %draw other edges 
        \begin{scope}[thick,
        decoration = {markings,mark=between positions 0.48 and .52 step .04 with 
        {\draw(0,-3pt)--(0,3pt);}}
        ] 
        \draw[postaction={decorate},very thick] (c) -- (a);
        \end{scope} 
        \begin{scope}[thick,
        decoration = {markings,mark=at position 0.5 with 
        {\draw(0,-3pt)--(0,3pt);}}
        ] 
        \draw[postaction={decorate},very thick] (d) -- (a);
        \end{scope}   
        \begin{scope}[thick,decoration = {markings,mark=between positions 0.46 and .54 step .04 with 
        {\draw(0,-3pt)--(0,3pt);}}
        ] 
        \draw[postaction={decorate},very thick] (b) -- (c);
        \end{scope}
        \filldraw (a) circle (0.05em);
        \filldraw (b) circle (0.05em);
        \filldraw (c) circle (0.05em);
        \filldraw (d) circle (0.05em);
        \node[anchor = south east] at (ad) {$z''_1$};
        \node[anchor = south west] at (ac) {$z_1$};
        \node[anchor = north east] at (bd) {$z_1$};
        \node[anchor = north west] at (bc) {$z''_1$};
        \node[above left, yshift = 5mm] at (ab) {$z'_1$};
        \node[below] at (cd) {$z'_1$};
        %\node[above] at (a) {$a$};
        %\node[below] at (b) {$b$};
        %\node[right] at (c) {$c$};
        %\node[left] at (d) {$d$};
    \end{scope}
    \end{tikzpicture}
    }}
    \hspace{1.5cm}
    \vcenter{\hbox{
    \tdplotsetmaincoords{60}{80}
    \begin{tikzpicture}[tdplot_main_coords]
    \begin{scope}[scale = 0.8, tdplot_main_coords]
        \coordinate (o) at (0, 0, 0);
        \coordinate (a) at (0, 0, 3);   
        \coordinate (b) at ({2*sqrt(2)}, 0, -1);
        \coordinate (c) at ({-sqrt(2)}, {sqrt(6)}, -1);   
        \coordinate (d) at ({-sqrt(2)}, {-sqrt(6)}, -1);
        \coordinate (ab) at ({sqrt(2)}, 0, 1);
        \coordinate (ac) at ({-sqrt(2)/2}, {sqrt(6)/2}, 1);
        \coordinate (ad) at ({-sqrt(2)/2}, {-sqrt(6)/2}, 1);
        \coordinate (bc) at ({sqrt(2)/2}, {sqrt(6)/2}, -1);
        \coordinate (bd) at ({sqrt(2)/2}, {-sqrt(6)/2}, -1);
        \coordinate (cd) at ({-sqrt(2)}, 0, -1);
        %add back edge
        \begin{scope}[thick,decoration = {markings,mark=between positions 0.48 and .52 step .04 with 
        {\draw(0,-3pt)--(0,3pt);}}
        ] 
        \draw[postaction={decorate},very thick] (d) -- (c);
        \end{scope}
        %adds appearance of crossing info
        \draw[white, ultra thick] (a) -- (b);   
        %draw "blue edges"
        \begin{scope}[thick,decoration={
        markings,
        mark=at position 0.5 with \arrow{stealth}}
        ] 
        \draw[postaction={decorate},very thick] (b) -- (a);
        \draw[postaction={decorate},very thick] (b) -- (d);
        \end{scope}
        %draw other edges 
        \begin{scope}[thick,decoration = {markings,mark=between positions 0.46 and .54 step .04 with 
        {\draw(0,-3pt)--(0,3pt);}}
        ] 
        \draw[postaction={decorate},very thick] (a) -- (c);
        \end{scope} 
        \begin{scope}[thick,
        decoration = {markings,mark=at position 0.5 with 
        {\draw(0,-3pt)--(0,3pt);}}
        ] 
        \draw[postaction={decorate},very thick] (d) -- (a);
        \end{scope}  
        \begin{scope}[thick,decoration = {markings,mark=between positions 0.48 and .52 step .04 with 
        {\draw(0,-3pt)--(0,3pt);}}
        ] 
        \draw[postaction={decorate},very thick] (b) -- (c);
        \end{scope}
        \filldraw (a) circle (0.05em);
        \filldraw (b) circle (0.05em);
        \filldraw (c) circle (0.05em);
        \filldraw (d) circle (0.05em);
        \node[anchor = south east] at (ad) {$z''_2$};
        \node[anchor = south west] at (ac) {$z_2$};
        \node[anchor = north east] at (bd) {$z_2$};
        \node[anchor = north west] at (bc) {$z''_2$};
        \node[above left, yshift = 5mm] at (ab) {$z'_2$};
        \node[below] at (cd) {$z'_2$};
        %\node[above] at (a) {$N$};
        %\node[below] at (b) {$E$};
        %\node[right] at (c) {$S$};
        %\node[left] at (d) {$W$};
    \end{scope}
    \end{tikzpicture}
    }}\\
    \vcenter{\hbox{
    \tdplotsetmaincoords{60}{80}
    \begin{tikzpicture}[tdplot_main_coords]
    \begin{scope}[scale = 0.8, tdplot_main_coords]
        \coordinate (o) at (0, 0, 0);
        \coordinate (a) at (0, 0, 3);   
        \coordinate (b) at ({2*sqrt(2)}, 0, -1);
        \coordinate (c) at ({-sqrt(2)}, {sqrt(6)}, -1);   
        \coordinate (d) at ({-sqrt(2)}, {-sqrt(6)}, -1);
        \coordinate (ab) at ({sqrt(2)}, 0, 1);
        \coordinate (ac) at ({-sqrt(2)/2}, {sqrt(6)/2}, 1);
        \coordinate (ad) at ({-sqrt(2)/2}, {-sqrt(6)/2}, 1);
        \coordinate (bc) at ({sqrt(2)/2}, {sqrt(6)/2}, -1);
        \coordinate (bd) at ({sqrt(2)/2}, {-sqrt(6)/2}, -1);
        \coordinate (cd) at ({-sqrt(2)}, 0, -1);
        %add back edge
        \begin{scope}[thick,decoration = {markings,mark=between positions 0.46 and .54 step .04 with 
        {\draw(0,-3pt)--(0,3pt);}}
        ] 
        \draw[postaction={decorate},very thick] (d) -- (c);
        \end{scope}
        %adds appearance of crossing info
        \draw[white, ultra thick] (a) -- (b);   
        %draw "blue edges"
        \begin{scope}[thick,decoration={
        markings,
        mark=at position 0.5 with \arrow{stealth}}
        ] 
        \draw[postaction={decorate},very thick] (b) -- (a);
        \draw[postaction={decorate},very thick] (d) -- (a);
        \end{scope}
        %draw other edges 
        \begin{scope}[thick,
        decoration = {markings,mark=between positions 0.48 and .52 step .04 with 
        {\draw(0,-3pt)--(0,3pt);}}
        ] 
        \draw[postaction={decorate},very thick] (b) -- (c);
        \end{scope} 
        \begin{scope}[thick,
        decoration = {markings,mark=between positions 0.48 and .52 step .04 with 
        {\draw(0,-3pt)--(0,3pt);}}
        ] 
        \draw[postaction={decorate},very thick] (a) -- (c);
        \end{scope} 
        \begin{scope}[thick,
        decoration = {markings,mark=at position 0.5 with 
        {\draw(0,-3pt)--(0,3pt);}}
        ] 
        \draw[postaction={decorate},very thick] (b) -- (d);
        \end{scope}    
        \filldraw (a) circle (0.05em);
        \filldraw (b) circle (0.05em);
        \filldraw (c) circle (0.05em);
        \filldraw (d) circle (0.05em);
        \node[anchor = south east] at (ad) {$z'_3$};
        \node[anchor = south west] at (ac) {$z''_3$};
        \node[anchor = north east] at (bd) {$z''_3$};
        \node[anchor = north west] at (bc) {$z'_3$};
        \node[above left, yshift = 5mm] at (ab) {$z_3$};
        \node[below] at (cd) {$z_3$};
        %\node[above] at (a) {$N$};
        %\node[below] at (b) {$E$};
        %\node[right] at (c) {$S$};
        %\node[left] at (d) {$W$};
    \end{scope}
    \end{tikzpicture}
    }}
    \hspace{1.5cm}
    \vcenter{\hbox{
    \tdplotsetmaincoords{60}{80}
    \begin{tikzpicture}[tdplot_main_coords]
    \begin{scope}[scale = 0.8, tdplot_main_coords]
        \coordinate (o) at (0, 0, 0);
        \coordinate (a) at (0, 0, 3);   
        \coordinate (b) at ({2*sqrt(2)}, 0, -1);
        \coordinate (c) at ({-sqrt(2)}, {sqrt(6)}, -1);   
        \coordinate (d) at ({-sqrt(2)}, {-sqrt(6)}, -1);
        \coordinate (ab) at ({sqrt(2)}, 0, 1);
        \coordinate (ac) at ({-sqrt(2)/2}, {sqrt(6)/2}, 1);
        \coordinate (ad) at ({-sqrt(2)/2}, {-sqrt(6)/2}, 1);
        \coordinate (bc) at ({sqrt(2)/2}, {sqrt(6)/2}, -1);
        \coordinate (bd) at ({sqrt(2)/2}, {-sqrt(6)/2}, -1);
        \coordinate (cd) at ({-sqrt(2)}, 0, -1);
        %add back edge
        \begin{scope}[thick,
        decoration = {markings,mark=between positions 0.48 and .52 step .04 with 
        {\draw(0,-3pt)--(0,3pt);}}
        ] 
        \draw[postaction={decorate},very thick] (d) -- (c);
        \end{scope}
        %adds appearance of crossing info
        \draw[white, ultra thick] (a) -- (b);   
        %draw "blue edges"
        \begin{scope}[thick,decoration={
        markings,
        mark=at position 0.5 with \arrow{stealth}}
        ] 
        \draw[postaction={decorate},very thick] (b) -- (a);
        \draw[postaction={decorate},very thick] (d) -- (a);
        \end{scope}
        %draw other edges 
        \begin{scope}[thick,
        decoration = {markings,mark=between positions 0.46 and .54 step .04 with 
        {\draw(0,-3pt)--(0,3pt);}}
        ] 
        \draw[postaction={decorate},very thick] (b) -- (c);
        \end{scope} 
        \begin{scope}[thick,
        decoration = {markings,mark=between positions 0.46 and .54 step .04 with 
        {\draw(0,-3pt)--(0,3pt);}}
        ] 
        \draw[postaction={decorate},very thick] (a) -- (c);
        \end{scope} 
        \begin{scope}[thick,
        decoration = {markings,mark=at position 0.5 with 
        {\draw(0,-3pt)--(0,3pt);}}
        ] 
        \draw[postaction={decorate},very thick] (b) -- (d);
        \end{scope}    
        \filldraw (a) circle (0.05em);
        \filldraw (b) circle (0.05em);
        \filldraw (c) circle (0.05em);
        \filldraw (d) circle (0.05em);
        \node[anchor = south east] at (ad) {$z'_4$};
        \node[anchor = south west] at (ac) {$z''_4$};
        \node[anchor = north east] at (bd) {$z''_4$};
        \node[anchor = north west] at (bc) {$z'_4$};
        \node[above left, yshift = 5mm] at (ab) {$z_4$};
        \node[below] at (cd) {$z_4$};
        %\node[above] at (a) {$N$};
        %\node[below] at (b) {$E$};
        %\node[right] at (c) {$S$};
        %\node[left] at (d) {$W$};
    \end{scope}
    \end{tikzpicture}
    }}
    \end{gather*}
    \caption{A triangulation of the Whitehead link complement.}
    \label{fig:Whitehead_tetrahedra}
\end{figure}
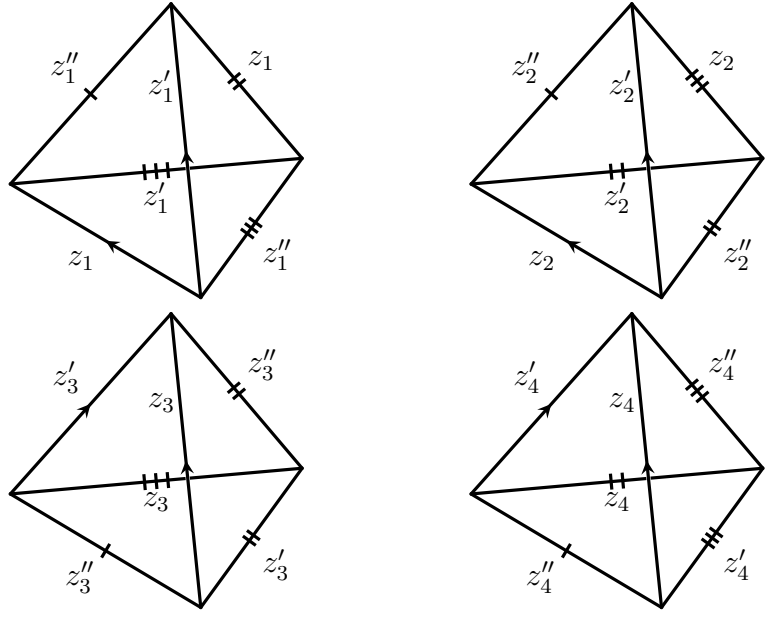

With shape parameters assigned as in Figure~\ref{fig:Whitehead_tetrahedra}, it is straightforward to read off the (log) gluing constraints: 
\begin{eqnarray}
\mathcal{C}_{1}:&&Z''_{1} + Z''_{2} + Z''_{3} + Z''_{4} = -2\pi i,\nonumber\\
\mathcal{C}_{2}:&&Z_{1} + Z'_{2} + Z''_{2} + Z'_{3} + Z''_{3} + Z_{4} = -2\pi i,\nonumber\\
\mathcal{C}_{3}:&&Z'_{1} + Z''_{1} + Z_{2} + Z_{3} + Z'_{4} + Z''_{4} = -2\pi i,\nonumber\\
\mathcal{C}_{4}:&&Z_{1} + Z'_{1} + Z_{2} + Z'_{2} + Z_{3} + Z'_{3} + Z_{4} + Z'_{4} = -2\pi i.
\end{eqnarray}

This triangulation also induces triangulations on the toroidal cusps. 
We record these triangulations in Figure~\ref{fig:cusp_triangulations}.

\begin{figure}[htbp]
    \centering
    \begin{subfigure}[b]{0.44\textwidth}
        \centering
        \begin{tikzpicture}[scale=1]
\draw[very thick] (0,0) -- (8,0) -- (8,4) -- (0,4) --cycle;
\draw[very thick] (4,0) -- (4,4);
\draw[very thick] (0,0) -- (4,4);
\draw[very thick] (4,0) -- (8,4);

\draw[fill=black] (0,0) circle (.1mm);
\draw[fill=black] (0,4) circle (.1mm);
\draw[fill=black] (4,4) circle (.1mm);
\draw[fill=black] (4,0) circle (.1mm);
\draw[fill=black] (8,0) circle (.1mm);
\draw[fill=black] (8,4) circle (.1mm);

%(0,0)
\node[shift={({.5*cos(67.5)}, {.5*sin(67.5)})},scale=.75] at (0,0) {$z''_1$};
\node[shift={({.5*cos(22.5)}, {.5*sin(22.5)})},scale=.75] at (0,0) {$z_3$};
%(0,4)
\node[shift={({.3*cos(-45)}, {.5*sin(-45)})},scale=.75] at (0,4) {$z'_1$};
%(4,0)
\node[shift={({.5*cos(22.5)}, {.5*sin(22.5)})},scale=.75] at (4,0) {$z_4$};
\node[shift={({.5*cos(67.5)}, {.5*sin(67.5)})},scale=.75] at (4,0) {$z''_2$};
\node[shift={({.3*cos(135)}, {.3*sin(135)})},scale=.75] at (4,0) {$z'_3$};
%(4,4)
\node[shift={({.5*cos(-112.5)}, {.5*sin(-112.5)})},scale=.75] at (4,4) {$z''_3$};
\node[shift={({.5*cos(-157.5)}, {.5*sin(-157.5)})},scale=.75] at (4,4) {$z_1$};
\node[shift={({.3*cos(-45)}, {.3*sin(-45)})},scale=.75] at (4,4) {$z'_2$};
%(8,0)
\node[shift={({.3*cos(135)}, {.3*sin(135)})},scale=.75] at (8,0) {$z'_4$};
%(8,4)
\node[shift={({.5*cos(-112.5)}, {.5*sin(-112.5)})},scale=.75] at (8,4) {$z''_4$};
\node[shift={({.5*cos(-157.5)}, {.5*sin(-157.5)})},scale=.75] at (8,4) {$z_2$};

%draw in the meridian and longitude we're going to use
\draw[thick, red, ->] (6,-.5) .. controls (6.5,1) and (5.5,3) .. (6,4.5) coordinate[pos=.25] (m);
\node[left,red] at (m) {$\mathfrak{M}_1$};
\draw[thick, red, ->] (-.5,2.5) .. controls (2,3) and (6,2) .. (8.5,2.5) coordinate[pos=.25] (l);
\node[above,red] at (l) {$\mathfrak{L}_1$};

\end{tikzpicture}
        \caption{The triangulation of the ``small'' cusp.}
    \end{subfigure}
    \hfill 
    \begin{subfigure}[b]{0.44\textwidth}
        \centering
        \begin{tikzpicture}[scale=1]
% --- EDGES ---
\draw[very thick] (0,0) -- (8,0) -- (8,4) -- (0,4) --cycle;
\draw[very thick] (2,0) -- (2,4);
\draw[very thick] (4,0) -- (4,4);
\draw[very thick] (6,0) -- (6,4);
\draw[very thick] (0,4) -- (2,0);
\draw[very thick] (0,4) -- (2,2);
\draw[very thick] (2,4) -- (4,0);
\draw[very thick] (2,2) -- (4,0);
\draw[very thick] (4,0) -- (6,4);
\draw[very thick] (4,0) -- (6,2);
\draw[very thick] (6,2) -- (8,0);
\draw[very thick] (6,4) -- (8,0);

% --- VERTICES ---
\draw[fill=black] (0,0) circle (.1mm);
\draw[fill=black] (2,0) circle (.1mm);
\draw[fill=black] (4,0) circle (.1mm);
\draw[fill=black] (6,0) circle (.1mm);
\draw[fill=black] (8,0) circle (.1mm);
\draw[fill=black] (0,4) circle (.1mm);
\draw[fill=black] (2,4) circle (.1mm);
\draw[fill=black] (4,4) circle (.1mm);
\draw[fill=black] (6,4) circle (.1mm);
\draw[fill=black] (8,4) circle (.1mm);
\draw[fill=black] (2,2) circle (.1mm); % Added missing dot
\draw[fill=black] (6,2) circle (.1mm); % Added missing dot

% --- LABELS ---

% (0,0)
\node[shift={({.3*cos(45)}, {.3*sin(45)})},scale=.75] at (0,0) {$z'_2$};

% (2,0)
\node[shift={({.3*cos(45)}, {.3*sin(45)})},scale=.75] at (2,0) {$z_1$};
\node[shift={({.75*cos(105)}, {.75*sin(105)})},scale=.75] at (2,0) {$z'_3$};
\node[shift={({.5*cos(150)}, {.5*sin(150)})},scale=.75] at (2,0) {$z''_2$};

% (4,0)
\node[shift={({.5*cos(22.5)}, {.5*sin(22.5)})},scale=.75] at (4,0) {$z_1$};
\node[shift={({.75*cos(56)}, {.75*sin(56)})},scale=.75] at (4,0) {$z'_2$};
\node[shift={({.75*cos(79)}, {.75*sin(79)})},scale=.75] at (4,0) {$z'_4$};
\node[shift={({.75*cos(101)}, {.75*sin(101)})},scale=.75] at (4,0) {$z_3$};
\node[shift={({.75*cos(125)}, {.75*sin(125)})},scale=.75] at (4,0) {$z_2$};
\node[shift={({.5*cos(157.5)}, {.5*sin(157.5)})},scale=.75] at (4,0) {$z'_1$};

% (6,0)
\node[shift={({.3*cos(45)}, {.3*sin(45)})},scale=.75] at (6,0) {$z_3$};
\node[shift={({.3*cos(135)}, {.3*sin(135)})},scale=.75] at (6,0) {$z'_1$};

% (8,0)
\node[shift={({.75*cos(101)}, {.75*sin(101)})},scale=.75] at (8,0) {$z_1$};
\node[shift={({.75*cos(122.5)}, {.75*sin(122.5)})},scale=.75] at (8,0) {$z_4$};
\node[shift={({.5*cos(157.5)}, {.5*sin(157.5)})},scale=.75] at (8,0) {$z'_3$};

% (0,4)
\node[shift={({.5*cos(-20)}, {.5*sin(-20)})},scale=.75] at (0,4) {$z'_4$};
\node[shift={({.75*cos(-55)}, {.75*sin(-55)})},scale=.75] at (0,4) {$z_3$};
\node[shift={({.75*cos(-78.5)}, {.75*sin(-78.5)})},scale=.75] at (0,4) {$z_2$};

% (2,4)
\node[shift={({.5*cos(-30)}, {.5*sin(-30)})},scale=.75] at (2,4) {$z''_3$};
\node[shift={({.4*cos(-135)}, {.4*sin(-135)})},scale=.75] at (2,4) {$z_4$};
\node[shift={({.5*cos(-75)}, {.5*sin(-75)})},scale=.75] at (2,4) {$z'_2$};

% (4,4)
\node[shift={({.4*cos(-45)}, {.4*sin(-45)})},scale=.75] at (4,4) {$z_4$};
\node[shift={({.4*cos(-135)}, {.4*sin(-135)})},scale=.75] at (4,4) {$z'_3$};

% (6,4)
\node[shift={({.5*cos(-25)}, {.5*sin(-25)})},scale=.75] at (6,4) {$z''_1$};
\node[shift={({.75*cos(-75)}, {.75*sin(-75)})},scale=.75] at (6,4) {$z'_4$};
\node[shift={({.75*cos(-105)}, {.75*sin(-105)})},scale=.75] at (6,4) {$z_2$};
\node[shift={({.5*cos(-150)}, {.5*sin(-150)})},scale=.75] at (6,4) {$z''_4$};

% (8,4)
\node[shift={({.4*cos(-135)}, {.4*sin(-135)})},scale=.75] at (8,4) {$z'_1$};

% (2,2)
\node[shift={({.25*cos(0)}, {.25*sin(0)})},scale=.75] at (2,2) {$z''_2$};
\node[shift={({.5*cos(112.5)}, {.5*sin(112.5)})},scale=.75] at (2,2) {$z''_4$};
\node[shift={({.25*cos(180)}, {.25*sin(180)})},scale=.75] at (2,2) {$z''_3$};
\node[shift={({.5*cos(-67.5)}, {.5*sin(-67.5)})},scale=.75] at (2,2) {$z''_1$};

% (6,2)
\node[shift={({.3*cos(22.5)}, {.3*sin(22.5)})},scale=.75] at (6,2) {$z''_4$};
\node[shift={({.3*cos(157.5)}, {.3*sin(157.5)})},scale=.75] at (6,2) {$z''_2$};
\node[shift={({.5*cos(247.5)}, {.5*sin(247.5)})},scale=.75] at (6,2) {$z''_1$};
\node[shift={({.5*cos(-67.5)}, {.5*sin(-67.5)})},scale=.75] at (6,2) {$z''_3$};

%draw in the meridian and longitude we're going to use
\draw[thick, red, ->] (3,-.5) .. controls (3.5,1) and (2.5,3) .. (3,4.5) coordinate[pos=.7] (m);
\node[right,red] at (m) {$\mathfrak{M}_2$};
\draw[thick, red, ->] (-.5,1) .. controls (2,1.5) and (6,.5) .. (8.5,1) coordinate[pos=.15] (l);
\node[above,red] at (l) {$\mathfrak{L}_2$};

\end{tikzpicture}
        \caption{The triangulation of the ``large'' cusp.}
    \end{subfigure}   
    \caption{Triangulations of the boundary tori induced by $\mathcal{T}$.}
    \label{fig:cusp_triangulations}
\end{figure}
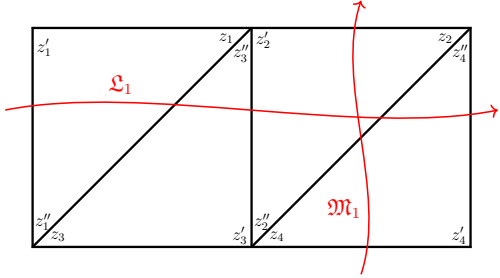
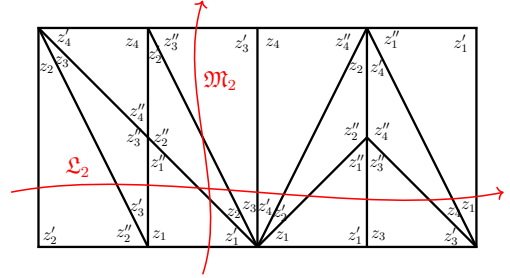

From the triangulations of the cusps, we can find the equations for the meridians and longitudes:
\begin{eqnarray}
\mathfrak{M}_{1}: &&Z_{4}-Z_{2}=0 ,\nonumber\\
\mathfrak{M}_{2}: &&Z''_{3} - Z'_{1} - Z_{2}=0,
\end{eqnarray}
and
\begin{eqnarray}\label{longitudes}
\mathfrak{L}_{1}: && Z''_{3} + Z''_{4} - Z''_{1} - Z''_{2}=0,\nonumber\\
\mathfrak{L}_{2}: && 2Z''_1 - Z'_3 - Z_3 - Z'_4 - Z'_2 + Z''_3 - Z_4 - Z_1 = 0.
 \end{eqnarray}

\subsection{Twist knots through Dehn filling the Whitehead link} \label{sec:dehnFillingParams}

It is a well-known result in low-dimensional topology that the complement of any twist knot can be obtained by performing a specific Dehn filling on the complement of the Whitehead link.
This brief section is dedicated to elucidating the specific parameters of this filling. 

Take $M$ to be the complement of the Whitehead link with the ideal triangulation from Figure~\ref{fig:Whitehead_tetrahedra} with boundary tori in Figure~\ref{fig:cusp_triangulations}.
Then, the complement of the twist knot $\mathcal{K}_p$ is the $(1,p)$ Dehn filling of $M$ along the boundary component corresponding to the crossing circle.
That is, $S^3 \setminus \mathcal{K}_p$ is the $3$-manifold obtained by gluing a solid torus into the Whitehead link complement so that the curve $\mathfrak{M}_1 + p \mathfrak{L}_1$ bounds a disk. 

To obtain the shape parameters that yield the complement of $\mathcal{K}_p$, one solves the usual gluing equations with the completeness equation for the crossing circle replaced by the equation 
\[Z_4 - Z_2 + p (Z''_3+Z''_4-Z''_1-Z''_2) = -2\pi i.\]

%%%%%%%%%%%%%%%%%%%%%%%%%%%%%%%%%%%%%%%%%%%%%%%%
\section{Colored Jones polynomials of twist knots with insertions} \label{sec:colored_jones_polynomials}

This section is dedicated to obtaining the colored Jones polynomial of $\mathcal{K}_p$ with the insertion of an arbitrary number of cables of the knot $K_b$ (see Figure~\ref{fig:twist_knot_curves}), and subsequently computing the asymptotic expansion of this expression as the color approaches infinity. 
Fortunately for us, much of the work has already been done. 
In particular, we will follow \cite{M} to obtain our expression for the colored Jones polynomial, and then \cite{CZ1,CZ2} for its asymptotics. 

Before proceeding, let's establish the standard quantum integer notation that will be used throughout this section. 
Set $a=A^2=q^{1/2}$.
Define the quantum integers,
\begin{equation}
    \{n\} = a^n - a^{-n}, \quad \text{and} \quad [n] = \frac{a^n - a^{-n}}{a - a^{-1}},
\end{equation}
and the corresponding quantum factorials,
\begin{equation}
    \{n\}! = \prod_{i=1}^n \{i\}, \quad \text{and} \quad [n]! = \prod_{i=1}^n [i].
\end{equation}
We also introduce two other sequences:
\begin{align}
    \mu_n &= (-1)^n A^{n^2+2n}, \\
    \lambda_n &= -a^{n+1} - a^{-n-1}.
\end{align}
Finally, let $h_k(x_1, \dots, x_n)$ denote the complete homogeneous symmetric polynomial of degree $k$ in the variables $x_1, \dots, x_n$. 
As an example,
\[h_3(x_1,x_2,x_3) = x_1^3+x_2^3+x_3^3+x_1^2x_2+x_1^2x_3+x_2^2x_1+x_2^2x_3+x_3^2x_1+x_3^2x_2+x_1x_2x_3.\]

\subsection{Computation of the colored Jones polynomial}
We begin by introducing a formal skein-theoretical twist operator $\omega$, which acts on an even number of strands by inducing a full positive twist: 
\begin{gather*}
    \vcenter{\hbox{
    \includegraphics[scale=.7]{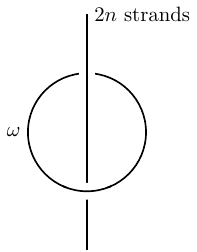}
    }}
    \;\;=\;\;
    \vcenter{\hbox{
    \includegraphics[scale=.7]{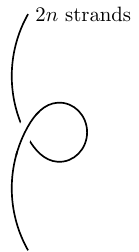}
    }}
\end{gather*}
Correspondingly, for any $p \in \mathbb{Z}$, the operator $\omega^p$ applies $p$ full twists (positive or negative, depending on the sign of $p$).
$\omega$ is not a standard, finite skein element, but rather a formal power series of skein elements.

Working within the skein algebra of the solid torus $\mathrm{Sk}(S^1 \times D^2) \cong \mathbb{Z}[A^{\pm 1}][z]$, we can define a convenient basis $\{R_n\}_{n=0}^\infty$:
\begin{equation}
R_n := \prod_{i=0}^{n-1} (z - \lambda_{2i}), \text{ with } R_0 := 1.
\end{equation}
The twist operator $\omega^p$ can be expanded as an infinite linear combination of these basis elements.

\begin{thm}[\cite{M}]
The twist operator can be expanded in the basis $R_n$ as $\omega^p = \sum_{n=0}^\infty c_{n,p} R_n$, where the coefficients are given by:
\begin{equation}
c_{n,p} = \frac{1}{(a - a^{-1})^{2n}} \sum_{k=0}^n (-1)^k \mu_{2k}^{p} \frac{[2k + 1]}{[n + k + 1]![n - k]!}.
\end{equation}
\end{thm}

Our primary goal in this subsection is to compute the normalized $n$-th colored Jones polynomial 
\[J'_{n,2}(\mathcal{K}_p \cup K_b^l;q) := \frac{J_{n,2}(\mathcal{K}_p \cup K_b^l;q)}{[n]},\]
for $0 \le l \le 2|p|$. 
The key observation is that the insertion of $l$ cables of $K_b$ into the complement of $\mathcal{K}_p$ corresponds precisely to multiplying the formal twist operator $\omega^p$ by $z^l$; see Figure~\ref{fig:mult_by_z}.

\begin{figure}[htbp]
    \centering
    \begin{tabular}{cccccc}
        \includegraphics[scale=.75,valign=c]{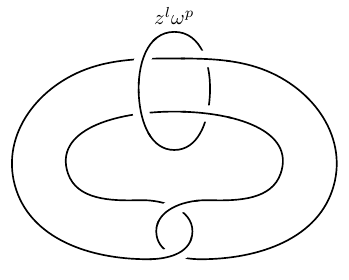}&
        $=$&
        \includegraphics[scale=.75,valign=c]{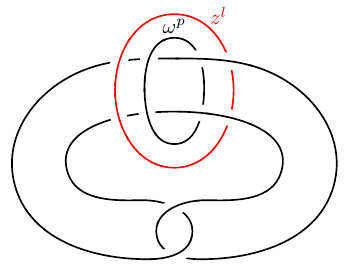}&
        $=$&
        \includegraphics[scale=.75,valign=c]{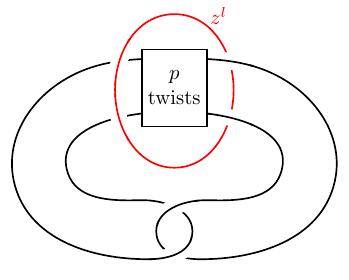}&       
    \end{tabular}
  \caption{Multiplying the operator $\omega^p$ by $z^l$ inserts $l$ cables of $K_b$ into the complement of $\mathcal{K}_p$.}
  \label{fig:mult_by_z}
\end{figure}

With this observation in hand, our task reduces to finding a formula for $z^l \omega^p$ as a linear combination of the basis elements $R_n$. 
Once we obtain this expansion, we can follow Masbaum's original argument exactly. 
We first compute the action of $z^l$ on the basis $R_n$, using a simple induction argument.

\begin{lem} \label{lem:zl_Rn}
    For any integer $l \ge 0$,
    \begin{equation}
        z^l R_n = \sum_{i=0}^l h_{l-i}(\lambda_{2n}, \lambda_{2n+2}, \dots, \lambda_{2n+2i}) R_{n+i}.
    \end{equation}
\end{lem}

It's useful to define a new set of coefficients:
\begin{equation}
    d_{i,p,l} := \sum_{j=0}^l c_{i-j, p} \, h_{l-j}(\lambda_{2i-2j}, \dots, \lambda_{2i}).
\end{equation}

The next lemma then follows immediately.
\begin{lem}
The operator $z^l \omega^p$ expands in the $R_n$ basis as:
\begin{equation}
z^l \omega^p = \sum_{i=0}^\infty d_{i,p,l} R_i.
\end{equation}
\end{lem}

From here, it is straightforward to obtain the desired colored Jones polynomial. 
We omit the proof here, as it follows Masbaum's derivation for the standard twist knot virtually line by line, simply replacing the coefficients $c_{n,p}$ with $d_{n,p,l}$.

\begin{thm} \label{thm:jones_insertion}
The normalized colored Jones polynomial of $\mathcal{K}_p$ with $l$ insertions of $K_b$ is given by:
\begin{equation}
J'_{n,2}(\mathcal{K}_p \cup K_b^l) = \sum_{i=0}^{n-1} d_{i,p,l} a^{\frac{i(i+3)}{2}} \frac{\{n + i\} \dots \{n - i\}}{\{n\}}\{i\}!.
\end{equation}
\end{thm}

\begin{eg}
Let us use Theorem \ref{thm:jones_insertion} to compute the colored Jones polynomial with insertions for the figure-eight knot, which corresponds to the twist knot $\mathcal{K}_{-1}$ (i.e., $p = -1$).
Define
\begin{equation}
w_{n,i} := \frac{\{n + i\} \dots \{n - i\}}{\{n\}}.
\end{equation}
When $p = -1$, the usual coefficients simplify to:
\begin{equation}
c_{i,-1} = \frac{a^{-\frac{i(i+3)}{2}}}{\{i\}!}.
\end{equation}
We can now compute the modified coefficients $d_{i,-1,l}$ for small numbers of insertions.
For a single insertion ($l = 1$):
\begin{equation}
        d_{i,-1,1} = c_{i,-1}\lambda_{2i} + c_{i-1,-1} = \frac{a^{\frac{-i(i+3)}{2}}}{\{i\}!}\lambda_{2i} + \frac{a^{\frac{-(i-1)(i+2)}{2}}}{\{i-1\}!}.
\end{equation}
Substituting this into the formula from Theorem \ref{thm:jones_insertion}:
\begin{align}
    J'_{n,2}(4_1 \cup K_b) &= \sum_{i=0}^{n-1} \left( \frac{a^{\frac{-i(i+3)}{2}}}{\{i\}!}\lambda_{2i} + \frac{a^{\frac{-(i-1)(i+2)}{2}}}{\{i-1\}!} \right) a^{\frac{i(i+3)}{2}} \{i\}! \, w_{n,i} \nonumber \\
    &= \sum_{i=0}^{n-1} \left( \lambda_{2i} + \{i\}a^{i+1} \right) w_{n,i} \nonumber \\
    &= -\sum_{i=0}^{n-1} \left( a^{-2i-1} + a \right) w_{n,i}.
\end{align}
Similarly, for two insertions ($l=2$), the modified coefficient becomes:
\begin{equation}
    d_{i,-1,2} = \frac{a^{\frac{-i(i+3)}{2}}}{\{i\}!}\lambda_{2i}^2 + \frac{a^{\frac{-(i-1)(i+2)}{2}}}{\{i-1\}!}(\lambda_{2i} + \lambda_{2i-2}) + \frac{a^{\frac{-(i-2)(i+1)}{2}}}{\{i-2\}!}.
\end{equation}
Substituting this into the main formula yields:
\begin{align}
    J'_{n,2}(4_1 \cup K_b^2) &= \sum_{i=0}^{n-1} \left( \lambda_{2i}^2 + \{i\}(\lambda_{2i} + \lambda_{2i-2})a^{i+1} + \{i\}\{i-1\}a^{2i+1} \right) w_{n,i} \nonumber \\
    &= \sum_{i=0}^{n-1} \left( 1 + a^{-4i-2} + a^{-2i} + a^{-2i+2} \right) w_{n,i}.
\end{align} 

After adjusting for minor differences in framing and variable conventions, these expressions match the conjectural formulas for the colored Jones polynomials with insertions proposed in \cite{AGLR}.
\end{eg}

\subsection{Asymptotics}

This subsection is dedicated to computing the asymptotic expansion of the colored Jones polynomial of the twist knot $\mathcal{K}_p$ with the insertion of $K_b^l$. 

We begin with the explicit formula for the colored Jones polynomial with an insertion, modified with the replacement $q=a^2$ and with the square bracket quantum integers replaced by their curly cousins. 
Then,
\begin{align}
    J_{n,2}'(\mathcal{K}_p \cup K_b^l;q) = \sum_{i=0}^{n-1} \sum_{j=0}^{l} \sum_{k=0}^{i-j} &\ h_{l-j}(\lambda_{2i}, \dots, \lambda_{2i-2j}) \nonumber \\
    &\times \frac{(-1)^k \{2k+1\} \{i\}! q^{\frac{i(i+3)}{4} + p k(k+1)}}{\{i-j+k+1\}! \{i-j-k\}!} \prod_{m=1}^{i} \{n+m\}\{n-m\}.
\end{align}
Multiply both the numerator and the denominator by $\{i+k+1\}!\{i-k\}!$ and extend the upper bound of the $k$-summation from $i-j$ to $i$ (the newly added terms where $k > i-j$ evaluate to zero), obtaining
\begin{align}
    J_{n,2}'(\mathcal{K}_p \cup K_b^l;q) = \sum_{i=0}^{n-1} \sum_{j=0}^{l} \sum_{k=0}^{i} &\ h_{l-j}(\lambda_{2i}, \dots, \lambda_{2i-2j}) \prod_{m=0}^{j-1}\left(\{i+k+1-m\}\{i-k-m\}\right) \nonumber \\
    &\times \frac{(-1)^k \{2k+1\} \{i\}! q^{\frac{i(i+3)}{4} + p k(k+1)}}{\{i+k+1\}! \{i-k\}!} \prod_{m=1}^{i} \{n+m\}\{n-m\}.
\end{align}
It is straightforward to prove:
\begin{lem}
\[\sum_{j=0}^{l}h_{l-j}(\lambda_{2i},\ldots,\lambda_{2i-2j})\prod_{m=0}^{j-1}\left(\{i+k+1-m\}\{i-k-m\}\right) = \lambda_{2k}^l\]
\end{lem}
Thus, we obtain 
\begin{equation}
    J_{n,2}'(\mathcal{K}_p \cup K_b^l;q) = \sum_{i=0}^{n-1} \sum_{k=0}^{i} \frac{(-1)^k \lambda_{2k}^l \{2k+1\} \{i\}! q^{\frac{i(i+3)}{4} + p k(k+1)}}{\{i+k+1\}! \{i-k\}!} \prod_{m=1}^{i} \{n+m\}\{n-m\}.
\end{equation}

Crucially, except for the addition of the insertion factor $\lambda_{2k}^l$, this expression is structurally identical to the colored Jones polynomial of the twist knot $\mathcal{K}_p$. Thus, we can directly follow the asymptotic analysis established by Chen and Zhu in \cite{CZ1,CZ2}.

We evaluate the polynomial at the root of unity $q = e^{\frac{2\pi i}{n+\frac 1 m}}$ and take the limit as $n \to \infty$. 
Following their methodology, we rewrite the $q$-factorials in terms of the quantum dilogarithm function $\varphi_n(x)$ and apply the Poisson summation formula to convert the discrete sums into a double integral over the continuous variables $t = \frac{i+1/2}{n+\frac 1 m}$ and $s = \frac{k+1/2}{n + \frac 1 m}$. 
Under this substitution, the new insertion term can be written as:
\begin{equation}
    \lambda_{2k}^l \rightarrow \left(-e^{\pi i (2s)} - e^{-\pi i (2s)} \right)^l = \left(-2\cos(2\pi s)\right)^l.
\end{equation}

By taking $m \rightarrow \infty$ and mimicking the argument of \cite{CZ2}, we obtain:
\begin{thm}\label{thmCZ}
For a twist knot $\mathcal{K}_p$ with $l$ insertions of $K_b$, the colored Jones polynomial evaluated at the root of unity $q = e^{\frac{2\pi i}{n}}$ admits the following asymptotic integral representation as $n \to \infty$:
\begin{equation}\label{integralCZ}
    J'_{n,2} (\mathcal{K}_p \cup K_b^l;q) \sim \frac{2(-1)^p e^{\frac{\pi i}{4}} n^{\frac{3}{2}}} {\sin\left(\frac{\pi}{2n}\right)} \iint (-2 \cos(2\pi s))^l \sin(2\pi s) e^{n V_n (p,t,s)} dt \, ds,
\end{equation}
where the accompanying potential function $V_n(p, t, s)$ is given by
\begin{equation}\label{VnCZ}
\begin{aligned}
    V_n(p,t,s) &:= \pi i\left((2p+1)s^2 - (2p+3)s + \left(\frac{2}{n}-2\right)t - \frac{3p+2}{6n^2}\right) \\
    &\qquad + \frac{1}{n}\varphi\left(t+s+\frac{1}{2n}-1\right) + \frac{1}{n}\varphi\left(t-s+\frac{1}{2n}\right) \\
    &\qquad - \frac{3}{n}\varphi(t) - \frac{\pi i}{12}.
\end{aligned}
\end{equation}
$V_n(p, t, s)$, as $n \rightarrow \infty$, can be expanded as
\begin{equation}\label{VNtruncated}
\begin{aligned}
    V_n(p, t, s) &\sim \pi i \big((2p + 1)s^2 - (2p + 3)s - 2t\big) \\
    &\quad + \frac{1}{2\pi i} \left( \mathrm{Li}_2(e^{2\pi i(t+s)}) + \mathrm{Li}_2(e^{2\pi i(t-s)}) - 3\mathrm{Li}_2(e^{2\pi i t}) + \frac{\pi^2}{6} \right) \\
    &\quad - \frac{1}{2n} \left(\log( 1 - e^{2\pi i(t+s)}) + \log(1 - e^{2\pi i(t-s)}) - 4 \pi i t\right) + \mathcal{O}\left(\frac{1}{n^2}\right).
\end{aligned}
\end{equation}
\end{thm}

%%%%%%%%%%%%%%%%%%%%%%%%%%%%%%%%%%%%%%%%%%%%%%%%
\section{The state integral for twist knots} \label{sec:state_integrals}

In this section we will derive an explicit formula for the state integral for twist knots, given an ideal triangulation $\mathcal{T}$, with insertions of $\Tr_{\mathcal{T}}(K_b^l)$.

\subsection{The quantum trace of $K_b$ and its cables}
\begin{figure}[htbp]
    \centering
    \begin{gather*}
    \vcenter{\hbox{
    \tdplotsetmaincoords{45}{60}
    \begin{tikzpicture}[tdplot_main_coords]  
    \begin{scope}[tdplot_main_coords]
    \coordinate (N) at (0, 0, 1);
    \coordinate (S) at (0, 0, -1);  
    \coordinate (b) at ({2*sqrt(2)*cos(0)}, {2*sqrt(2)*sin(0)}, 0);
    \coordinate (c) at ({2*sqrt(2)*cos(120)}, {2*sqrt(2)*sin(120)}, 0);
    \coordinate (d) at ({2*sqrt(2)*cos(240)}, {2*sqrt(2)*sin(240)}, 0);   
    \coordinate (bc) at ($1/2*(b) + 1/2*(c)$);
    \coordinate (bd) at ($1/2*(b) + 1/2*(d)$);
    \coordinate (cd) at ($1/2*(c) + 1/2*(d)$);
    \coordinate (kb1) at ($(bc)!.6!(N)$);
    \coordinate (kb4) at ($(bd)!.6!(S)$);   
    \begin{scope}[thick,decoration={
    markings,
    mark=at position 0.5 with {\arrow{>}}}
    ] 
    %kb and the marking it is "above"
    \draw[postaction={decorate}, dashed, orange] (bc) -- (S);
    \draw[blue] (kb1) .. controls ({1/2*cos(0)},{1/2*sin(0)},0) .. (kb4);
    %markings
    \draw[postaction={decorate}, orange] (bd) -- (N);
    \draw[postaction={decorate}, orange] (bc) -- (N);
    \draw[postaction={decorate}, orange] (cd) -- (N);
    \draw[postaction={decorate}, dashed, orange] (bd) -- (S);
    \draw[postaction={decorate}, dashed, orange] (cd) -- (S);
    \end{scope}
    %edges
    \begin{scope}[decoration={
    markings,
    mark=between positions 0.48 and .52 step .04 with {\draw(0,-3pt)--(0,3pt);}},very thick]
    \draw[postaction={decorate}] (d)--(b);
    \end{scope}
    \begin{scope}[decoration={
    markings,
    mark=between positions 0.46 and .54 step .04 with {\draw(0,-3pt)--(0,3pt);}},very thick]
    \draw[postaction={decorate}] (b)--(c);
    \end{scope}
    \begin{scope}[decoration={
    markings,
    mark=at position 0.6 with \arrow{stealth}},very thick]
    \draw[postaction={decorate}] (c)--(d);
    \end{scope}    
    \filldraw[orange] (N) circle (0.05em);
    \filldraw[orange] (S) circle (0.05em);
    \filldraw (b) circle (0.05em);
    \filldraw (c) circle (0.05em);
    \filldraw (d) circle (0.05em);
    %\node[anchor=north east] at (bd) {$a$};
    \node[below] at (bd) {$z''_3$};
    \node[above left] at (bc) {$z''_1$};
    %\node[anchor=south] at (c) {$\alpha$};
    %\node[anchor=east] at (d) {$\beta$};
    %\node[anchor=north] at (b) {$\gamma$};
    \node[above,scale=.75] at (kb1) {$\epsilon_1$};
    \node[above left,scale=.75] at (kb4) {$\epsilon_2$};    
    \end{scope}
    \end{tikzpicture}
    }}
    \hspace{1.5cm}
    \vcenter{\hbox{
    \tdplotsetmaincoords{45}{60}
    \begin{tikzpicture}[tdplot_main_coords]  
    \begin{scope}[tdplot_main_coords]
    \coordinate (N) at (0, 0, 1);
    \coordinate (S) at (0, 0, -1);  
    \coordinate (b) at ({2*sqrt(2)*cos(0)}, {2*sqrt(2)*sin(0)}, 0);
    \coordinate (c) at ({2*sqrt(2)*cos(120)}, {2*sqrt(2)*sin(120)}, 0);
    \coordinate (d) at ({2*sqrt(2)*cos(240)}, {2*sqrt(2)*sin(240)}, 0);   
    \coordinate (bc) at ($1/2*(b) + 1/2*(c)$);
    \coordinate (bd) at ($1/2*(b) + 1/2*(d)$);
    \coordinate (cd) at ($1/2*(c) + 1/2*(d)$);
    \coordinate (kb1) at ($(bd)!.6!(N)$);
    \coordinate (kb2) at ($(bc)!.6!(N)$);
    \coordinate (kb3) at ($(bc)!.4!(N)$);
    \coordinate (kb4) at ($(bc)!.4!(S)$);   
    \begin{scope}[thick,decoration={
    markings,
    mark=at position 0.5 with {\arrow{>}}}
    ] 
    %kb and the marking it is "above"
    \draw[postaction={decorate}, dashed, orange] (bd) -- (S);
    \draw[blue] (kb1) .. controls ({1/2*cos(0)},{1/2*sin(0)},0) .. (kb4);
    %markings
    \draw[postaction={decorate}, orange] (bd) -- (N);
    \draw[postaction={decorate}, orange] (bc) -- (N);
    \draw[postaction={decorate}, orange] (cd) -- (N);
    \draw[postaction={decorate}, dashed, orange] (bc) -- (S);
    \draw[postaction={decorate}, dashed, orange] (cd) -- (S);
    \end{scope}
    %edges
    \begin{scope}[decoration={
    markings,
    mark=between positions 0.48 and .52 step .04 with {\draw(0,-3pt)--(0,3pt);}},very thick]
    \draw[postaction={decorate}] (b)--(c);
    \end{scope}
    \begin{scope}[decoration={
    markings,
    mark=between positions 0.48 and .52 step .04 with {\draw(0,-3pt)--(0,3pt);}},very thick]
    \draw[postaction={decorate}] (d)--(b);
    \end{scope}
    \begin{scope}[decoration={
    markings,
    mark=at position 0.6 with \arrow{stealth}},very thick]
    \draw[postaction={decorate}] (c)--(d);
    \end{scope}    
    \filldraw[orange] (N) circle (0.05em);
    \filldraw[orange] (S) circle (0.05em);
    \filldraw (b) circle (0.05em);
    \filldraw (c) circle (0.05em);
    \filldraw (d) circle (0.05em);
    %\node[anchor=north east] at (bd) {$a$};
    \node[above] at (bd) {$z''_3$};
    \node[below right] at (bc) {$z''_2$};
    %\node[anchor=south] at (c) {$\alpha$};
    %\node[anchor=east] at (d) {$\beta$};
    %\node[anchor=north] at (b) {$\gamma$};
    \node[right,scale=.75] at (kb1) {$\epsilon_2$};
    \node[below right,scale=.75] at (kb4) {$\epsilon_3$};    
    \end{scope}
    \end{tikzpicture}
    }}\\ 
    \vspace{1cm}\\
    \vcenter{\hbox{
    \tdplotsetmaincoords{45}{60}
    \begin{tikzpicture}[tdplot_main_coords]  
    \begin{scope}[tdplot_main_coords]
    \coordinate (N) at (0, 0, 1);
    \coordinate (S) at (0, 0, -1);  
    \coordinate (b) at ({2*sqrt(2)*cos(0)}, {2*sqrt(2)*sin(0)}, 0);
    \coordinate (c) at ({2*sqrt(2)*cos(120)}, {2*sqrt(2)*sin(120)}, 0);
    \coordinate (d) at ({2*sqrt(2)*cos(240)}, {2*sqrt(2)*sin(240)}, 0);   
    \coordinate (bc) at ($1/2*(b) + 1/2*(c)$);
    \coordinate (bd) at ($1/2*(b) + 1/2*(d)$);
    \coordinate (cd) at ($1/2*(c) + 1/2*(d)$);
    \coordinate (kb1) at ($(bc)!.6!(N)$);
    \coordinate (kb4) at ($(bd)!.6!(S)$);   
    \begin{scope}[thick,decoration={
    markings,
    mark=at position 0.5 with {\arrow{>}}}
    ] 
    %kb and the marking it is "above"
    \draw[postaction={decorate}, dashed, orange] (bc) -- (S);
    \draw[blue] (kb1) .. controls ({1/2*cos(0)},{1/2*sin(0)},0) .. (kb4);
    %markings
    \draw[postaction={decorate}, orange] (bd) -- (N);
    \draw[postaction={decorate}, orange] (bc) -- (N);
    \draw[postaction={decorate}, orange] (cd) -- (N);
    \draw[postaction={decorate}, dashed, orange] (bd) -- (S);
    \draw[postaction={decorate}, dashed, orange] (cd) -- (S);
    \end{scope}
    %edges
    \begin{scope}[decoration={
    markings,
    mark=between positions 0.46 and .54 step .04 with {\draw(0,-3pt)--(0,3pt);}},very thick]
    \draw[postaction={decorate}] (d)--(b);
    \end{scope}
    \begin{scope}[decoration={
    markings,
    mark=between positions 0.48 and .52 step .04 with {\draw(0,-3pt)--(0,3pt);}},very thick]
    \draw[postaction={decorate}] (b)--(c);
    \end{scope}
    \begin{scope}[decoration={
    markings,
    mark=at position 0.6 with \arrow{stealth}},very thick]
    \draw[postaction={decorate}] (c)--(d);
    \end{scope}    
    \filldraw[orange] (N) circle (0.05em);
    \filldraw[orange] (S) circle (0.05em);
    \filldraw (b) circle (0.05em);
    \filldraw (c) circle (0.05em);
    \filldraw (d) circle (0.05em);
    \node[below] at (bd) {$z''_4$};
    \node[above left] at (bc) {$z''_2$};
    %\node[anchor=south] at (cd) {$c$};
    %\node[anchor=south] at (c) {$\alpha$};
    %\node[anchor=east] at (d) {$\beta$};
    %\node[anchor=north] at (b) {$\gamma$};
    \node[above,scale=.75] at (kb1) {$\epsilon_3$};
    \node[above left,scale=.75] at (kb4) {$\epsilon_4$};    
    \end{scope}
    \end{tikzpicture}
    }}
    \hspace{1.5cm}
    \vcenter{\hbox{
    \tdplotsetmaincoords{45}{60}
    \begin{tikzpicture}[tdplot_main_coords]  
    \begin{scope}[tdplot_main_coords]
    \coordinate (N) at (0, 0, 1);
    \coordinate (S) at (0, 0, -1);  
    \coordinate (b) at ({2*sqrt(2)*cos(0)}, {2*sqrt(2)*sin(0)}, 0);
    \coordinate (c) at ({2*sqrt(2)*cos(120)}, {2*sqrt(2)*sin(120)}, 0);
    \coordinate (d) at ({2*sqrt(2)*cos(240)}, {2*sqrt(2)*sin(240)}, 0);   
    \coordinate (bc) at ($1/2*(b) + 1/2*(c)$);
    \coordinate (bd) at ($1/2*(b) + 1/2*(d)$);
    \coordinate (cd) at ($1/2*(c) + 1/2*(d)$);
    \coordinate (kb1) at ($(bd)!.6!(N)$);
    \coordinate (kb2) at ($(bc)!.6!(N)$);
    \coordinate (kb3) at ($(bc)!.4!(N)$);
    \coordinate (kb4) at ($(bc)!.4!(S)$);   
    \begin{scope}[thick,decoration={
    markings,
    mark=at position 0.5 with {\arrow{>}}}
    ] 
    %kb and the marking it is "above"
    \draw[postaction={decorate}, dashed, orange] (bd) -- (S);
    \draw[blue] (kb1) .. controls ({1/2*cos(0)},{1/2*sin(0)},0) .. (kb4);
    %markings
    \draw[postaction={decorate}, orange] (bd) -- (N);
    \draw[postaction={decorate}, orange] (bc) -- (N);
    \draw[postaction={decorate}, orange] (cd) -- (N);
    \draw[postaction={decorate}, dashed, orange] (bc) -- (S);
    \draw[postaction={decorate}, dashed, orange] (cd) -- (S);
    \end{scope}
    %edges
    \begin{scope}[decoration={
    markings,
    mark=between positions 0.46 and .54 step .04 with {\draw(0,-3pt)--(0,3pt);}},very thick]
    \draw[postaction={decorate}] (b)--(c);
    \end{scope}
    \begin{scope}[decoration={
    markings,
    mark=between positions 0.46 and .54 step .04 with {\draw(0,-3pt)--(0,3pt);}},very thick]
    \draw[postaction={decorate}] (d)--(b);
    \end{scope}
    \begin{scope}[decoration={
    markings,
    mark=at position 0.6 with \arrow{stealth}},very thick]
    \draw[postaction={decorate}] (c)--(d);
    \end{scope}    
    \filldraw[orange] (N) circle (0.05em);
    \filldraw[orange] (S) circle (0.05em);
    \filldraw (b) circle (0.05em);
    \filldraw (c) circle (0.05em);
    \filldraw (d) circle (0.05em);
    \node[above] at (bd) {$z''_4$};
    \node[below right] at (bc) {$z''_1$};
    %\node[anchor=south] at (cd) {$c$};
    %\node[anchor=south] at (c) {$\alpha$};
    %\node[anchor=east] at (d) {$\beta$};
    %\node[anchor=north] at (b) {$\gamma$};
    \node[right,scale=.75] at (kb1) {$\epsilon_4$};
    \node[below right,scale=.75] at (kb4) {$\epsilon_1$};    
    \end{scope}
    \end{tikzpicture}
    }}
    \end{gather*}
    \caption{$K_b$'s image after splitting into face suspensions.}
    \label{fig:kbInFaceSuspensions}
\end{figure}
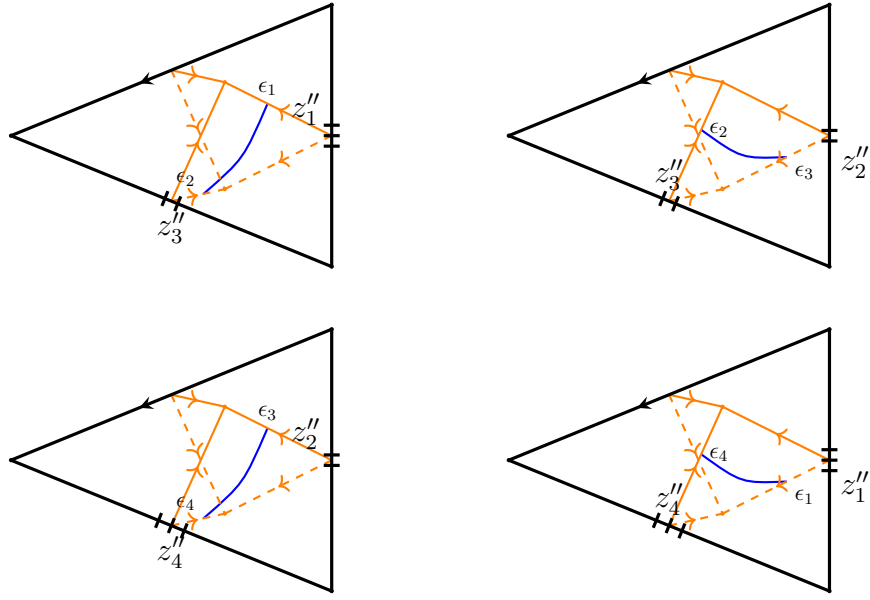

Figure~\ref{fig:kbInFaceSuspensions} shows a compatible state obtained after splitting $K_b$ into face suspensions; that is, the result of applying the splitting map $\overline{\sigma}$ of Theorem~\ref{thm:splitting_map}.
The image of $K_b$ under the splitting map is the sum over all such compatible states.

To compute the quantum trace, we collide the tangle in each face suspension with a boundary marking and apply the stated skein relations to factor it into an arc in the triangle algebra and an arc in the biangle algebra.
Figure~\ref{fig:kb_collision} illustrates this collision procedure for the top-left face suspension of Figure~\ref{fig:kbInFaceSuspensions}.

\begin{figure}[htbp]
    \centering
    \begin{gather*}
    \vcenter{\hbox{
    \tdplotsetmaincoords{45}{60}
    \begin{tikzpicture}[tdplot_main_coords]  
    \begin{scope}[tdplot_main_coords]
    \coordinate (N) at (0, 0, 1);
    \coordinate (S) at (0, 0, -1);  
    \coordinate (b) at ({2*sqrt(2)*cos(0)}, {2*sqrt(2)*sin(0)}, 0);
    \coordinate (c) at ({2*sqrt(2)*cos(120)}, {2*sqrt(2)*sin(120)}, 0);
    \coordinate (d) at ({2*sqrt(2)*cos(240)}, {2*sqrt(2)*sin(240)}, 0);   
    \coordinate (bc) at ($1/2*(b) + 1/2*(c)$);
    \coordinate (bd) at ($1/2*(b) + 1/2*(d)$);
    \coordinate (cd) at ($1/2*(c) + 1/2*(d)$);
    \coordinate (kb1) at ($(bc)!.6!(N)$);
    \coordinate (kb4) at ($(bd)!.6!(S)$);   
    \begin{scope}[thick,decoration={
    markings,
    mark=at position 0.5 with {\arrow{>}}}
    ] 
    %kb and the marking it is "above"
    \draw[postaction={decorate}, dashed, orange] (bc) -- (S);
    \draw[blue] (kb1) .. controls ({1/2*cos(0)},{1/2*sin(0)},0) .. (kb4);
    %markings
    \draw[postaction={decorate}, orange] (bd) -- (N);
    \draw[postaction={decorate}, orange] (bc) -- (N);
    \draw[postaction={decorate}, orange] (cd) -- (N);
    \draw[postaction={decorate}, dashed, orange] (bd) -- (S);
    \draw[postaction={decorate}, dashed, orange] (cd) -- (S);
    \end{scope}
    %edges
    \begin{scope}[decoration={
    markings,
    mark=between positions 0.48 and .52 step .04 with {\draw(0,-3pt)--(0,3pt);}},very thick]
    \draw[postaction={decorate}] (d)--(b);
    \end{scope}
    \begin{scope}[decoration={
    markings,
    mark=between positions 0.46 and .54 step .04 with {\draw(0,-3pt)--(0,3pt);}},very thick]
    \draw[postaction={decorate}] (b)--(c);
    \end{scope}
    \begin{scope}[decoration={
    markings,
    mark=at position 0.6 with \arrow{stealth}},very thick]
    \draw[postaction={decorate}] (c)--(d);
    \end{scope}    
    \filldraw[orange] (N) circle (0.05em);
    \filldraw[orange] (S) circle (0.05em);
    \filldraw (b) circle (0.05em);
    \filldraw (c) circle (0.05em);
    \filldraw (d) circle (0.05em);
    %\node[anchor=north east] at (bd) {$a$};
    \node[below] at (bd) {$z''_3$};
    \node[above left] at (bc) {$z''_1$};
    %\node[anchor=south] at (c) {$\alpha$};
    %\node[anchor=east] at (d) {$\beta$};
    %\node[anchor=north] at (b) {$\gamma$};
    \node[above,scale=.75] at (kb1) {$\epsilon_1$};
    \node[above left,scale=.75] at (kb4) {$\epsilon_2$};    
    \end{scope}
    \end{tikzpicture}
    }}
    \;\;=\;\;
    \sum_{\mu} (-A^2)^{\frac \mu 2} \hspace{-3mm}
    \vcenter{\hbox{
    \tdplotsetmaincoords{45}{60}
    \begin{tikzpicture}[tdplot_main_coords]  
    \begin{scope}[tdplot_main_coords]
    \coordinate (N) at (0, 0, 1);
    \coordinate (S) at (0, 0, -1);  
    \coordinate (b) at ({2*sqrt(2)*cos(0)}, {2*sqrt(2)*sin(0)}, 0);
    \coordinate (c) at ({2*sqrt(2)*cos(120)}, {2*sqrt(2)*sin(120)}, 0);
    \coordinate (d) at ({2*sqrt(2)*cos(240)}, {2*sqrt(2)*sin(240)}, 0);   
    \coordinate (bc) at ($1/2*(b) + 1/2*(c)$);
    \coordinate (bd) at ($1/2*(b) + 1/2*(d)$);
    \coordinate (cd) at ($1/2*(c) + 1/2*(d)$);
    \coordinate (kb1) at ($(bc)!.6!(N)$);
    \coordinate (kb4) at ($(bd)!.6!(S)$);   
    \coordinate (kb2) at ($(bc)!.4!(S)$);
    \coordinate (kb3) at ($(bc)!.6!(S)$);   
    \begin{scope}[thick,decoration={
    markings,
    mark=at position 0.5 with {\arrow{>}}}
    ] 
    %kb and the marking it is "above"
    \draw[postaction={decorate}, dashed, orange] (bc) -- (S);
    \draw[blue] (kb1) .. controls ({1/2*cos(0)},{1/2*sin(0)},0) .. (kb2);
    \draw[blue] (kb3) .. controls ({1/2*cos(0)},{1/2*sin(0)},0) .. (kb4);
    %markings
    \draw[postaction={decorate}, orange] (bd) -- (N);
    \draw[postaction={decorate}, orange] (bc) -- (N);
    \draw[postaction={decorate}, orange] (cd) -- (N);
    \draw[postaction={decorate}, dashed, orange] (bd) -- (S);
    \draw[postaction={decorate}, dashed, orange] (cd) -- (S);
    \end{scope}
    %edges
    \begin{scope}[decoration={
    markings,
    mark=between positions 0.48 and .52 step .04 with {\draw(0,-3pt)--(0,3pt);}},very thick]
    \draw[postaction={decorate}] (d)--(b);
    \end{scope}
    \begin{scope}[decoration={
    markings,
    mark=between positions 0.46 and .54 step .04 with {\draw(0,-3pt)--(0,3pt);}},very thick]
    \draw[postaction={decorate}] (b)--(c);
    \end{scope}
    \begin{scope}[decoration={
    markings,
    mark=at position 0.6 with \arrow{stealth}},very thick]
    \draw[postaction={decorate}] (c)--(d);
    \end{scope}    
    \filldraw[orange] (N) circle (0.05em);
    \filldraw[orange] (S) circle (0.05em);
    \filldraw (b) circle (0.05em);
    \filldraw (c) circle (0.05em);
    \filldraw (d) circle (0.05em);
    %\node[anchor=north east] at (bd) {$a$};
    \node[below] at (bd) {$z''_3$};
    \node[above left] at (bc) {$z''_1$};
    %\node[anchor=south] at (c) {$\alpha$};
    %\node[anchor=east] at (d) {$\beta$};
    %\node[anchor=north] at (b) {$\gamma$};
    \node[above,scale=.75] at (kb1) {$\epsilon_1$};
    \node[above left,scale=.75] at (kb4) {$\epsilon_2$};   
    \node[right,scale=.75] at (kb2) {$-\mu$};  
    \node[below,scale=.75] at (kb3) {$\mu$}; 
    \end{scope}
    \end{tikzpicture}
    }}
    \end{gather*}
    \caption{Colliding a tangle with a boundary marking.}
    \label{fig:kb_collision}
\end{figure}
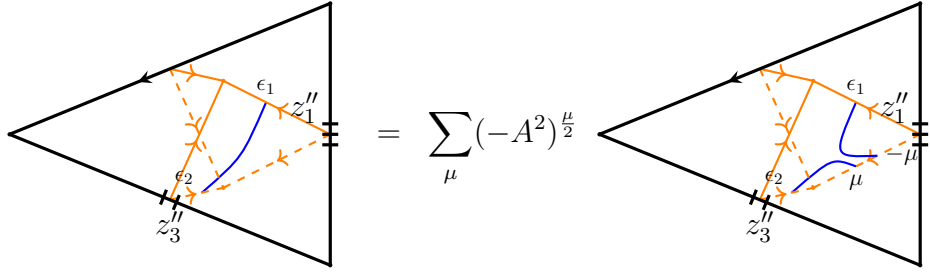

The bad-arc relations of Figures~\ref{fig:triangle_bad_arcs} and \ref{fig:biangle_bad_arcs} reduce the sum over states to just two nonzero contributions:
\begin{itemize}
    \item $(\epsilon_1, \epsilon_2, \epsilon_3, \epsilon_4) = (+, -, +, -)$, with quantum trace $\dfrac{\hat{z}''_1\,\hat{z}''_2}{\hat{z}''_3\,\hat{z}''_4}$;
    \item $(\epsilon_1, \epsilon_2, \epsilon_3, \epsilon_4) = (-, +, -, +)$, with quantum trace $\dfrac{\hat{z}''_3\,\hat{z}''_4}{\hat{z}''_1\,\hat{z}''_2}$.
\end{itemize}
Applying the quantized gluing relation \eqref{item:gluing_rel}
\[
    \hat{z}''_1 \hat{z}''_2 \hat{z}''_3 \hat{z}''_4 \;=\; -A^{2},
\]
we obtain
\begin{equation}\label{eq:trace-Kb-single}
    \Tr_{\mathcal{T}}(K_b) \;=\; -A^{-2}(\hat{z}''_1)^{2}(\hat{z}''_2)^{2} - A^{2}(\hat{z}''_1)^{-2}\,(\hat{z}''_2)^{-2}.
\end{equation}

We now extend the calculation to the cables of $K_b$.
The cable $K_b^l$ consists of $l$ framed parallel copies of $K_b$.
After splitting into face suspensions, each of the four face suspensions visited by $K_b$ contains $l$ parallel strands running between the same pair of edge cones as in the single-copy case.

A compatible state of $K_b^l$ assigns a sign to each of the $4l$ boundary points produced by splitting.
The bad-arc relations of Figures~\ref{fig:triangle_bad_arcs} and \ref{fig:biangle_bad_arcs} again force the two endpoints of each strand to carry opposite signs, so fixing the signs at the top boundary marking of the top-left face suspension of Figure~\ref{fig:kbInFaceSuspensions} determines the state at every other endpoint.
A compatible state is therefore specified by an $l$-tuple
\[
    \vec{\epsilon} \;=\; (\epsilon_1, \ldots, \epsilon_l) \;\in\; \{\pm 1\}^l,
\]
recording the states of the $l$ strands at the top marking of the top-left face suspension.
Write $k_+ = k_+(\vec{\epsilon})$ and $k_- = k_-(\vec{\epsilon})$ for the number of positive and negative entries of $\vec{\epsilon}$, so that $k_+ + k_- = l$.

In the computation that follows, we will use unhatted shape parameters to denote the relevant face suspension variables, setting aside their hatted versions for the actual square-root quantized shape parameters.

Fix a state $\vec\epsilon$ and consider its contribution to the top-left face suspension.
After colliding with the boundary marking, this contribution is a product of $l$ triangle factors and $l$ biangle factors.
The triangle factors are all powers of the same triangle generator and so commute among themselves; the biangle factors lie in the commutative algebra $\mathbb{B}^{\otimes 3}$ and so also commute among themselves.
Grouping the like factors together, the contribution can therefore be written as
\[
    \big(-(-A^{2})^{-\frac{1}{2}}[z''_3\,z_3]\big)^{k_+ - k_-} \big(z_3\,z''_1\big)^{k_- - k_+}.
\]
Reordering using the $A$-commutation relations, this becomes
\[
    \alpha(\vec\epsilon)\left(-(-A^{2})^{-\frac{1}{2}}[z''_3\,z_3] (z_3)^{-1}(z''_1)^{-1} \right)^{k_+ - k_-},
\]
for some $\alpha(\vec\epsilon) \in R$.
The expression inside the parentheses is precisely the contribution of a single $K_b$ strand in the top-left face suspension, so this expresses the top-left face suspension's contribution as a power of its single-strand counterpart, up to the prefactor $\alpha(\vec\epsilon)$.

Crucially, performing the analogous reordering in the top-right face suspension produces the analogous expression but with $\alpha(\vec\epsilon)^{-1}$ in place of $\alpha(\vec\epsilon)$, so the prefactors cancel between this pair.
The same cancellation occurs between the other two face suspensions.
Putting this all together, the contribution of $\vec\epsilon$ to $\Tr_\mathcal{T}(K_b^l)$ is exactly the $(k_+ - k_-)$-th power of the single-strand expression $\frac{\hat z''_1 \hat z''_2}{\hat z''_3 \hat z''_4}$ appearing in \eqref{eq:trace-Kb-single}.
Summing over $\vec\epsilon$ and applying the binomial theorem yields the following proposition:

\begin{prop}\label{prop:trace-Kb-cable}
For each integer $l \geq 0$,
\begin{equation}
    \Tr_{\mathcal{T}}(K_b^l) \;=\; \Big(-A^{-2}(\hat z''_1)^2 (\hat z''_2)^2 - A^{2}(\hat z''_1)^{-2} (\hat z''_2)^{-2} \Big)^{l}.
\end{equation}
\end{prop}

From this computation of the quantum trace of $K_b$ and its cables, it is clear that all of these skeins lie in the even part of the skein module. 

\begin{rmk}  
The skein module of an ideally triangulated $3$-manifold is naturally a module over the skein algebras of its boundary tori. 
It is natural to expect that the $3$d quantum trace is compatible with some quantum trace map on these boundary tori. 
$\Tr_\mathcal{T}(K_b^l) = \Tr_\mathcal{T}(K_b)^l$ would follow immediately from this fact. 
\end{rmk}

\subsection{Formula for the state integral}\label{sec:stintpq}

We choose the four independent equations $\{\mathcal{C}_{1},\mathcal{C}_{2},\mathfrak{M}_{1},\mathfrak{M}_{2}\}$ and the following choice of quad:
\begin{eqnarray}
(y_{1},\ldots,y_{4})=(z''_{1},z''_{2},z''_{3},z'_{4}),\qquad (y''_{1},\ldots,y''_{4})=(z_{1},z_{2},z_{3},z''_{4}).
 \end{eqnarray}
 With this choice, the matrices $A$ and $B$ and the vector $\nu$ become:
\begin{eqnarray}
A=\left[\begin{array}{cccc}
1 & 1 & 1 & 0
\\
 0 & 0 & 0 & -1
\\
 0 & 0 & 0 & -1
\\
 1 & 0 & 1 & 0
\end{array}\right],\qquad B = \left[\begin{array}{cccc}
0 & 0 & 0 & 1
\\
 1 & -1 & -1 & -1
\\
 0 & -1 & 0 & -1
\\
 1 & -1 & 0 & 0
\end{array}\right],
\end{eqnarray}
\begin{eqnarray}
\nu=(2,-1,-1,1).
\end{eqnarray}
In order to compute the matrices $C$ and $D$, we need to use the equations from the longitudes. We find, however, that in order to satisfy \eqref{condsympl}, we need to use the equations for $-\mathfrak{L}_{i}$, $i=1,2$, rather than \eqref{longitudes}. Then, $C$ and $D$ are explicitly given by:
\begin{eqnarray}
C&=&\left[\begin{array}{cccc}
c-a+b & c+\frac{1}{2}-a+b& c-a+b & -\frac{1}{2}-b+a
\\
 -a+b & -a+b& 1-a+b & a
\\
 \frac{1}{2} & \frac{1}{2} & -\frac{1}{2} & 0
\\
 -1 & -\frac{1}{2} & -1 & 0
\end{array}\right],\nonumber\\
D &=& \left[\begin{array}{cccc}
-\frac{1}{2}+b-a & 1+a-b & a-b & c
\\
 -a & a & a & b
\\
 0 & 0 & 0 & -\frac{1}{2}
\\
 \frac{1}{2} & -\frac{1}{2} & 0 & -\frac{1}{2}
\end{array}\right],\nonumber
\end{eqnarray}
where $a,b,c\in\frac{1}{2}\mathbb{Z}$ are arbitrary. The vector $\nu'$ becomes
\begin{eqnarray}
\nu'=\left(0,\frac{3}{2}\right),
\end{eqnarray}
and the flattening equations can be solved by
\begin{eqnarray}
f=(-2-a_{1},3+a_{1},a_{1},a_{2}),\qquad f''=(3+a_{1}-a_{2},a_{1}-a_{2},3+a_{1}-a_{2},1-a_{1}),
\end{eqnarray}
where $a_{1},a_{2}\in\mathbb{Z}$ are arbitrary. Then $Q(Y,X)$ in \eqref{PartFn} becomes
\begin{eqnarray}\label{Qwhite}
Q(Y,X)&=&\left(\frac{a_{1}}{2}-a_{2}\right)
\left(\pi^{2}+i\pi\hbar-\frac{\hbar^{2}}{4}\right)\nonumber\\
&+&i\pi(2(X_{1}+2X_{2})-Y_{2}+2Y_{4})-2X_{2}(Y_{1}+Y_{3})+Y_{1}(2X_{1}-Y_{2}+Y_{4})\nonumber\\
&-&\frac{Y_{2}^{2}}{2}+Y_{2}(2X_{1}+Y_{4}-Y_{3})+Y_{3}Y_{4}+\frac{3}{2}\pi^{2}+X_{2}^{2}\nonumber\\
&+&\hbar\left(\frac{Y_{2}}{2}-X_{1}-2X_{2}-Y_{4}\right)+\frac{3}{2}i\pi\hbar-\frac{3}{8}\hbar^{2}.
\end{eqnarray}
By choosing $\frac{a_{1}}{2}-a_{2}=-\frac{3}{2}$, we can get rid of the constant terms in \eqref{Qwhite}. We implement the Dehn filling along 
$-\mathfrak{q}\mathfrak{L}_{1}+\mathfrak{p}\mathfrak{M}_{1}$ by the prescription outlined in Section \ref{sec:DehnFillsi}. 
We will be interested in the asymptotics as $\hbar\rightarrow 0$. Therefore assuming\footnote{This is consistent with $\hbar\in i\mathbb{R}_{<0}$.} $\mathrm{Re}\left(\frac{X_{1}}{\mathfrak{q}}\right)<0$ and setting $X_{2}=0$, we define
\begin{equation}\label{partfnlimit}
Z_{\hbar}^{(\mathfrak{p},\mathfrak{q})}(\mathcal{T},\vec{a}):=-2ie^{-\hbar\frac{\mathfrak{s}}{4\mathfrak{q}}}\int\frac{dX_{1}}{\sqrt{2\pi \hbar \mathfrak{q}}} \prod_{i=1}^{4}\frac{dY_{i}}{\sqrt{2\pi \hbar}}e^{\frac{1}{\hbar}Q_{(\mathfrak{p},\mathfrak{q})}(Y,X_{1})}\sinh\left(\frac{X_{1}-i\pi \mathfrak{s}}{\mathfrak{q}}\right)y^{\vec{a}}x_{1}^{a_{5}}\prod_{i=1}^{4}\psi_{\hbar}(Y_{i}),
\end{equation}
where
\begin{eqnarray}\label{Qpq}
Q_{(\mathfrak{p},\mathfrak{q})}(Y,X_{1})&:=&i\pi(2X_{1}-Y_{2}+2Y_{4})+Y_{1}(2X_{1}-Y_{2}+Y_{4})-\frac{Y_{2}^{2}}{2}+Y_{2}(2X_{1}+Y_{4}-Y_{3})+Y_{3}Y_{4}\nonumber\\
&-&\hbar\left(X_{1}+Y_{4}-\frac{Y_{2}}{2}\right)+\frac{\mathfrak{p}}{\mathfrak{q}}X_{1}^{2}+\frac{2\pi i}{\mathfrak{q}}X_{1}+\frac{\mathfrak{s}\pi^{2}}{\mathfrak{q}},\nonumber\\
x_{1}&:=&\exp X_{1},\qquad y^{\vec{a}}:=\prod_{i=1}^{4}y_{i}^{a_{i}}, \qquad \vec{a}\in\mathbb{Z}^{5}.
\end{eqnarray}

%%%%%%%%%%%%%%%%%%%%%%%%%%%%%%%%%%%%%%%%%%%%%%%%

\section{The length conjecture for twist knots} \label{sec:proof}

This section is dedicated to putting it all together, culminating in a proof of Theorem~\ref{thm:main}.
The main idea is to expand both integrals \eqref{integralCZ} and \eqref{partfnlimit}, and explicitly compare their one-loop terms using a change of variables. 

\subsection{Expanding the state integral to first order}\label{sec:stateintexp}

First, we determine the critical points of the integrand of $Z^{(\mathfrak{p},\mathfrak{q})}_{\hbar}(\mathcal{T},\vec{a})$, defined in Section \ref{sec:stintpq}, as $\hbar\rightarrow 0$. Using the expansion \eqref{eq:psi expansion} of $\psi_{\hbar}(Y_{i})$, it is straightforward to show
\begin{eqnarray}\label{h0expform}
e^{\frac{1}{\hbar}Q_{(\mathfrak{p},\mathfrak{q})}(Y,X_{1})}\sinh\left(\frac{X_{1}-i\pi \mathfrak{s}}{\mathfrak{q}}\right)y^{\vec{a}}x_{1}^{a_{5}}\prod_{i=1}^{4}\psi_{\hbar}(Y_{i})=e^{\frac{1}{\hbar}\mathcal{W}_{0}(Y,X_{1})+\mathcal{O}(\hbar^{0})},\quad\text{as \ }\hbar\rightarrow 0,
\end{eqnarray}
where
\begin{eqnarray}
\mathcal{W}_{0}(Y,X_{1})&:=&i\pi(2X_{1}-Y_{2}+2Y_{4})+Y_{1}(2X_{1}-Y_{2}+Y_{4})-\frac{Y_{2}^{2}}{2}+Y_{2}(2X_{1}+Y_{4}-Y_{3})+Y_{3}Y_{4}\nonumber\\
&+&\frac{\mathfrak{p}}{\mathfrak{q}}X_{1}^{2}+\frac{2\pi i}{\mathfrak{q}}X_{1}+\frac{\mathfrak{s}\pi^{2}}{\mathfrak{q}}+\sum_{i=1}^{4}\mathrm{Li}_{2}\left(e^{-Y_{i}}\right).
\end{eqnarray}
Then, the critical points correspond to the solutions of $\exp(\frac{\partial \mathcal{W}_{0}}{\partial Y_{i}})=\exp(\frac{\partial \mathcal{W}_{0}}{\partial X_{1}})=1,$ which are equivalent to the (deformed) gluing equations:
\begin{eqnarray}\label{defglueqs}
\frac{(y_{1}-1)x_{1}^{2}y_{4}}{y_{1}y_{2}}&=&1,\nonumber\\
-\frac{(y_{2}-1)x_{1}^{2}y_{4}}{y_{2}^{2}y_{1}y_{3}}&=&1,\nonumber\\
\frac{(y_{3}-1)y_{4}}{y_{2}y_{3}}&=&1,\nonumber\\
\frac{(y_{4}-1)y_{1}y_{2}y_{3}}{y_{4}}&=&1,\nonumber\\
y_{1}^{2}y_{2}^{2}x_{1}^{\frac{2\mathfrak{p}}{\mathfrak{q}}}e^{\frac{2\pi i}{\mathfrak{q}}}&=&1.
\end{eqnarray}
These equations imply
\begin{eqnarray}
y^{2\mathfrak{q}}_{1}y_{2}^{2\mathfrak{q}}x_{1}^{2\mathfrak{p}}=e^{-\mathfrak{q}\mathfrak{L}_{1}}e^{\mathfrak{p}\mathfrak{M}_{1}}=1,
\end{eqnarray}
as expected, since they must reproduce the Dehn filling condition. 
\begin{rmk} A nondegenerate solution of \eqref{defglueqs} is equivalent to a solution of the following simplified system
\begin{align}\label{simplglueqs}
&y_{2} = y_{1}, \quad y_{3} = -y_{1}^{-1}, \quad y_{4} = \frac{y_{1}}{1+y_{1}}, \notag \\
&y_{1}^{4}\, x_{1}^{\frac{2\mathfrak{p}}{\mathfrak{q}}}\, e^{\frac{2\pi i}{\mathfrak{q}}} = 1, \quad y_{1}^{2} + (1-x_{1}^{2})\, y_{1} + x_{1}^{2} = 0.
\end{align}
\end{rmk}
From \eqref{simplglueqs}, it is evident that the critical points are finite. We then choose a critical point
\begin{eqnarray}
\mathbf{Y}^{(c)}:=(Y^{(c)}_{1},\ldots,Y^{(c)}_{4},X^{(c)})
\end{eqnarray}
and make the replacement $\mathbf{Y}\rightarrow \mathbf{Y}+\mathbf{Y}^{(c)}$ in the exponent of the $\hbar\rightarrow 0$ expansion \eqref{h0expform} and  Taylor expand around $\mathbf{Y}=0$. This will result precisely in an integrand of the form of the one used in Appendix \ref{app:Feynmann}, where the potential function \eqref{potfn} is identified with $\mathcal{W}_{0}(\mathbf{Y})$ and the vertex functions \eqref{verticesG} can be read from the expression:
\begin{eqnarray}
\sum_{k\geq 1}\frac{1}{k!}\Gamma_{i_{1},\ldots,i_{k}}(\mathbf{Y}^{(c)},\hbar)\mathbf{Y}_{i_{1}}\cdots \mathbf{Y}_{i_{k}}&=&a_{i}Y_{i}+ fB^{-1}X-\frac{1}{2}YB^{-1}\nu\nonumber\\
&+&\frac{X_{1}}{\mathfrak{q}}-\sum_{s=1}^{\infty}(-1)^{s}\frac{(2X_{1}/\mathfrak{q})^s}{s!}\mathrm{Li}_{1-s}\left(e^{\frac{2\pi i \mathfrak{s}}{\mathfrak{q}}}e^{-2\frac{X_{1}^{(c)}}{\mathfrak{q}}}\right)\nonumber\\
&+&\sum_{i=1}^{4}\sum_{s=1}^{2}(-1)^{s}\frac{Y_{i}^{s}}{s!}\sum_{n=1}^{\infty}B_{n}\frac{\hbar^{n-1}}{n!}\mathrm{Li}_{2-n-s}\left(e^{-Y_{i}^{(c)}}\right)\nonumber\\
&+&\sum_{i=1}^{4}\sum_{s=3}^{\infty}(-1)^{s}\frac{Y_{i}^{s}}{s!}\sum_{n=0}^{\infty}B_{n}\frac{\hbar^{n-1}}{n!}\mathrm{Li}_{2-n-s}\left(e^{-Y_{i}^{(c)}}\right).
\end{eqnarray}
Denote by $Z_{\hbar}^{(\mathfrak{p},\mathfrak{q})}(\mathbf{y}^{(c)},\mathcal{O}_{K_{b}^{l}})$ the integral \eqref{partfnlimit}, with the insertion of the operator
\begin{eqnarray}
\mathcal{O}_{K_{b}^{l}}:=(-1)^{l}\left(e^{-\frac{\hbar}{2}}y_{1}y_{2}+\frac{e^{\frac{\hbar}{2}}}{y_{1}y_{2}}\right)^{l} = \Tr_\mathcal{T}(K_b^l),
\end{eqnarray}
expanded around the critical point $\mathbf{y}^{(c)}=\exp(\mathbf{Y}^{(c)})$. Then, we can write \eqref{expoperator} in this case as
\begin{eqnarray}\label{expansionlengthconj}
\frac{Z_{\hbar}^{(\mathfrak{p},\mathfrak{q})}(\mathbf{y}^{(c)},\mathcal{O}_{K_{b}^{l}})}{Z_{\hbar}^{(\mathfrak{p},\mathfrak{q})}(\mathbf{y}^{(c)},\mathbf{1})}=\exp\left(\sum_{k\geq 0}\hbar^{k}W^{\mathcal{I}}_{k}(K_{b}^{l})\right).
\end{eqnarray}
In the notation of \eqref{operatorinserted}, we have that $\vec{a}\in\mathbb{Z}(1,1,0,0,0)$. In the following we use the fact that, in order to obtain the twist knot $\mathcal{K}_{p}$, we can set the Dehn filling parameters to (see Section \ref{sec:dehnFillingParams})
\begin{eqnarray}
|\mathfrak{p}|=1,\qquad p=-\mathrm{sgn}(\mathfrak{p})\mathfrak{q}, \qquad \mathfrak{q}\in\mathbb{Z}_{\geq 1}.
\end{eqnarray}

Denote $\mathbf{v}:=(1,1,0,0,0)$. Then, the expressions for $W^{\mathcal{I}}_{0}(K_{b}^{l})$ and  $W^{\mathcal{I}}_{1}(K_{b}^{l})$ in \eqref{expressW} greatly simplify to

\begin{align}\label{explicitWstate}
\exp\left(W^{\mathcal{I}}_{0}(K_{b}^{l})\right) &= (-1)^{l}\left(y^2+\frac{1}{y^{2}}\right)^{l}, \\
W^{\mathcal{I}}_{1}(K_{b}^{l}) &= \frac{1}{(y^{2}+y^{-2})^{l}}\sum_{k=0}^{l}\binom{l}{k}(2k-l)\,y^{4k-2l} \notag \\
&\qquad \times \left(-\frac{1}{2}+\widetilde{S}_{1}^{\mathbf{v}}(u)+\left(k-\frac{l}{2}-\frac{1}{2}\right)\mathbf{v}_{i}\Pi_{ij}\mathbf{v}_{j}\right),
\end{align}
where
\begin{eqnarray}
u&:=&\mathfrak{p}\mathfrak{q}=\mathrm{sgn}(\mathfrak{p})\mathfrak{q},\qquad y:=y^{(c)}_{1},
\end{eqnarray}
when evaluated at a critical point (i.e. a solution of \eqref{simplglueqs}). The notation in \eqref{explicitWstate} is as in \eqref{propagator} and \eqref{individualS}. With the use of \eqref{simplglueqs} we can arrive at the explicit form for $\widetilde{S}_{1}^{\mathbf{v}}(u)$ and $\mathbf{v}_{i}\Pi_{ij}\mathbf{v}_{j}$:
\begin{align}
\widetilde{S}_{1}^{\mathbf{v}}(u) &= \frac{1}{(y^{2}+1)(4uy^{2}+y^{2}-4u-2y-1)^{2}} \Big[ (8u^{2}+14u+3)y^{6} - 8(1+u)y^{5} \notag \\
&\qquad - 2(4u^{2}+9u+1)y^{4} - (8u^{2}-2u-3)y^{2} + 8u(u+y) + 2u \Big], \\
\mathbf{v}_{i}\Pi_{ij}\mathbf{v}_{j} &= \frac{2(y^{2}-1)}{(1+4u)y^{2}-2y-1-4u}.
\label{defsstateelts}
\end{align}

\subsection{Expanding the colored Jones polynomial}

Our starting point is the integral expression \eqref{integralCZ} with $l=0$, and with \eqref{VnCZ} truncated at $\mathcal{O}(n^{-2})$, given in \eqref{VNtruncated}. In addition we allow for the insertion of a monomial in the integration variables $t$ and $s$. For this purpose it is convenient to define:
\begin{eqnarray}
Z_{1}:=2\pi i s,\qquad Z_{2}:=2\pi i t,\qquad \hbar:=\frac{2\pi i}{n}
\end{eqnarray}
and
\begin{align}\label{CZinsertion}
Z_{\hbar}(\mathcal{K}_{p}, \vec{a}) &:= \int z^{\vec{a}}\, e^{\frac{1}{\hbar}V_{n}(Z)}\, dZ_{1}\, dZ_{2}, \\
V_{n}(Z) &:= (2p+1)\frac{Z_{1}^{2}}{2} - \pi i(2p+3) Z_{1} - 2\pi i\, Z_{2} \notag \\
&\quad + \mathrm{Li}_{2}(e^{Z_{1}+Z_{2}}) + \mathrm{Li}_{2}(e^{Z_{2}-Z_{1}}) - 3\, \mathrm{Li}_{2}(e^{Z_{2}}) + \frac{\pi^{2}}{6} \notag \\
&\quad + \hbar \left( \frac{1}{2}\left[ \mathrm{Li}_{1}(e^{Z_{1}+Z_{2}}) + \mathrm{Li}_{1}(e^{Z_{2}-Z_{1}}) + 2 Z_{2} \right] + \ln \sin(-i Z_{1}) \right), \\
z^{\vec{a}} &:= \prod_{i=1}^{2} z_{i}^{a_{i}}, \qquad z_{i} = \exp(Z_{i}), \quad \vec{a} \in \mathbb{Z}^{2}.
\end{align}
Then we have a setup where we can apply the techniques of Appendix \ref{app:Feynmann} with the potential function \eqref{potfn} identified with $\mathcal{V}(Z):=V_{n}(Z)|_{\hbar=0}$, hence the critical point equations are given by
\begin{align}\label{critCZ}
\exp\!\left(\frac{\partial \mathcal{V}(Z)}{\partial Z_{1}}\right) &= \frac{(z_{1}-z_{2})\, z_{1}^{2p}}{z_{1}z_{2}-1} = 1, \nonumber\\
\exp\!\left(\frac{\partial \mathcal{V}(Z)}{\partial Z_{2}}\right) &= \frac{(z_{2}-1)^{3}\, z_{1}}{(z_{1}z_{2}-1)(z_{1}-z_{2})} = 1.
\end{align}
Therefore, given a solution $\mathbf{Z}^{(c)}=(Z_{1}^{(c)},Z_{2}^{(c)})$ of \eqref{critCZ}, we proceed as in Section \ref{sec:stateintexp} and expand $Z_{\hbar}(\mathcal{K}_{p},\vec{a})$ around $\mathbf{Z}^{(c)}$ as $\hbar\rightarrow 0$. The vertices can be read from:
\begin{eqnarray}
\sum_{k\geq 1}\frac{1}{k!}\Gamma_{i_{1},\ldots,i_{k}}(\mathbf{Z}^{(c)},\hbar)Z_{i_{1}}\cdots Z_{i_{k}}&=&\frac{1}{2}\sum_{k\geq 1}\sum_{\varepsilon\in\{\pm1\}}\sum_{k_{1}+k_{2}=k}\frac{\varepsilon^{k_{1}}}{k_{1}!k_{2}!}\mathrm{Li}_{1-k}\left(e^{\varepsilon Z_{1}^{(c)}+Z_{2}^{(c)}}\right)Z_{1}^{k_{1}}Z_{2}^{k_{2}}\nonumber\\
&+&\frac{1}{\hbar}\sum_{k\geq 3}\sum_{\varepsilon\in\{\pm1\}}\sum_{k_{1}+k_{2}=k}\frac{\varepsilon^{k_{1}}}{k_{1}!k_{2}!}\mathrm{Li}_{2-k}\left(e^{\varepsilon Z_{1}^{(c)}+Z_{2}^{(c)}}\right)Z_{1}^{k_{1}}Z_{2}^{k_{2}}\nonumber\\
&-&\frac{3}{\hbar}\sum_{k\geq 3}\frac{1}{k!}\mathrm{Li}_{2-k}\left(e^{Z_{2}^{(c)}}\right)Z_{2}^{k}\nonumber\\
&+&\sum_{k\geq 1}\frac{1}{k!}\frac{\partial^{k}}{\partial Z_{1}^{k}}\ln\sin(-iZ^{(c)}_{1})Z_{1}^{k}+a_{i}Z_{i}+Z_{2}.
\end{eqnarray}
Denote by $Z_{\hbar}(\mathcal{K}_{p};\mathbf{z}^{(c)},K_{b}^{l})$ the integral \eqref{CZinsertion}, with the insertion of the operator
\begin{eqnarray}
(-2\cos(2\pi s))^{l}=(-1)^{l}\left(z_{1}+\frac{1}{z_{1}}\right)^{l},
\end{eqnarray}
expanded around the critical point $\mathbf{z}^{(c)}=\exp(\mathbf{Z}^{(c)})$. Then we can write \eqref{expoperator} in this case as
\begin{eqnarray} \label{expansionlengthconjJ}
\frac{Z_{\hbar}(\mathcal{K}_{p};\mathbf{z}^{(c)},K_{b}^{l})}{Z_{\hbar}(\mathcal{K}_{p};\mathbf{z}^{(c)},\mathbf{1})}=\exp\left(\sum_{k\geq 0}\hbar^{k}W^{\mathcal{J}}_{k}(K_{b}^{l})\right).
\end{eqnarray}
In the notation of \eqref{operatorinserted}, we have that $\vec{a}\in\mathbb{Z}(1,0)$. Denote $\mathbf{e}:=(1,0)$, then the expressions for $W^{\mathcal{J}}_{0}(K_{b}^{l})$ and  $W^{\mathcal{J}}_{1}(K_{b}^{l})$ in \eqref{expressW} simplify to
\begin{align}
\exp\left(W^{\mathcal{J}}_{0}(K_{b}^{l})\right) &= (-1)^{l}\left(z+\frac{1}{z}\right)^{l}, \\
W^{\mathcal{J}}_{1}(K_{b}^{l}) &= \frac{1}{(z+z^{-1})^{l}}\sum_{k=0}^{l}\binom{l}{k}(2k-l)\, z^{2k-l} \notag \\
&\qquad \times \left(\widetilde{S}_{1}^{\mathbf{e}}(z)+\left(k-\frac{l}{2}-\frac{1}{2}\right)\mathbf{e}_{i}\Pi_{ij}\mathbf{e}_{j}\right).
\end{align}
The expression for $\widetilde{S}_{1}^{\mathbf{e}}(z)$ is a bit too convoluted to record here (see the Maple file submitted with the arXiv version of this paper). On the other hand
\begin{equation}\label{eqpropagatorCZ}
\mathbf{e}_{i}\Pi_{ij}\mathbf{e}_{j} = \Pi_{1,1} = \frac{-(2z_{2}+1)(z_{1}^{2}+1)+(z_{2}^{2}+2z_{2}+3)\, z_{1}}{((4p-1)z_{2}+2p+1)(z_{1}^{2}+1)+((-2p+1)(z_{2}^{2}+2z_{2})-6p-3)\, z_{1}},
\end{equation}
where we did not use any of the critical point equations \eqref{critCZ}.

\subsection{A proof to first order}
\begin{thm}\label{thm:main_theorem}
Let $p \geq 6$, and let $W^{\mathcal{J}}_{k}(K_{b}^{l})$ and
$W^{\mathcal{I}}_{k}(K_{b}^{l})$ be the coefficients of
\eqref{expansionlengthconjJ} and \eqref{expansionlengthconj}, so that
\begin{align*}
\left.\frac{J_{n,2}\!\left(\mathcal{K}_{p}\cup K_{b}^{l};q\right)}
{J_{n}\!\left(\mathcal{K}_{p};q\right)}\right|_{q=e^{2\pi i/n}}
\;&\sim\;
\exp\!\left(\sum_{k\geq 0}\hbar^{k}W^{\mathcal{J}}_{k}(K_{b}^{l})\right),
\qquad n \to \infty,\\
\frac{Z_{\hbar}^{(\mathfrak{p},\mathfrak{q})}\!\left(\mathbf{y}^{(c)},
\mathcal{O}_{K_{b}^{l}}\right)}
{Z_{\hbar}^{(\mathfrak{p},\mathfrak{q})}\!\left(\mathbf{y}^{(c)},\mathbf{1}\right)}
\;&=\;
\exp\!\left(\sum_{k\geq 0}\hbar^{k}W^{\mathcal{I}}_{k}(K_{b}^{l})\right).
\end{align*}
Then, under the identifications
\begin{eqnarray}\label{changevars}
p=-u,\quad z_{1}=y^2,\quad z_{2}=\frac{x_{1}^{-2}+y^{2}}{x_{1}^{-2}y^{2}+1},  
\end{eqnarray}
a nondegenerate solution of the critical point equations \eqref{critCZ}
implies
\begin{eqnarray}
\exp\left(W^{\mathcal{J}}_{0}(K_{b}^{l})\right)=\exp\left(W^{\mathcal{I}}_{0}(K_{b}^{l})\right),\qquad W^{\mathcal{J}}_{1}(K_{b}^{l})=W^{\mathcal{I}}_{1}(K_{b}^{l})
\end{eqnarray}
for all $0\leq l \leq 2|p|$.
\end{thm}
\begin{proof}
Consider a nondegenerate solution $(z_{1},z_{2})$ of \eqref{critCZ}. According to Proposition 6.3 of \cite{CZ1}, this solution is unique, and more importantly, it must satisfy $-\pi<\mathrm{Im}(Z_{i})<0$, $i=1,2$. Then consider the change of variables \eqref{changevars}. This implies $-\pi< \mathrm{Im}(Y)<0$, $Y=\ln y$. Moreover, equations  \eqref{critCZ} become
\begin{eqnarray}
&&x_{1}^{2}y^{4u}=1,\\
&&(y^{2}-1)^{2}(y^{2}+x_{1}^{-2})(1+y^{2}x_{1}^{-2}+(x_{1}^{-2}-1)y)(1+y^{2}x_{1}^{-2}+(1-x_{1}^{-2})y)=0.
\label{eqsproof1}
\end{eqnarray}
For a nondegenerate solution $(z_{1},z_{2})$ we must have $y^{2}\neq 1$ and $y^{2}+x_{1}^{-2}\neq 0$, therefore \eqref{eqsproof1} implies
\begin{eqnarray}\label{eqsproof2}
x_{1}^{2}=-\frac{y(y+\varepsilon)}{1-\varepsilon y},
\end{eqnarray}
where $\varepsilon$ is either $1$ or $-1$. Under \eqref{eqsproof2}, we obtain $z_{2}=1+\varepsilon(y-y^{-1})$, so it is straightforward to see that if $-\pi<\mathrm{Im}(Z_{2})<0$ we must have that $(x_{1},y)$ solve the equation with $\varepsilon=1$. So, we have concluded that given a nondegenerate solution $(z_{1},z_{2})$ of \eqref{critCZ}, the variables $(x_{1},y)$ defined by the change of variables \eqref{changevars} must be a nondegenerate solution of \eqref{simplglueqs}. Therefore,
\begin{eqnarray}
\exp\left(W^{\mathcal{J}}_{0}(K_{b}^{l})\right)=\exp\left(W^{\mathcal{I}}_{0}(K_{b}^{l})\right)
\end{eqnarray}
under \eqref{changevars}. It is tedious but straightforward (see the Maple file attached to the arXiv version of this paper) to show that under the same change of variables

\begin{equation}
\widetilde{S}_{1}^{\mathbf{v}}(u)=\frac{1}{2}+\widetilde{S}_{1}^{\mathbf{e}}(z),\qquad \mathbf{v}_{i}\Pi_{ij}\mathbf{v}_{j}=  \mathbf{e}_{i}\Pi_{ij}\mathbf{e}_{j}, 
\end{equation}
where $\mathbf{v}_{i}\Pi_{ij}\mathbf{v}_{j}$ and $\mathbf{e}_{i}\Pi_{ij}\mathbf{e}_{j}$ are defined in equations \eqref{defsstateelts} and \eqref{eqpropagatorCZ}, respectively. This implies
\begin{equation}
 W^{\mathcal{J}}_{1}(K_{b}^{l})=W^{\mathcal{I}}_{1}(K_{b}^{l}), 
\end{equation}
completing the proof. 
\end{proof}

\begin{rmk}
Insertions of the meridian are trivial with respect to the length conjecture. 
Indeed, an addition of the meridian modifies the $n$th colored Jones polynomial of a knot by a factor of $-(q^{n/2}+q^{-n/2})$, which simply goes to $2$ when evaluated at an $n$th root of unity. 
This is compatible with the state integral case, since $e^{\mathfrak{M}_{2}}$ is identified with $x_{2}^{2}=\exp(2X_{2})$, as can be seen by keeping the variable $X_{2}$ in \eqref{Qpq}: the equations \eqref{defglueqs} get modified, and it is straightforward to solve for $x_{2}$: 
\begin{equation}
x_{2}^{2}=\frac{(y_{1}-1)y_{4}x^{2}_{1}}{y_{1}y_{2}}=-\frac{(y_{1}-1)y_{2}y_{3}}{(y_{2}-1)}=e^{\mathfrak{M}_{2}},    
\end{equation}
therefore an insertion of $K_{m}^{l}$ in $Z_{\hbar}^{(\mathfrak{p},\mathfrak{q})}(\mathcal{T})$ will be equivalent to inserting $(x^{2}_{2}+x_{2}^{-2})^{l}$, which at $X_{2}=0$ just becomes an overall factor of $2^{l}$ in $Z_{\hbar}^{(\mathfrak{p},\mathfrak{q})}(\mathcal{T})$, leaving the expansion \eqref{expansionlengthconj} unaffected. This is compatible with the length conjecture, as stated in \cite{AGLR}.
\end{rmk}

%%%%%%%%%%%%%%%%%%%%%%%%%%%%%%%%%%%%%%%%%%%%%%%%

\section{The Dehn-filled quantum trace} \label{sec:filled_quantum_trace}
We conclude this paper with an interesting conjecture suggested by our work. 

As discussed in Section \ref{sec:quantum_trace}, there is a natural $3$d quantum trace map from the stated skein module of $M$ to the square-root quantum gluing module of its ideal triangulation, denoted
\[ \Tr_{\mathcal{T}}: \mathrm{Sk}(M) \to \mathrm{SQGM}_{\mathcal{T}}(M). \]
By classical results on skein modules (see, for instance, \cite{Pr}), the skein module of a Dehn-filled manifold is obtained as a natural quotient of the skein module of the original manifold.
Specifically, there is a natural surjective homomorphism of stated skein modules, which we denote as
\[ q_S: \mathrm{Sk}(M) \to \mathrm{Sk}(M(\mathfrak{p},\mathfrak{q})). \]
The kernel of this map is generated by the relations induced by allowing skein elements to pass freely through the newly attached solid torus.

It is natural to expect that this topological Dehn filling operation has a corresponding algebraic counterpart on the square-root quantum gluing module.
That is, there should exist some quotient map $q_Q$ from $\mathrm{SQGM}_{\mathcal{T}}(M)$ to a reduced gluing module $\mathrm{SQGM}_{\mathcal{T}}(M)_{\mathfrak{p},\mathfrak{q}}$.
The proposed Dehn-filled quantum trace map is then the natural map making the resulting square commute:
\begin{conj} \label{conj:dehn_filled_trace}
Given an ideally triangulated, oriented $3$-manifold $M$ and a $(\mathfrak{p},\mathfrak{q})$ Dehn filling on a chosen boundary component $\partial_i M$, there exists a nontrivial quotient module $\mathrm{SQGM}_{\mathcal{T}}(M)_{\mathfrak{p},\mathfrak{q}}$ of the square-root quantum gluing module and a quotient map $q_Q: \mathrm{SQGM}_{\mathcal{T}}(M) \to \mathrm{SQGM}_{\mathcal{T}}(M)_{\mathfrak{p},\mathfrak{q}}$ such that the Dehn-filled quantum trace map 
\[ \Tr^{\mathfrak{p},\mathfrak{q}}_{\mathcal{T}}: \mathrm{Sk}(M(\mathfrak{p},\mathfrak{q})) \to \mathrm{SQGM}_{\mathcal{T}}(M)_{\mathfrak{p},\mathfrak{q}} \]
is well-defined and makes the following diagram commute:

\begin{equation*}
    \begin{tikzcd}
        \mathrm{Sk}(M) \arrow[r, "\Tr_{\mathcal{T}}"] \arrow[d, "q_S"'] & \mathrm{SQGM}_{\mathcal{T}}(M) \arrow[d, "q_Q"] \\
        \mathrm{Sk}(M(\mathfrak{p},\mathfrak{q})) \arrow[r, "\Tr^{\mathfrak{p},\mathfrak{q}}_{\mathcal{T}}"', dashed] & \mathrm{SQGM}_{\mathcal{T}}(M)_{\mathfrak{p},\mathfrak{q}}
    \end{tikzcd}
\end{equation*}
\end{conj}

Conjecture~\ref{conj:dehn_filled_trace} ties the present results back to the original idea of the length conjecture proposed in \cite{AGLR}.
Throughout this paper, the insertion appearing in our state integral is $\Tr_{\mathcal{T}}(K_b^{l})$, the standard $3$d quantum trace of $K_b^{l}$ viewed as a curve in the unfilled Whitehead link complement.
The original proposal of the length conjecture, on the other hand, calls for an insertion intrinsic to $S^3 \setminus \mathcal{K}_p$.
Assuming the conjecture, we have
\[
q_Q\!\left( \Tr_{\mathcal{T}}(K_b^{l}) \right) = \Tr^{1, p}_{\mathcal{T}}(K_b^{l}),
\]
expressing the (conjectural) quantum trace intrinsic to $S^3 \setminus \mathcal{K}_p$ in terms of the established $3$d quantum trace.

%%%%%%%%%%%%%%%%%%%%%%%%%%%%%%%%%%%%%%%%%%%%%%%%
\appendix 
%%%%%%%%%%%%%%%%%%%%%%%%%%%%%%%%%%%%%%
\section{Conventions on the quantum dilogarithm function}
\label{App:QDL}
%%%%%%%%%%%%%%%%%%%%%%%%%%%%%%%%%%%%%%%%%%%%%%%%%%%%%%%%%%%%%%

The quantum dilogarithm function $\psi_{\hbar}(Z)$ was introduced by Faddeev \cite{Faddeev:1993rs,Faddeev:1995nb}; we use the normalization adopted in the Chern-Simons context by \cite{Dimofte:2009yn}:
\begin{align}
\begin{split}
\psi_{\hbar}(Z) := 
\begin{cases}
\prod_{r=1}^{\infty} \frac{1 - {\sf{q}}^r e^{-Z}}{1 - \tilde{\sf{q}}^{-r+1} e^{-\tilde{Z} } }  \quad \text{if} \quad |{\sf{q}}| < 1\;,
\\
\prod_{r=1}^{\infty} \frac{1 - \tilde{\sf{q}}^r e^{-\tilde{Z}}}{1 - {\sf{q}}^{-r+1} e^{-Z } }  \quad \text{if} \quad |{\sf{q}}| > 1\;,
\end{cases}
\end{split}
\end{align}
with
\begin{align}
\begin{split}
{\sf{q}} := e^{\hbar}
,\qquad
\tilde{\sf{q}} := e^{-\frac{4\pi^2}{\hbar}}
\,,\qquad
\tilde{Z} := \frac{2\pi i Z}{\hbar }.
\end{split}
\end{align}
The function satisfies the following functional identities:
\begin{align}
\begin{split}
\psi_{\hbar}(Z+\hbar) = (1-e^{-Z}) \psi_{\hbar}(Z)
\,,\quad
\psi_{\hbar}(Z+2\pi i ) = (1-e^{-\tilde{Z}}) \psi_{\hbar}(Z)\;. \label{difference for QDL}
\end{split}
\end{align}
In the limit $\hbar \rightarrow 0$, $\psi_{\hbar}(Z)$ has the asymptotic expansion
\begin{align}
\begin{split}
\log \psi_{\hbar}(Z) \xrightarrow{\hbar\rightarrow 0}
\sum_{n=0}^{\infty} \frac{B_n \hbar^{n-1}}{n !} \text{Li}_{2-n}(e^{-Z})\;,
\label{eq:psi expansion}
\end{split}
\end{align}
where the branch cuts of the functions $\mathrm{Li}_{2}$ and $\mathrm{Li}_{1}$ appearing in \eqref{eq:psi expansion} are located along $\mathrm{Re}(Z)=0$ \cite{Dimofte:2012qj}. 

When $\mathrm{Im}(Z)\in[-\pi,0]$, the location of the branch cuts is irrelevant. $B_n$ stands for the $n$-th Bernoulli number with $B_1=1/2$. 
The Fourier transform of the quantum dilogarithm function  satisfies the following identity \cite{Faddeev:2000if,Ponsot:2000mt}:
\begin{align}
\int \frac{dZ}{\sqrt{2\pi \hbar}} e^{-\frac{Z U}\hbar}\psi_\hbar (Z) = e^{i \delta}e^{\frac{U^2-(2\pi i +\hbar)U }{2\hbar} +\frac{i \pi}{12} (\frac{\hbar}{2\pi i}+\frac{2\pi i}{\hbar})} \psi_\hbar (U) \label{Fourier of QDL},
\end{align}
with a constant overall phase factor $e^{i\delta}$.
%%%%%%%%%%%%%%%%%%%%%%%%%%%%%%%%%%%%%%%

\section{Feynman diagram expansions}\label{app:Feynmann}

Consider a parameter $\hbar$ and a Laurent-polynomial function $\mathcal{O}(y;\hbar)\in\mathbb{C}[[\hbar]][y_{1},\ldots,y_{N}]$, of the form
\begin{eqnarray}\label{operatorinserted}
\mathcal{O}(y;\hbar)=\sum_{\vec{a}}c_{\vec{a}}(\hbar)y^{\vec{a}},\qquad \vec{a}\in\mathbb{Z}^N,\quad y^{\vec{a}}:=\prod_{i=1}^{N}y_{i}^{a_{i}},
\end{eqnarray}
where $c_{\vec{a}}(\hbar)$ has a power series expansion,
\begin{eqnarray}
c_{\vec{a}}(\hbar)=\sum_{k\geq 0}\hbar^{k}c_{\vec{a},k},
\end{eqnarray}
and a potential function of $\ln(y_{i})$ (upon a choice of branch cut for the $\ln$),
\begin{eqnarray}\label{potfn}
V_{0}(Y),\qquad Y_{i}:=\ln(y_{i}).
\end{eqnarray}
Assume $V_{0}$ is independent of $\hbar$ and the critical points of $V_{0}(Y)$ are isolated, and choose one of them; call it $Y^{(c)}$. Then, define the integrals\footnote{The integrals that will appear in our formulae have an overall factor of the form $\exp\Gamma_{0}(Y^{(c)};\hbar)=\exp(\hbar^{-1}V_{0}(Y^{(c)})+O(\hbar^{0}))$, but this will not play any role on the subsequent computation in this section.}
\begin{eqnarray}
Z(y^{(c)},\mathcal{O}(y^{(c)};\hbar)):=\int \prod_{i=1}^{N} dY_{i}\mathcal{O}(y;\hbar)\exp\left(\frac{1}{2\hbar}\frac{\partial^{2}V_{0}(Y^{(c)})}{\partial Y_{i}\partial Y_{j}}Y_{i}Y_{j}+\sum_{k\geq 1}\frac{1}{k!}\Gamma_{i_{1},\ldots,i_{k}}(Y^{(c)},\hbar)Y_{i_{1}}\cdots Y_{i_{k}}\right),\nonumber
\end{eqnarray}
where the vertex functions $\Gamma_{i_{1},\ldots,i_{k}}(Y^{(c)},\hbar)$ take the form:
\begin{equation}\label{verticesG}
\begin{aligned}
\Gamma_{i}(Y^{(c)},\hbar) &= \Gamma^{(0)}_{i}(Y^{(c)})+O(\hbar)\\
\Gamma_{ij}(Y^{(c)},\hbar) &= \Gamma^{(0)}_{ij}(Y^{(c)})+O(\hbar)\\
\Gamma_{i_{1},\ldots,i_{k}}(Y^{(c)},\hbar) &= \hbar^{-1}\Gamma^{(-1)}_{i_{1},\ldots,i_{k}}(Y^{(c)})+O(\hbar^{0})
\end{aligned}
\end{equation}
We are then interested in computing the asymptotic expansion (i.e. the saddle point approximation around $Y=0$) of the ratio of integrals:
\begin{eqnarray}\label{expoperator}
\frac{Z(y^{(c)},\mathcal{O}(y^{(c)};\hbar))}{Z(y^{(c)},\mathbf{1})}=\exp\left(\sum_{k\geq 0}\hbar^{k}W_{k}(\mathcal{O}(y^{(c)}))\right),
\end{eqnarray}
where $W_{k}(\mathcal{O}(y^{(c)}))$ are functions of $y^{(c)}$, independent of $\hbar$. The functions $W_{k}(\mathcal{O}(y^{(c)}))$ are conveniently written in terms of the functions $\widetilde{S}^{\vec{a}}(y^{(c)})$, determined by the identity
\begin{eqnarray}
\frac{Z(y^{(c)},y^{\vec{a}})}{Z(y^{(c)},\mathbf{1})}=\exp\left(\sum_{k\geq 0}\hbar^{k}\widetilde{S}^{\vec{a}}_{k}(y^{(c)})\right),
\end{eqnarray}
and the propagator $\Pi_{i,j}$:
\begin{eqnarray}\label{propagator}
\Pi_{i,j}:=-\hbar(\mathcal{H}^{-1})_{i,j},\qquad \mathcal{H}_{i,j}:=\frac{\partial^{2}V_{0}(Y^{(c)})}{\partial Y_{i}\partial Y_{j}}.
\end{eqnarray}
The coefficients $\widetilde{S}^{\vec{a}}_{k}(y^{(c)})$ can be written using Feynman diagram techniques (see for instance \cite{polyak2005feynman}). We only need the terms for $k=0,1$:
\begin{eqnarray}\label{individualS}
\widetilde{S}^{\vec{a}}_{0}(y^{(c)})&=&a_{i}Y^{(c)}_{i}\nonumber\\
\widetilde{S}^{\vec{a}}_{1}(y^{(c)})&=&\mathrm{coeff}\Big\{\frac{1}{2}a_{i}\Pi_{ij}\Gamma_{jkl}\Pi_{kl}+\frac{1}{2}a_{i}\Pi_{ij}a_{j}+a_{i}\Pi_{ij}\Gamma_{j},\hbar\Big\},
\end{eqnarray}
then, explicit expressions for $W_{k}(\mathcal{O}(y^{(c)}))$, $k=0,1$ follow:
\begin{eqnarray}\label{expressW}
W_{0}(\mathcal{O}(y^{(c)}))&=&\ln\left(\sum_{\vec{a}}c_{\vec{a},0}(y^{(c)})^{\vec{a}}\right)\nonumber\\
W_{1}(\mathcal{O}(y^{(c)}))&=&\frac{\sum_{\vec{a}}(c_{\vec{a},1}+c_{\vec{a},0}\widetilde{S}^{\vec{a}}_{1}(y^{(c)}) )(y^{(c)})^{\vec{a}} }{\exp W_{0}(\mathcal{O}(y^{(c)}))}.
\end{eqnarray}
\bibliographystyle{alpha}
\bibliography{ref}

\end{document}